\documentclass[12pt,leqno]{amsart}
\usepackage{amsfonts,amsthm,amssymb,amsmath,amscd,hyperref}

\def\N{{\mathbb N}}
\def\Z{{\mathbb Z}}
\def\Qq{{\mathbb Q}}
\def\R{{\mathbb R}}
\def\C{{\mathbb C}}
\def\H{{\mathfrak H}}

\def\sq{\hbox{\rlap{$\sqcap$}$\sqcup$}}
\def\qed{\ifmmode\sq\else{\unskip\nobreak\hfil
         \penalty50\hskip1em\null\nobreak\hfil\sq
         \parfillskip=0pt\finalhyphendemerits=0\endgraf}\fi}

\def\smat#1#2#3#4{\left(\begin{smallmatrix}#1&#2\\#3&#4\end{smallmatrix}\right)}

\newtheorem{theorem}{Theorem}
\newtheorem{lemma}[theorem]{Lemma}
\newtheorem{prop}[theorem]{Proposition}
\newtheorem{cor}[theorem]{Corollary}

\numberwithin{theorem}{section}
\numberwithin{equation}{section}

\title[
Cohen-Eisenstein series of degree two and of degree three
]
{
Maass relations for generalized Cohen-Eisenstein series
of degree two and of degree three
}
\author[S.~Hayashida]{Shuichi Hayashida
}
\date{\today}

\begin{document}

\begin{abstract}
The aim of this paper is to generalize the Maass relation
for generalized Cohen-Eisenstein series of degree two and of degree three.
Here the generalized Cohen-Eisenstein series are certain Siegel modular forms of half-integral
weight, and generalized Maass relations are certain relations among Fourier-Jacobi coefficients
of them.
\end{abstract}

\maketitle

\section{Introduction}\label{s:introduction}
\subsection{}
The generalized Cohen-Eisenstein series are certain Siegel modular forms of half-integral weight,
which have been introduced by Arakawa~\cite{Ar3}.
Originally the Cohen-Eisenstein series have been introduced by Cohen~\cite{Co}
as one variable functions.
In the case of degree one, it is known that the Cohen-Eisenstein series correspond to the Eisenstein series
with respect to $\mbox{SL}(2,\Z)$ by the Shimura correspondence.
On the other hand,
in arbitrary degree we can identify the generalized Cohen-Eisenstein series with Jacobi-Eisenstein series of index $1$.
Hence, we can expect some significant properties on the generalized Cohen-Eisenstein series.

This article is devoted to show generalized Maass relations for generalized Cohen-Eisenstein series
of degree two (Theorem~\ref{thm:deg2}) and of degree three (Theorem~\ref{thm:deg3}).
Here, the generalized Maass relations are certain relations among Fourier coefficients,
and these relations are equivalent to certain relations among Fourier-Jacobi coefficients.

It is known that a certain kind of Siegel modular forms of degree two 
are obtained from elliptic modular forms through the Saito-Kurokawa lift.
Such Siegel modular forms are characterized by the Maass relation.
The generalized Maass relations for generalized Cohen-Eisenstein series of degree three (Theorem \ref{thm:deg3}) are
applied for the images of the Duke-Imamoglu-Ikeda lift.
It means that certain Siegel cusp forms of half-integral weight of degree three satisfy generalized Maass relations.
These relations
give a key of the proof for a certain lifting
which is a lifting from pairs of two elliptic modular forms
to Siegel modular forms of half-integral weight of degree two (cf.~\cite{half_deg2_lift}.)
This lifting has been conjectured in~\cite{HI}.

Because generalized Maass relations in this paper are expressed as relations among Fourier-Jacobi coefficients,
we shall explain the Fourier-Jacobi coefficients of Siegel modular forms.
Let $F$ be a Siegel modular form of degree $n$ of integral weight or half-integral weight,
which is a holomorphic function on the Siegel upper half space $\H_n$ of size $n$.
We consider an expansion:
\begin{eqnarray*}
 F\left(\begin{pmatrix} \tau_{n-m} & z \\ {^t z} & \omega_m \end{pmatrix} \right)
 &=&
 \sum_{\mathcal{S} \in Sym_m^*} \phi_{\mathcal{S}}(\tau_{n-m},z) \, e^{2\pi \sqrt{-1} \, tr(\mathcal{S} \omega_m)},
\end{eqnarray*}
where $\mbox{Sym}_m^*$ denotes the set of all half-integral symmetric matrices of size $m$,
and where $\smat{\tau_{n-m}}{z}{^t z}{\omega_m} \in \H_n$, $\tau_{n-m} \in \H_{n-m}$, $\omega_m \in \H_m$ and
$z \in M_{n-m,m}(\C)$.
Then $\phi_{\mathcal{S}}$ is a Jacobi form of index $\mathcal{S}$ of degree $n-m$.
The above expansion is called the \textit{Fourier-Jacobi expansion} of $F$.
Moreover, the above forms $\phi_{\mathcal{S}}$ are called
\textit{Fourier-Jacobi coefficients} of $F$.
The generalized Maass relations are certain relations among $\phi_{\mathcal{S}}$, and such relations
are equivalent to certain relations among Fourier coefficients of $F$.

The generalized Maass relations in this paper are identified to certain generalized Maass relations for Siegel-Eisenstein series
of \textit{integral weight} of degree three and of degree four.
Actually this identification is a key of the proof of our result.
As for generalizations of the Maass relation for Siegel-Eisenstein series of integral weight of higher degree,
Yamazaki \cite{Ya,Ya2} obtained certain relations among Fourier-Jacobi coefficients of
Siegel-Eisenstein series of arbitrary degree.
Our generalization is different from his result, because he showed relations among Fourier-Jacobi coefficients with integer index,
while our generalization in this paper concerns with the Fourier-Jacobi coefficients with
$2\times 2$ matrix index.
In his paper~\cite{Ko2} W.Kohnen gives also a generalization of the Maass relation for Siegel modular forms
of even degree $2n$. However, his result is also different from our generalization,
because his result is concerned with the Fourier-Jacobi coefficients
with $(2n-1) \times (2n-1)$ matrix index.

For our purpose we generalize some results in \cite{Ya,Ya2} on Fourier-Jacobi coefficients of Siegel-Eisenstein
series of integer indices to $2\times 2$ matrix indices.
Here the right-lower part of these $2\times 2$ matrices is $1$,
and we need to introduce certain index-shift maps on Jacobi forms of $2\times 2$ matrix indices.
For the calculation of the action of index-shift maps on Fourier-Jacobi coefficients of Siegel-Eisenstein series,
we require certain relations between Jacobi-Eisenstein series and Fourier-Jacobi coefficients of Siegel-Eisenstein series.
This relation is basically shown in \cite[Satz7]{Bo}.
We also need to show certain identity relation between Jacobi forms of integral weight of $2\times 2$ matrix index
and Jacobi forms of half-integral weight of integer index.
Moreover, we need to show a compatibility between this identity relation and index-shift maps.
Through these relations, the generalized Maass relation of
generalized Cohen-Eisenstein series are equivalent to certain relations among Jacobi-Eisenstein series
of integral weight of $2\times 2$ matrix indices. 
Finally, we calculate the action of index-shift maps on Jacobi-Eisenstein series of integral weight
of $2\times 2$ matrix index.
Because of this calculation, we have to restrict ourself to Jacobi-Eisenstein series of degree one or degree two.
It means we have to restrict ourself to generalized Cohen-Eisenstein series of degree two or degree three.

\subsection{}
We explain our results more precisely.
Let $k$ be an \textit{even} integer and $\mathcal{H}_{k-\frac12}^{(n+1)}$ be
the generalized Cohen-Eisenstein series of degree $n+1$ of weight $k-\frac12$
(see \S\ref{ss:fj_expansion_half} for the definition.)

For integer $m$,
we denote by $e_{k,m}^{(n)}$ the $m$-th Fourier-Jacobi coefficient of $\mathcal{H}_{k-\frac12}^{(n+1)}$:
\begin{eqnarray*}
 \mathcal{H}_{k-\frac12}^{(n+1)}\!\!\left(\begin{pmatrix} \tau & z \\ ^t z & \omega \end{pmatrix}\right)
&=&
 \sum_{\begin{smallmatrix} m \geq 0 \\ m \equiv 0,3 \!\!\! \mod 4  \end{smallmatrix}}
 e_{k,m}^{(n)}(\tau,z)\, e^{2\pi \sqrt{-1} m \omega},
\end{eqnarray*}
where $\tau \in \H_n$ and $\omega \in \H_1$, and where $\H_n$ denotes the Siegel upper half space
of degree $n$.
It will be shown that $e_{k,m}^{(n)}$ belongs to $J_{k-\frac12,m}^{(n)*}$ (cf. \S\ref{ss:fj_expansion_half}),
where $J_{k-\frac12,m}^{(n)*}$ denotes a certain subspace of $J_{k-\frac12,m}^{(n)}$,
and $J_{k-\frac12,m}^{(n)}$ denotes the space of all Jacobi forms of degree $n$ of weight $k-\frac12$ of index $m$
(cf. \S\ref{ss:def_jacobi_half_weight}).
Because $\mathcal{H}_{k-\frac12}^{(n+1)}$ belongs to the generalized plus-space (cf. \S\ref{ss:fj_expansion_half}),
we have $e_{k,m}^{(n)} = 0$ unless $m \equiv 0$, $3 \!\! \mod 4$.
We note that we use the symbol $e_{k,m}^{(n)}$ instead of using $e_{k-\frac12,m}^{(n)}$ for the sake of simplicity.
We remark that the weight of the form $e_{k,m}^{(n)}$ is not $k$, but $k-\frac12$.

We define a function $e_{k,m}^{(n)}|S_p^{(n)}$ of $(\tau,z) \in \H_n \times \C^{(n,1)}$ by
\begin{eqnarray*}
 (e_{k,m}^{(n)}|S_p^{(n)})(\tau,z)
 &:=&
 e_{k,mp^2}^{(n)}(\tau,z) 
  + \left(\frac{-m}{p}\right) p^{k-2}\, e_{k,m}^{(n)}(\tau,pz) + p^{2k-3}\, e_{k,\frac{m}{p^2}}^{(n)}(\tau,p^2 z).
\end{eqnarray*}
Here we regard $ e_{k,\frac{m}{p^2}}^{(n)}$ as zero,
if $\frac{m}{p^2}$ is not an integer or $\frac{m}{p^2} \not \equiv 0$, $3 \mod 4$.
Moreover, $\left(\frac{ * }{p}\right)$ denotes the Legendre symbol for odd prime $p$,
and $\left(\frac{a}{2}\right)$ $:=$ $0,1,-1$ accordingly as $a$ is even,
$a \equiv \pm 1$ mod $8$ or $a \equiv \pm 3$ mod $8$.

For any prime $p$, we introduce index-shift maps $V_p^{(1)}$, $V_{1,p}^{(2)}$ and $V_{2,p}^{(2)}$
in \S\ref{ss:hecke_V}, which are certain linear maps
 $V_p^{(1)} : J_{k-\frac12,m}^{(1)*} \rightarrow J_{k-\frac12,mp^2}^{(1)}$
and
 $V_{i,p}^{(2)} : J_{k-\frac12,m}^{(2)*} \rightarrow J_{k-\frac12,mp^2}^{(2)} \quad (i=1,2)$.
These maps are generalizations of the $V_l$-map in \cite[p. 43]{EZ}
for half-integral weight of degree $1$ and of degree $2$.

The problem now is to express $e_{k,m}^{(1)}|V_{p}^{(1)}$ and $e_{k,m}^{(2)}|V_{i,p}^{(2)}$
as linear combinations of two forms $e_{k,m}^{(n)}$ and $e_{k,m}^{(n)}|S_p^{(n)}$.
For the degree $n=1$ we obtain the following theorem.
\begin{theorem}
\label{thm:deg2}
For any prime $p$, we obtain
\begin{eqnarray*}
 e_{k,m}^{(1)}|V_{p}^{(1)}
 &=&
 e_{k,m}^{(1)}|S_p^{(1)}.
\end{eqnarray*}
\end{theorem}
The proof of this theorem will be given in \S\ref{s:thm:deg2}.

Let $\mathcal{H}_{k-\frac12}^{(2)}(Z) = \sum_{N} C(N)\, e^{2\pi \sqrt{-1} tr(NZ)}$
be the Fourier expansion of $\mathcal{H}_{k-\frac12}^{(2)}$,
where $N$ runs over all half-integral symmetric matrices of size $2$.
The identity in Theorem~\ref{thm:deg2} is translated to relations among Fourier coefficients of $\mathcal{H}_{k-\frac12}^{(2)}$
as follows.
\begin{cor}\label{cor:fourier_deg2}
For any prime $p$ and for any half-integral symmetric matrix
$\left(\begin{smallmatrix} n & \frac{r}{2} \\ \frac{r}{2} & m \end{smallmatrix}\right)$ of size $2$,
we obtain
\begin{eqnarray*}
&& C\!\begin{pmatrix} np^2 & \frac{r}{2} \\ \frac{r}{2} & m \end{pmatrix}
 +
 \left(\frac{-n}{p}\right) p^{k-2}\,
 C\!\begin{pmatrix} n & \frac{r}{2p}  \\ \frac{r}{2p}  & m \end{pmatrix}
 +
 p^{2k-3}\,
 C\!\begin{pmatrix} \frac{n}{p^2} & \frac{r}{2p^2}  \\ \frac{r}{2p^2}  & m \end{pmatrix} \\
&=&
 C\!\begin{pmatrix} n & \frac{r}{2} \\ \frac{r}{2} & mp^2 \end{pmatrix}
 +
 \left(\frac{-m}{p}\right) p^{k-2}\,
 C\!\begin{pmatrix} n & \frac{r}{2p}  \\ \frac{r}{2p}  & m \end{pmatrix}
 +
 p^{2k-3}\,
 C\!\begin{pmatrix} n & \frac{r}{2p^2}  \\ \frac{r}{2p^2}  & \frac{m}{p^2} \end{pmatrix},
\end{eqnarray*}
where we regard $C(M)$ as $0$ if the matrix $M$ is not a half-integral symmetric matrix.
\end{cor}

Because $e_{k,m}^{(1)}$ corresponds to a Fourier-Jacobi coefficient
of Siegel-Eisenstein series of weight $k$ of degree three,
we can regard the above identity also as a certain relation among Fourier coefficients
of Siegel-Eisenstein series of degree three.

As the second corollary of Theorem \ref{thm:deg2} we have the followings.
\begin{cor}\label{cor:deg2}
We obtain
\begin{eqnarray*}
 \mathcal{H}_{k-\frac12}^{(2)}\!\!\left.\left(\begin{pmatrix} \tau & 0 \\ 0 & \omega \end{pmatrix}\right)\right|_{\tau}T_1(p^2)
 &=&
 \mathcal{H}_{k-\frac12}^{(2)}\!\!\left.\left(\begin{pmatrix} \tau & 0 \\ 0 & \omega \end{pmatrix}\right)\right|_{\omega}T_1(p^2).
\end{eqnarray*}
Here, in LHS we regard
$\mathcal{H}_{k-\frac12}^{(2)}\!\!\left(\begin{pmatrix} \tau & 0 \\ 0 & \omega \end{pmatrix}\right)$
as a function of $\tau \in \H_1$,
while we regard it as a function of $\omega \in \H_1$ in RHS,
and where $T_1(p^2)$ denotes the Hecke operator acting on the Kohnen plus-space (cf. \cite[p.250]{Ko} or \S\ref{ss:cor:deg2}).
\end{cor}
The proof of Corollary \ref{cor:deg2} will be given in \S\ref{ss:cor:deg2}.

We now consider the case of degree $n = 2$.
As for the Fourier-Jacobi coefficients of the generalized Cohen-Eisenstein series $\mathcal{H}_{k-\frac12}^{(3)}$ of degree $3$,
we obtain the following theorem.
\begin{theorem}
\label{thm:deg3}
For any prime $p$, we obtain
\begin{eqnarray*}
 (e_{k,m}^{(2)}|V_{1,p}^{(2)})(\tau,z)
 &=&
 p \left( p^{2k-5} + 1 \right) e_{k,m}^{(2)}(\tau,pz)
 + (e_{k,m}^{(2)}|S_p^{(2)})(\tau,z),
\\
 (e_{k,m}^{(2)}|V_{2,p}^{(2)})(\tau,z)
 &=&
 (p^{2k-4} - p^{2k-6})\, e_{k,m}^{(2)}(\tau,pz)
 + (p^{2k-5}+1)(e_{k,m}^{(2)}|S_p^{(2)})(\tau,z).
\end{eqnarray*}
\end{theorem}
These identities can be also regarded as relations among Fourier coefficients of
$\mathcal{H}_{k-\frac12}^{(3)}$.
The expression of the Fourier coefficients of $e_{k,m}^{(2)}|V_{i,p}^{(2)}$ $(i=1,2)$
will be given in the appendix \S\ref{ss:hecke_V_12}.

Because $e_{k,m}^{(2)}$ corresponds to a Fourier-Jacobi coefficient
of Siegel-Eisenstein series of weight $k$ of degree four,
we can regard these identities as relations among Fourier coefficients
of Siegel-Eisenstein series of degree four.

Now, let $T_{2,1}(p^2)$ and $T_{2,2}(p^2)$ be Hecke operators which generate the local Hecke ring at $p$
acting on the generalized plus-space of degree two.
These $T_{2,1}(p^2)$ and $T_{2,2}(p^2)$ are denoted as
$X_1(p)$ and $p^{-k+\frac72} X_2(p)$ in \cite[p.513]{HI}, respectively.
As a corollary of Theorem \ref{thm:deg3} we have the following.

\begin{cor}\label{cor:deg3}
For any prime $p$ we obtain
\begin{eqnarray*}
 \mathcal{H}_{k-\frac12}^{(3)}\!\!\left.\left(\begin{pmatrix} \tau & 0 \\ 0 & \omega \end{pmatrix}\right)\right|_{\tau}T_{2,1}(p^2)
 &=&
 \mathcal{H}_{k-\frac12}^{(3)}\!\!\left.\left(\begin{pmatrix} \tau & 0 \\ 0 & \omega \end{pmatrix}\right)\right|_{\omega}
 (p(p^{2k-5} + 1) +  T_1(p^2)),
\end{eqnarray*}
and
\begin{eqnarray*}
 && \mathcal{H}_{k-\frac12}^{(3)}\!\!\left.\left(\begin{pmatrix} \tau & 0 \\ 0 & \omega \end{pmatrix}\right)\right|_{\tau}T_{2,2}(p^2) \\
 &=&
 \mathcal{H}_{k-\frac12}^{(3)}\!\!\left.\left(\begin{pmatrix} \tau & 0 \\ 0 & \omega \end{pmatrix}\right)\right|_{\omega}
 ((p^{2k-4}-p^{2k-6}) + p(p^{2k-5} + 1) T_1(p^2)).
\end{eqnarray*}
Here, in LHS of the above identities we regard
$\mathcal{H}_{k-\frac12}^{(3)}\!\!\left(\begin{pmatrix} \tau & 0 \\ 0 & \omega \end{pmatrix}\right)$
as a function of $\tau \in \H_2$,
while we regard it as a function of $\omega \in \H_1$ in RHS.
The proof of this corollary will be given in \S\ref{ss:cor:deg3}.
\end{cor}

\vspace{0.3cm}
\noindent
\textit{Remark 1.1}

We remark that Tanigawa~\cite[\S5]{Tani} has obtained the same identity in Corollary \ref{cor:fourier_deg2}
for \textit{Siegel-Eisenstein series of half-integral weight} of degree two with arbitrary level $N$ which satisfies $4|N$.
He showed the identity by using the formula of local densities under the assumption $p {\not|} N$.
In our case we treat the \textit{generalized Cohen-Eisenstein series} of degree two,
which is essentially level $1$.
Hence, our result contains the relation also for $p=2$.

\vspace{0.3cm}
\noindent
\textit{Remark 1.2}

Corollary \ref{cor:deg2} follows from also the pullback formula
shown by Arakawa \cite[Theorem 0.1]{Ar2}, which is a certain formula
with respect to Jacobi-Eisenstein series of index $1$.
In this paper we show Corollary \ref{cor:deg2} as the consequence of the
generalized Maass relation of generalized Cohen-Eisenstein
series of degree $3$.
\vspace{0.5cm}

This paper is organized as follows:
in Sect.~2, the necessary notation and definitions are introduced.
In Sect.~3, the relation among Fourier-Jacobi coefficients of Siegel-Eisenstein series
and Jacobi-Eisenstein series
is derived.
In Sect.~4, a certain map from a subspace of Jacobi forms of matrix index to a space of Jacobi forms of integer index
is defined.
Moreover, the compatibility of this map with certain index-shift maps is studied.
In Sect.~5, we calculate the action of certain maps on the Jacobi-Eisenstein series.
We express this function as a summation of certain exponential functions with generalized Gauss sums.
In Sect.~6, Theorem~\ref{thm:deg2} and Corollary~\ref{cor:deg2} will be proved,
while we will give the proof of Theorem~\ref{thm:deg3} and Corollary~\ref{cor:deg3} in Sect.~7.
We shall give some auxiliary calculations as an appendix in Sect.~8.

\section{Notation and definitions}

$\R^+$ : the set of all positive real numbers
 
$R^{(n,m)}$ : the set of $n\times m$ matrices with entries in a commutative ring $R$

$\mbox{Sym}_n^*$ : the set of all half-integral symmetric matrices of size $n$

$\mbox{Sym}_n^+$ : all positive definite matrices in $\mbox{Sym}_n^*$

${^t B}$ : the transpose of a matrix $B$

$A[B]$ $:=$ ${^t B} A B $ for two matrices $A \in R^{(n,n)}$ and $B \in R^{(n,m)}$

$1_n$ (resp. $0_n$) : identity matrix (resp. zero matrix) of size $n$

$\mbox{tr}(S)$ : the trace of a square matrix $S$

$e(S) := e^{2 \pi \sqrt{-1}\, \mbox{tr}(S)}$ for a square matrix $S$

$\mbox{rank}_p(x)$ : the rank of matrix $x \in \Z^{(n,m)}$ over the finite field $\Z/p\Z$

$\mbox{diag}(a_1,...,a_n)$ : the diagonal matrix $\left(\begin{smallmatrix} a_1 & & \\ &\ddots & \\ & & a_n\end{smallmatrix}\right)$
for square matrices $a_1$, ..., $a_n$

$\left(\frac{*}{p}\right)$ : the Legendre symbol for odd prime $p$

$\left(\frac{*}{2}\right)$ $:=$ $0,1,-1$ accordingly as $a$ is even,
                             $a \equiv \pm 1$ mod $8$ or $a \equiv \pm 3$ mod $8$

$M_{k-\frac12}(\Gamma_0^{(n)}(4))$ : the space of Siegel modular forms of weight $k-\frac12$ of degree $n$

$M_{k-\frac12}^{+}(\Gamma_0^{(n)}(4))$ : the plus-space of $M_{k-\frac12}(\Gamma_0^{(n)}(4))$ (cf. \cite{Ib}.)

$\mathfrak{H}_n$ : the Siegel upper half space of degree $n$

$\delta(\mathcal{S}) := 1 $ or $0$ accordingly as the statement $\mathcal{S}$ is true or false.

\subsection{Jacobi group}
For a positive integer $n$ we define the following groups:
\begin{eqnarray*}
 \mbox{GSp}_n^+(\R)
 &:=&
 \left\{
  g \in \R^{(2n,2n)} \, 
  | \, 
  g \left(\begin{smallmatrix}
     0_n & -1_n \\ 1_n & 0_n 
   \end{smallmatrix}\right)
  {^t g} 
  = n(g) 
  \left(\begin{smallmatrix}
     0_n & -1_n \\ 1_n & 0_n 
   \end{smallmatrix}\right)
  \mbox{ for some } n(g) \in \R^+
 \right\} ,
\\
 \mbox{Sp}_n(\R)
 &:=&
 \left\{
  g \in \mbox{GSp}_n^+(\R) \, | \, 
  n(g) = 1
 \right\} ,
\\
 \Gamma_n
 &:=&
 \mbox{Sp}_n(\R) \cap \Z^{(2n,2n)},
\\
 \Gamma_{\infty}^{(n)}
 &:=&
 \left.
 \left\{
  \begin{pmatrix} A & B \\ C & D \end{pmatrix} \in \Gamma_n
 \, \right| \,
  C = 0_n 
 \right\} ,\\
 \Gamma_0^{(n)}(4)
 &:=&
 \left\{
 \left.
  \begin{pmatrix}
   A & B \\ C & D
  \end{pmatrix}
  \in \Gamma_n
  \, \right| \, 
  C \equiv 0 \mod 4
 \right\} .
\end{eqnarray*}
For a matrix $g \in \mbox{GSp}_n^+(\R)$, the number $n(g)$ in the above definition of $\mbox{GSp}_n^+(\R)$ is called the \textit{similitude} of the matrix $g$.

For positive integers $n$, $r$ we define the subgroup $G_{n,r}^J \subset \mbox{GSp}_{n+r}^+(\R)$ by
\begin{eqnarray*}
 G_{n,r}^J 
 &:=&
 \left\{
  \begin{pmatrix}
   A &   & B &  \\
     & U &   &  \\
   C &   & D &  \\
     &   &   & V
\end{pmatrix}\begin{pmatrix}
   1_n &   &   & \mu  \\
   ^t \lambda  & 1_r & ^t \mu  & {^t \lambda} \mu + \kappa  \\
     & & 1_n & - \lambda  \\
     &   &   & 1_r
\end{pmatrix} 
\in \mbox{GSp}_{n+r}^+(\R)
\right\},
\end{eqnarray*}
where in the above definition
the matrices runs over
$\begin{pmatrix}
  A&B\\C&D
 \end{pmatrix}
 \in \mbox{GSp}_n^+(\R)$,
$\begin{pmatrix}
  U&0\\0&V
 \end{pmatrix}
 \in \mbox{GSp}_r^+(\R)$,
$\lambda, \mu \in \R^{(n,r)}$
and
$\kappa = {^t \kappa}\in \R^{(r,r)}$.

We will abbreviate such an element
 $\left(\begin{smallmatrix}
    A &   & B &  \\
      & U &   &  \\
    C &   & D &  \\
      &   &   & V
 \end{smallmatrix}\right)
 \left(\begin{smallmatrix}
   1_n &   &   & \mu  \\
   ^t \lambda  & 1_r & ^t \mu  & {^t \lambda}\mu + \kappa  \\
     & & 1_n & - \lambda  \\
     &   &   & 1_r
 \end{smallmatrix}\right)$
as
\[
\left(\begin{pmatrix}A&B\\C&D\end{pmatrix}
  \times
  \begin{pmatrix}U&0\\0&V\end{pmatrix},
  [
   (\lambda,\mu),\kappa
  ]
  \right) .
\]
We remark that two matrices $\left(\begin{smallmatrix}A&B\\C&D\end{smallmatrix}\right)$ and
$\left(\begin{smallmatrix}U&0\\0&V\end{smallmatrix}\right)$ have the same similitude in the above.
We will often write
\begin{eqnarray*}
\left(\left(\begin{matrix}A&B\\C&D\end{matrix}\right),
  [
   (\lambda,\mu),\kappa
  ]
  \right)
\end{eqnarray*}
instead of writing
 $\left(\left(\begin{smallmatrix}A&B\\C&D\end{smallmatrix}\right)
  \times
  1_{2r},
  [
   (\lambda,\mu),\kappa
  ]
  \right)$
for simplicity.
We remark that the element
 $\left(\left(\begin{smallmatrix}A&B\\C&D\end{smallmatrix}\right),
  [
   (\lambda,\mu),\kappa
  ]
  \right)$
belongs to $\mbox{Sp}_{n+r}(\R) $.
Similarly an element
\begin{eqnarray*}
 \left(\begin{smallmatrix}
   1_n &   &   & \mu  \\
   ^t \lambda  & 1_r & ^t \mu  & {^t \lambda}\mu + \kappa  \\
     & & 1_n & - \lambda  \\
     &   &   & 1_r
 \end{smallmatrix}\right)
 \left(\begin{smallmatrix}
    A &   & B &  \\
      & U &   &  \\
    C &   & D &  \\
      &   &   & V
 \end{smallmatrix}\right)
\end{eqnarray*}
will be abbreviated as
\begin{eqnarray*}
 \left(
  [
   (\lambda,\mu),\kappa
  ],
  \left(\begin{matrix}A&B\\C&D\end{matrix}\right)
  \times
  \left(\begin{matrix}U&0\\0&V\end{matrix}\right)
 \right),
\end{eqnarray*}
and we will abbreviate it as
 $
 \left(
  [
   (\lambda,\mu),\kappa
  ],
  \left(\begin{smallmatrix}A&B\\C&D\end{smallmatrix}\right)
 \right)
 $
for the case $U = V = 1_r$ .

We set a subgroup $\Gamma_{n,r}^J$ of $G_{n,r}^J$ by
\begin{eqnarray*}
 \Gamma_{n,r}^J 
  &:=&
 \left\{
  \left(M,
  [
   (\lambda,\mu),\kappa
  ]
  \right) 
   \in G_{n,r}^J
  \, \left| \,
  M \in \Gamma_n ,
  \lambda, \mu \in \Z^{(n,r)}, \kappa \in \Z^{(r,r)}
 \right\} \right. .
\end{eqnarray*}

\subsection{The universal covering groups $\widetilde{\mbox{GSp}_n^+(\R)}$
and $\widetilde{G_{n,1}^J}$}\label{ss:double_covering_groups}

We denote by $\widetilde{\mbox{GSp}_n^+(\R)}$ the universal covering group of $\mbox{GSp}_n^+(\R)$
which consists of pairs $(M,\varphi(\tau))$,
where $M$ is a matrix $M = \left(\begin{smallmatrix} A&B\\C&D \end{smallmatrix}\right)\in \mbox{GSp}_n^+(\R)$,
and where $\varphi$ is any holomorphic function on $\mathfrak{H}_n$
such that $|\varphi(\tau)|^2 = \det(M)^{-\frac12} |\det(C\tau + D)|$.
The group operation on $\widetilde{\mbox{GSp}_n^+(\R)}$ is given by
$(M,\varphi(\tau))(M',\varphi'(\tau)) := (M M', \varphi(M'\tau)\varphi'(\tau))$.

We embed $\Gamma_0^{(n)}(4)$ into the group $\widetilde{\mbox{GSp}_n^+(\R)}$
via $M \rightarrow (M,\theta^{(n)}(M\tau)\, \theta^{(n)}(\tau)^{-1})$, where
 $ \theta^{(n)}(\tau) 
  :=
   \displaystyle{\sum_{p \in \Z^{(n,1)}} e(\tau[p])}$
is the theta constant.
We denote by $\Gamma_0^{(n)}(4)^*$ the image of $\Gamma_0^{(n)}(4)$
in $\widetilde{\mbox{GSp}_n^+(\R)}$ by this embedding.

We define the Heisenberg group
\begin{eqnarray*}
 H_{n,1}(\R)
 &:=&
 \left\{(1_{2n},[(\lambda,\mu),\kappa]) \in \mbox{Sp}_{n+1}(\R)
 \, | \,
 \lambda,\mu \in \R^{(n,1)}, \kappa \in \R \right\} .
\end{eqnarray*}
If there is no confusion, we will write $[(\lambda,\mu),\kappa]$
for the element $(1_{2n},[(\lambda,\mu),\kappa])$ for simplicity.

We define the group
\begin{eqnarray*}
 \widetilde{G_{n,1}^J}
 &:=&
 \widetilde{\mbox{GSp}_n^+(\R)} \ltimes H_{n,1}(\R) \\
 &=&
 \left. \left\{(\tilde{M},[(\lambda,\mu),\kappa]) \, \right| \,
   \tilde{M} \in \widetilde{\mbox{GSp}_n^+(\R)},
   [(\lambda,\mu),\kappa] \in H_{n,1}(\R) \right\} .
\end{eqnarray*}
Here the group operation on $\widetilde{G_{n,1}^J}$ is given by
\begin{eqnarray*}
 (\tilde{M_1},[(\lambda_1,\mu_1),\kappa_1])\cdot (\tilde{M_2},[(\lambda_2,\mu_2),\kappa_2])
 &:=&
 (\tilde{M_1}\tilde{M_2},[(\lambda',\mu'),\kappa'])
\end{eqnarray*}
for $(\tilde{M_i},[(\lambda_i,\mu_i),\kappa_i]) \in \widetilde{G_{n,1}^J}$ $(i=1,2)$,
and where $[(\lambda',\mu'),\kappa'] \in H_{n,1}(\R)$ is the matrix determined through the identity
\begin{eqnarray*}
 &&
 (M_1\times\left(\begin{smallmatrix}n(M_1)&0\\0&1\end{smallmatrix}\right),[(\lambda_1,\mu_1),\kappa_1])
 (M_2\times\left(\begin{smallmatrix}n(M_2)&0\\0&1\end{smallmatrix}\right),[(\lambda_2,\mu_2),\kappa_2])
\\
 &=&
 (M_1M_2 \times \left(\begin{smallmatrix}n(M_1)n(M_2)&0\\0&1\end{smallmatrix}\right),[(\lambda',\mu'),\kappa'])
\end{eqnarray*}
in $G_{n,1}^J$.
Here $n(M_i)$ is the similitude of $M_i$.

\subsection{Action of the Jacobi group}
The group $G_{n,r}^J$ acts on $\mathfrak{H}_n \times \C^{(n,r)}$ by
\begin{eqnarray*}
 \gamma \cdot (\tau,z) 
 &:=&
 \left(
  \begin{pmatrix}A&B\\C&D\end{pmatrix} \cdot \tau
 \, ,\,
  ^t(C\tau+D)^{-1}(z + \tau\lambda + \mu)^t U
 \right)
\end{eqnarray*}
for any $\gamma = \left(\left(\begin{smallmatrix}A&B\\C&D\end{smallmatrix}\right)
  \times
  \left(\begin{smallmatrix}U&0\\0&V\end{smallmatrix}\right),
  [
   (\lambda,\mu),\kappa
  ]
  \right) \in G_{n,r}^J$ 
and for any $(\tau,z) \in \mathfrak{H}_n \times \C^{(n,r)}$. 
Here $\begin{pmatrix}A&B\\C&D\end{pmatrix} \cdot \tau := (A\tau+B)(C\tau+D)^{-1}$ 
is the usual transformation.

The group $\widetilde{G_{n,1}^J}$ acts on $\mathfrak{H}_n\times \C^{(n,1)}$
through the action of $G_{n,1}^J$ by 
\begin{eqnarray*}
 \tilde{\gamma}\cdot(\tau,z) 
  &:=& 
  (M\times\left(\begin{smallmatrix}n(M)&0\\0&1\end{smallmatrix}\right),[(\lambda,\mu),\kappa])\cdot(\tau,z)
\end{eqnarray*}
for $\tilde{\gamma} = ((M,\varphi),[(\lambda,\mu),\kappa]) \in \widetilde{G_{n,1}^J}$
and for $(\tau,z) \in \mathfrak{H}_n\times \C^{(n,1)}$.
Here $n(M)$ is the similitude of $M \in \mbox{GSp}_n^+(\R)$.

\subsection{Factors of automorphy}\label{ss:factors_automorphy}
Let $k$ be an integer and let $\mathcal{M} \in \mbox{Sym}_r^+$.
For
 $\gamma = \left(\left(\begin{smallmatrix}A&B\\C&D\end{smallmatrix}\right)
  \times
  \left(\begin{smallmatrix}U&0\\0&V\end{smallmatrix}\right),
  [
   (\lambda,\mu),\kappa
  ]
  \right) \in G_{n,r}^J$
we define a factor of automorphy
\begin{eqnarray*}
 J_{k,\mathcal{M}}\left( 
   \gamma, (\tau,z) \right)
 &:=& 
  \det(V)^k \det(C\tau+D)^k\, e(V^{-1}\mathcal{M}U (((C\tau+D)^{-1}C)[z + \tau \lambda + \mu]) ) \\
 && \times
  e(- V^{-1} \mathcal{M} U ({^t \lambda} \tau \lambda
     + {^t z} \lambda + {^t \lambda} z + {^t \mu} \lambda + {^t \lambda}\mu + \kappa)).
\end{eqnarray*}
We define a slash operator $|_{k,\mathcal{M}}$ by
\begin{eqnarray*}
 (\phi|_{k,\mathcal{M}}\gamma)(\tau,z) 
 &:=& 
 J_{k,\mathcal{M}}(\gamma,(\tau,z))^{-1} \phi(\gamma\cdot(\tau,z))
\end{eqnarray*}
for any function $\phi$ on $\mathfrak{H}_n \times \C^{(n,r)}$ 
and for any $\gamma \in G_{n,r}^J$.
We remark that
\begin{eqnarray*}
 J_{k,\mathcal{M}}(\gamma_1 \gamma_2, (\tau,z))
 &=&
 J_{k,\mathcal{M}}(\gamma_1, \gamma_2 \cdot (\tau,z))
 J_{k,V_1^{-1}\mathcal{M}U_1}(\gamma_2, (\tau,z)),
\\
 \phi|_{k,\mathcal{M}}\gamma_1 \gamma_2 
 &=&
 (\phi|_{k,\mathcal{M}}\gamma_1) |_{k,V_1^{-1}\mathcal{M}U_1}\gamma_2 .
\end{eqnarray*}
for any $\gamma_i = \left(M_i
  \times
  \left(\begin{smallmatrix}U_i&0\\0&V_i\end{smallmatrix}\right),
  [
   (\lambda_i,\mu_i),\kappa_i
  ]
  \right) \in G_{n,r}^J$ $(i =1,2)$.

Let $k$ and $m$ be integers.
We define a slash operator $|_{k-\frac12,m}$ for any function $\phi$ on $\mathfrak{H}_n\times \C^{(n,1)}$ by
\begin{eqnarray*}
 \phi|_{k-\frac12,m}\tilde{\gamma}
 &:=&
 J_{k-\frac12,m}(\tilde{\gamma},(\tau,z))^{-1}
 \phi(\tilde{\gamma}\cdot(\tau,z))
\end{eqnarray*}
for any $\tilde{\gamma} = ((M,\varphi),[(\lambda,\mu),\kappa]) \in \widetilde{G_{n,1}^J}$.
Here we define a factor of automorphy
\begin{eqnarray*} 
 J_{k-\frac12,m}(\tilde{\gamma},(\tau,z))
 &:=&
 \varphi(\tau)^{2k-1} e(n(M) m (((C\tau+D)^{-1}C)[z + \tau \lambda + \mu]) ) \\
 && \times
  e(-  n(M) m ({^t \lambda} \tau \lambda + {^t z} \lambda + {^t \lambda} z + {^t \mu} \lambda + {^t \lambda}\mu + \kappa)),
\end{eqnarray*}
where $n(M)$ is the similitude of $M$.
We remark that
\begin{eqnarray*}
 J_{k-\frac12,m}(\tilde{\gamma_1}\tilde{\gamma_2},(\tau,z))
 &=&
 J_{k-\frac12,m}(\tilde{\gamma_1},\tilde{\gamma_2}\cdot(\tau,z))
 J_{k-\frac12,n(M_1) m}(\tilde{\gamma_2},(\tau,z))
\\
 \phi|_{k-\frac12,m}\tilde{\gamma_1}\tilde{\gamma_2}
 &=&
 (\phi|_{k-\frac12,m}\tilde{\gamma_1})|_{k-\frac12,n(M_1)m}\tilde{\gamma_2}
\end{eqnarray*}
for any $\tilde{\gamma_i} = ((M_i,\varphi_i),[(\lambda_i,\mu_i),\kappa_i]) \in \widetilde{G_{n,1}^J}$
$(i=1,2)$.

\subsection{Jacobi forms of matrix index}\label{ss:jacobi_forms_of_matrix_index}
We quote the definition of Jacobi forms of matrix index from \cite{Zi}.
For an integer $k$ and for an matrix $\mathcal{M} \in \mbox{Sym}_r^+$ 
a $\C$-valued holomorphic function $\phi$ on $\mathfrak{H}_n \times \C^{(n,r)}$ is called
\textit{a Jacobi form of weight $k$ of
index $\mathcal{M}$ of degree $n$}, if $\phi$ satisfies the following two conditions:
\begin{enumerate}
\item
the transformation formula
$\phi|_{k,\mathcal{M}} \gamma = \phi$ for any $\gamma \in \Gamma_{n,r}^J$,
\item
$\phi$ has the Fourier expansion:
$ \phi(\tau,z)
 = \!\!\!\!\!
 \displaystyle{
   \sum_{\begin{smallmatrix}N \in Sym_n^*,R \in Z^{(n,r)} \\ 4N - R \mathcal{M}^{-1} {^t R} \geq 0 
          \end{smallmatrix}} \!\!\!\!\! c(N,R) e(N\tau) e({^t R} z)
              }$.
\end{enumerate}
We remark that the second condition follows from the Koecher principle (cf.~\cite[Lemma~1.6]{Zi})
if $n > 1$.

We denote by $J_{k,\mathcal{M}}^{(n)}$ the $\C$-vector space of Jacobi forms of weight $k$ of index $\mathcal{M}$
of degree $n$.

\subsection{Jacobi forms of half-integral weight}\label{ss:def_jacobi_half_weight}

We set the subgroup $\Gamma_{n,1}^{J*}$ of $\widetilde{G_{n,1}^J}$ by
\begin{eqnarray*}
 \Gamma_{n,1}^{J*}
 &:=&
 \left\{
  (M^*,[(\lambda,\mu),\kappa]) \in \widetilde{G_{n,1}^J}
  \, | \,
  M^* \in \Gamma_0^{(n)}(4)^*, \,
  \lambda,\mu \in \Z^{(n,1)}, \kappa \in \Z
 \right\} \\
 &\cong&
  \Gamma_0^{(n)}(4)^* \ltimes H_{n,1}(\Z),
\end{eqnarray*}
where we put $H_{n,1}(\Z) := H_{n,1}(\R) \cap \Z^{(2n+2,2n+2)}$.
Here the group $\Gamma_0^{(n)}(4)^*$ was defined in \S\ref{ss:double_covering_groups}.

For an integer $k$ and for an integer $m$,
a holomorphic function $\phi$ on $\mathfrak{H}_n \times \C^{(n,1)}$
is \textit{called a Jacobi form of weight $k-\frac12$ of index $m$},
if $\phi$ satisfies the following two conditions:
\begin{enumerate}
\item
the transformation formula
$\phi|_{k-\frac12,m} \gamma^* = \phi$ for any $\gamma^* \in \Gamma_{n,1}^{J*}$,
\item
$\phi^2|_{2k-1,2m}\gamma$ has the Fourier expansion for any $\gamma \in \Gamma_{n,1}^J$:
\begin{eqnarray*}
 \left(\phi^2|_{2k-1,2m}\gamma\right) (\tau,z)
 &=&
 \sum_{\begin{smallmatrix}
        N \in Sym_n^*,R \in \Z^{(n,1)} \\
        4Nm - R {^t R} \geq 0
       \end{smallmatrix}}
      C(N,R)\, e\!\left(\frac{1}{h}N\tau\right) e\!\left(\frac{1}{h}{^t R} z\right)
\end{eqnarray*}
with a certain integer $h > 0$,
and where the slash operator $|_{2k-1,2m}$ was defined in $\S\ref{ss:jacobi_forms_of_matrix_index}$.
\end{enumerate}

We denote by $J_{k-\frac12,m}^{(n)}$ the $\C$-vector space of Jacobi forms
of weight $k-\frac12$ of index $m$.

\subsection{Index-shift maps}\label{ss:hecke_operators}
In this subsection we introduce two kinds of maps.
The both maps shift the index of Jacobi forms and these are generalizations of
the $V_l$-map in the book of Eichler-Zagier \cite{EZ}.
 
We define two groups $\mbox{GSp}_n^+(\Z) := \mbox{GSp}_n^+(\R) \cap \Z^{(2n,2n)}$ and
\begin{eqnarray*}
 \widetilde{\mbox{GSp}_n^+(\Z)}
 &:=&
 \left.
 \left\{
  (M,\varphi) \in \widetilde{\mbox{GSp}_n^+(\R)} \, \right| \, M \in \mbox{GSp}_n^+(\Z)
 \right\}.
\end{eqnarray*}

First we define index-shift maps for Jacobi forms of \textit{integral weight of matrix index}.
Let $\mathcal{M} = \smat{*}{*}{ * }{1} \in \mbox{Sym}_2^+$.
Let $X \in \mbox{GSp}_n^+(\Z)$ be a matrix such that the similitude of $X$ is $n(X)=p^2$ with a prime $p$.
For any $\phi \in J_{k,\mathcal{M}}^{(n)}$ we define the map
\begin{eqnarray*}
 \phi|V(X)
 &:=&
 \sum_{u,v \in (\Z/p\Z)^{(n,1)}}
 \sum_{M \in \Gamma_n \backslash \Gamma_n X \Gamma_n} \!\!\!\!\!
  \phi|_{k,\mathcal{M}} 
   \left(M\times
     \left(\begin{smallmatrix}p^2&0&0&0\\0&p&0&0\\0&0&1&0\\0&0&0&p\end{smallmatrix}\right),
     [((0,u),(0,v)),0_n]\right),
\end{eqnarray*}
where $(0,u),(0,v) \in (\Z/p\Z)^{(n,2)}$.
The above summation is a finite sum and do not depend on
the choice of the representatives $u$, $v$ and $M$.
A straightforward calculation shows that
$\phi|V(X)$ belongs to $J_{k,\mathcal{M}[\left(\begin{smallmatrix}p&0\\0&1\end{smallmatrix}\right)]}^{(n)}$.
Namely $V(X)$ is a map:
\begin{eqnarray*}
 V(X) \ : \ J_{k,\mathcal{M}}^{(n)} \rightarrow J_{k,\mathcal{M}[\left(\begin{smallmatrix}p&0\\0&1\end{smallmatrix}\right)]}^{(n)}.
\end{eqnarray*}
For the sake of simplicity we set
\begin{eqnarray*}
 V_{\alpha,n-\alpha}(p^2) 
 &:=&
 V(\mbox{diag}(1_{\alpha},p 1_{n-\alpha}, p^2 1_{\alpha}, p 1_{n-\alpha}))
\end{eqnarray*}
for any prime $p$ and for any $\alpha$ $(0\leq \alpha \leq n)$.

Next we shall define index-shift maps for Jacobi forms of \textit{half-integral weight of integer index}.
We assume that $p$ is \textit{an odd prime}.
Let $m$ be a positive integer.
Let $Y = (X,\varphi) \in \widetilde{\mbox{GSp}_n^+(\Z)}$ be a matrix such that the similitude of $X$ is $n(X) = p^2$.
For $\psi \in J_{k-\frac12,m}^{(n)}$ we define the map
\begin{eqnarray*}
 \psi|\widetilde{V}(Y)
 &:=&
 n(X)^{\frac{n(2k-1)}{4} - \frac{n(n+1)}{2}}\sum_{\tilde{M} \in \Gamma_0^{(n)}(4)^* \backslash \Gamma_0^{(n)}(4)^* Y \Gamma_0^{(n)}(4)^*}
  \psi|_{k-\frac12,m} (\tilde{M},[(0,0),0_n]) ,
\end{eqnarray*}
where the above summation is a finite sum and does not depend on the choice of the representatives $\tilde{M}$.
A direct computation shows
that $\psi|\widetilde{V}(Y)$ belongs to $J_{k-\frac12,mp^2}^{(n)}$.
For the sake of simplicity we set
\begin{eqnarray*}
 \tilde{V}_{\alpha,n-\alpha}(p^2)
 &:=&
 \tilde{V}((\mbox{diag}(1_{\alpha},p 1_{n-\alpha}, p^2 1_{\alpha}, p 1_{n-\alpha}),p^{\alpha/2}))
\end{eqnarray*}
for any odd prime $p$ and for any $\alpha$ $(0\leq \alpha \leq n)$.

As for $p=2$,
we will introduce  index-shift maps $\tilde{V}_{\alpha,n-\alpha}(4)$ in \S\ref{ss:hecke_p2},
which are maps from a certain subspace $J_{k-\frac12,m}^{(n)*}$ of
$J_{k-\frac12,m}^{(n)}$ to $J_{k-\frac12,4m}^{(n)}$.

\section{Fourier Jacobi expansion of  Siegel-Eisenstein series with matrix index}\label{s:fourier_matrix}

In this section we assume that $k$ is an even integer.

For $\mathcal{M} \in \mbox{Sym}_2^+$ and for an even integer $k$ 
we define the Jacobi-Eisenstein series of weight $k$ of index $\mathcal{M}$ by
\begin{eqnarray*}
 E_{k,\mathcal{M}}^{(n)}
 &:=& 
  \sum_{M \in \Gamma_{\infty}^{(n)}\backslash \Gamma_n}\sum_{\lambda \in \Z^{(n,2)}}
  1|_{k,\mathcal{M}}([(\lambda,0),0], M) .
\end{eqnarray*}
The Jacobi-Eisenstein series $E_{k,\mathcal{M}}^{(n)}$ is absolutely convergent for $k > n+3$ (cf.~\cite{Zi}).

The Siegel-Eisenstein series $E_k^{(n)}$ of weight $k$ of degree $n$ is defined by
\begin{eqnarray*}
 E_k^{(n)}(Z)
 &:=&
 \sum_{(C,D)} \det(C Z + D)^{-k},
\end{eqnarray*}
where $Z \in \H_n$ and $(C,D)$ runs over a complete set of representatives of the equivalence classes of
coprime symmetric pairs of size $n$.
Let
\begin{eqnarray*}
 E_k^{(n)}(\left(\begin{smallmatrix}\tau&z\\ {^t z}& \omega \end{smallmatrix}\right)) 
 &=& 
\displaystyle{\sum_{\mathcal{M} \in Sym_2^*}e_{k,\mathcal{M}}^{(n-2)}(\tau,z)\, e(\mathcal{M}\omega)}
\end{eqnarray*}
be the Fourier-Jacobi expansion of the Siegel-Eisenstein series $E_k^{(n)}$ of weight $k$ of degree $n$,
where $\tau \in \H_{n-2}$, $\omega \in \H_2$ and $z \in \C^{(n-2,2)}$.

The explicit formula for the Fourier-Jacobi expansion of Siegel-Eisenstein series
is given in \cite[Satz~7]{Bo} for arbitrary degree.
The purpose of this section is to express the Fourier-Jacobi coefficient $e_{k,\mathcal{M}}^{(n-2)}$ for
$\mathcal{M} = 
  \left(\begin{smallmatrix} *&* \\ * & 1 \end{smallmatrix}\right) \in \mbox{Sym}_2^+$
as a summation of Jacobi-Eisenstein series of matrix index (Proposition~\ref{prop:fourier_jacobi}.)

First, we obtain the following lemma.
\begin{lemma}\label{lemma:eisen_A}
For any $A \in \mbox{GL}_n(\Z)$ we have
 \begin{eqnarray*}
  E_{k,\mathcal{M}}^{(n)}(\tau,z)
  &=&
  E_{k,\mathcal{M}[A^{-1}]}^{(n)}(\tau,z {^t A})
\end{eqnarray*}
and
\begin{eqnarray*}
  e_{k,\mathcal{M}}^{(n)}(\tau,z)
  &=&
  e_{k,\mathcal{M}[A^{-1}]}^{(n)}(\tau,z {^t A}) .
 \end{eqnarray*}
\end{lemma}
\begin{proof}
The first identity follows directly from the definition.
The transformation formula
 $E_k^{(n+2)}\begin{pmatrix}\begin{pmatrix}1_n & \\ & A \end{pmatrix} 
                         \begin{pmatrix}\tau&z\\ {^t z}& \omega \end{pmatrix}
                         \begin{pmatrix}1_n & \\ & {^t A} \end{pmatrix} 
           \end{pmatrix}
  =
  E_k^{(n+2)}\begin{pmatrix}
            \begin{pmatrix}\tau&z\\ {^t z}& \omega \end{pmatrix}
           \end{pmatrix}
 $
gives the second identity.
\end{proof}

Let $m$ be a positive integer. We denote by $D_0$ the discriminant of $\mathbb{Q}(\sqrt{-m})$,
and we put $f := \sqrt{\frac{m}{|D_0|}}$. We note that $f$ is a positive integer if $-m \equiv 0, 1 \!\!\mod 4$.

We denote by $h_{k-\frac12}(m)$ the $m$-th Fourier coefficient
of the Cohen-Eisenstein series of weight $k-\frac12$ (cf.~Cohen~\cite{Co}).
The following formula is known (cf.~\cite{Co},~\cite{EZ}):
\begin{eqnarray*}
 &&
 h_{k-\frac12}(m) \\
&=& 
 \begin{cases}
  h_{k-\frac12}(|D_0|)\, m^{k-\frac32} \sum_{d|f}  \mu(d) \left(\frac{D_0}{d} \right) d^{1-k} \sigma_{3-2k}\left(\frac{f}{d} \right),
  & \mbox{if } -m \equiv 0, 1 \mod 4,
 \\
  0,
  & \mbox{otherwise},
 \end{cases}
\end{eqnarray*}
where
we define
$\sigma_{a}(b) := \displaystyle{\sum_{d|b}d^a}$.

We assume $-m \equiv 0, 1 \!\! \mod 4$. Let $D_0$ and $f$ be as above.
We define
\[
g_k(m) := \sum_{d|f} \mu(d) h_{k-\frac12}\left(\frac{m}{d^2}\right).
\]

We will use the following lemma for the proof of Theorem \ref{thm:deg2} and Theorem \ref{thm:deg3}.
\begin{lemma}\label{lemma:gk}
Let $m$ be a natural number such that $-m \equiv 0$, $1 \mod 4$. 
Then for any prime $p$ we have
\begin{eqnarray*}
 g_k(p^2m) &=& 
  \left(p^{2k-3} - \left(\frac{-m}{p} \right) p^{k-2} \right) g_k(m) . 
\end{eqnarray*}
\end{lemma}
\begin{proof}
Let $D_0$, $f$ be as above.
We have
\begin{eqnarray*}
 h_{k-\frac12}(m)
 &=&
 h_{k-\frac12}(|D_0|) |D_0|^{k-\frac32}
 \prod_{q|f}
  \left\{
   \sigma_{2k-3}(q^{l_q})
   - \left(\frac{D_0}{q}\right)q^{k-2}\sigma_{2k-3}(q^{l_q-1})
  \right\} ,
\end{eqnarray*}
where $q$ runs over all primes which divide $f$,
and where we put $l_q := \mbox{ord}_q(f)$.
In particular, the function $h_{k-\frac12}(m)(h_{k-\frac12}(|D_0|) |D_0|^{k-\frac32})^{-1}$
is multiplicative with respect to $f$.
We have also
\begin{eqnarray*}
 && h_{k-\frac12}(|D_0|q^{2l_q}) - h_{k-\frac12}(|D_0|q^{2l_q-2})
\\
 &=&
 h_{k-\frac12}(|D_0|) |D_0|^{k-\frac32}
 \left(
   q^{(2k-3)l_q}
   - \left(\frac{D_0}{q}\right)q^{k-2 + (2k-3)(l_q-1)}
 \right),
\end{eqnarray*}
Thus
\begin{eqnarray*}
 g_k(m)
 &=&
 h_{k-\frac12}(|D_0|) |D_0|^{k-\frac32}
 \prod_{q|f}\frac{h_{k-\frac12}(|D_0| q^{2l_q}) - h_{k-\frac12}(|D_0| q^{2l_q-2})}{h_{k-\frac12}(|D_0|) |D_0|^{k-\frac32}}\\
 &=&
 h_{k-\frac12}(|D_0|) |D_0|^{k-\frac32}
 \prod_{q|f}\left(q^{(2k-3)l_q}
   - \left(\frac{D_0}{q}\right)q^{k-2 + (2k-3)(l_q-1)}\right).
\end{eqnarray*}
The lemma follows from this identity, because
 $\left(\frac{-m}{p}\right) =  0$ if $p|f$; 
$\left(\frac{-m}{p}\right) =\left(\frac{D_0}{p}\right)$ if $p {\not|} f$.
\end{proof}

We obtain the following proposition.

\begin{prop}\label{prop:fourier_jacobi}
For $\mathcal{M} = \begin{pmatrix} * & * \\ * & 1 \end{pmatrix} \in \mbox{Sym}_2^+$ 
we put $m = \det(2\mathcal{M})$.
Let $D_0$, $f$ be as above.
If $k > n + 1$, then
\begin{eqnarray*}
 e_{k,\mathcal{M}}^{(n-2)}(\tau,z)
 &=&
 \sum_{d|f} g_k\!\left(\frac{m}{d^2}\right) E_{k,\mathcal{M}[{^t W_d}^{-1}]}^{(n-2)}(\tau,z W_d) ,
\end{eqnarray*}
where we chose a matrix $W_d \in \mbox{GL}_2(\Qq)\cap \Z^{(2,2)}$ for each $d$ which satisfies the conditions
$\det(W_d) = d$ and $W_d^{-1}\mathcal{M} {^t W_d}^{-1} = \begin{pmatrix} * & * \\ * & 1 \end{pmatrix} \in 
\mbox{Sym}_2^+$.
The above summation is independent of the choice of the matrix $W_d$.
\end{prop}
\begin{proof}
The Satz 7 in \cite{Bo} is the essential part of this proof.
For $\mathcal{M}' \in \mbox{Sym}_n^+$ we denote by $a_2^k(\mathcal{M}')$ the $\mathcal{M}'$-th
Fourier coefficient of Siegel-Eisenstein series of weight $k$ of degree 2.
We put 
\begin{eqnarray*} 
 \mbox{M}_2^n(\Z)^*
 &:=&
 \left\{ N \in \Z^{(2,2)}
  \, | \, 
  \det(N)\neq 0 
  \mbox{ and there exists }
   V = \left(\begin{smallmatrix}N&*\\ *&*\end{smallmatrix}\right) \in \mbox{GL}_n(\Z)
 \right\} .
\end{eqnarray*}
We call a matrix $N \in \Z^{(n,2)}$ \textit{primitive} if there exists a matrix $V \in \mbox{GL}_n(\Z)$
such that $V = (N\ *)$.
From~\cite[Satz~7]{Bo} we have 
\begin{eqnarray*}
 e_{k,\mathcal{M}}^{(n-2)}(\tau,z)
 &=& 
 \sum_{\begin{smallmatrix} N_1 \in M_2^n(\Z)^*/ GL(2,\Z) \\ 
                           N_1^{-1} \mathcal{M} {^t N_1}^{-1} \in Sym_2^+\end{smallmatrix}}
 a_2^k( \mathcal{M} [{^tN_1}^{-1}] )
 \sum_{\begin{smallmatrix} N_3 \in \Z^{(n-2,2)} \\ \left(\begin{smallmatrix} N_1\\N_3 \end{smallmatrix} \right) : primitive \end{smallmatrix}}
 f(\mathcal{M},N_1,N_3;\tau,z) ,
\end{eqnarray*}
where we define
\begin{eqnarray*}
\\
 &&
 f(\mathcal{M},N_1,N_3;\tau,z)
\\
 &:=& 
 \sum_{\left(\begin{smallmatrix}A&B\\C&D \end{smallmatrix}\right) \in \Gamma_{\infty}^{(n-2)} \backslash \Gamma_{n-2}}
 \!\!\!\!\!\!\!\!
 \det(C\tau+D)^{-k} \\
 &&
  \times \, e( \mathcal{M} 
  \left\{ 
   -{^t z}(C \tau + D)^{-1}C  z + {^t z}(C \tau + D)^{-1} N_3 N_1^{-1} \right. \\
&& \left. 
   + {^t N_1}^{-1} {^t N_3} {^t(C \tau + D)^{-1}} z
   + {^t N_1}^{-1} {^t N_3} {(A \tau + B) (C \tau + D)^{-1} N_3 N_1^{-1}}
  \right\}) .
\end{eqnarray*}

For positive integer $l$ we chose a matrix $W_l \in \Z^{(2,2)}$ which satisfies three conditions
$\det(W_l) = l$,
$W_l^{-1}\mathcal{M} {^t W_l}^{-1} \in \mbox{Sym}_2^+$
and $W_l^{-1}\mathcal{M} {^t W_l}^{-1} = \begin{pmatrix} * & * \\ * & 1 \end{pmatrix}$.
Because of these conditions,
$W_l$ has the form
 $W_l = \begin{pmatrix}l&x\\0&1\end{pmatrix}$
with some $x \in \Z$.
The set $W_l \mbox{GL}(2,\Z)$ is uniquely determined for each positive integer $l$
such that $l^2|m$.

Thus
\begin{eqnarray*}
 &&
 e_{k,\mathcal{M}}^{(n-2)}(\tau,z)
\\
 &=&
 \sum_{\begin{smallmatrix} l \\ l^2 | m \end{smallmatrix}} 
  a_2^k(\mathcal{M} [{^tW_l}^{-1}])
 \sum_{a | l} \mu(a) \!\!\!
 \sum_{N_3 \in \Z^{(n-2,2)}}
   f(\mathcal{M},W_l,N_3 \left(\begin{smallmatrix}a&0\\0&1\end{smallmatrix} \right);\tau,z)
\\
 &=&
 \sum_{\begin{smallmatrix} l \\ l^2 | m \end{smallmatrix}} 
  a_2^k(\mathcal{M} [{^tW_l}^{-1}])
 \sum_{a | l} \mu(a) \!\!\!
 \sum_{N_3 \in \Z^{(n-2,2)}}
   f(\mathcal{M} [{^t W_l}^{-1} \left(\begin{smallmatrix}a&0\\0&1\end{smallmatrix} \right) ],
     1_2,N_3;
     \tau,zW_l \left(\begin{smallmatrix}a&0\\0&1\end{smallmatrix} \right)^{-1})
\end{eqnarray*}
Therefore
\begin{eqnarray*}
 &&
 e_{k,\mathcal{M}}^{(n-2)}(\tau,z) \\
 &=&
 \sum_{\begin{smallmatrix} l \\ l^2 | m \end{smallmatrix}} 
  a_2^k(\mathcal{M} [{^tW_l}^{-1}])
 \sum_{a | l} \mu(a)
 \, E_{k,\mathcal{M}[{^t W_l^{-1}}\left(\begin{smallmatrix} a& \\ & 1 \end{smallmatrix}\right)]}^{(n-2)}
    (\tau,z W_l \left(\begin{smallmatrix} a^{-1}& \\ & 1 \end{smallmatrix}\right))
\\
 &=&
 \sum_{\begin{smallmatrix} d \\ d^2 | m \end{smallmatrix}}
 E_{k,\mathcal{M}[{^t W_d^{-1}}]}^{(n-2)}
    (\tau,z W_d)
 \sum_{\begin{smallmatrix}a \\ a^2 | \frac{m}{d^2}\end{smallmatrix}}
  \mu(a)\, a_2^k(\mathcal{M} [{^tW_d}^{-1}\left(\begin{smallmatrix} a^{-1}& \\ & 1 \end{smallmatrix}\right)]).
\end{eqnarray*}
Here we have
 $a_2^k(\mathcal{M}') = h_{k-\frac12}(\det(2\mathcal{M}'))$
for any $\mathcal{M}' = \left(\begin{smallmatrix}*&*\\ *&1\end{smallmatrix}\right)\in \mbox{Sym}_2^+$.
Moreover, if
 $m \not \equiv 0,3 \mod 4$,
then
 $h_{k-\frac12}(m) = 0$.
Hence
\begin{eqnarray*}
 e_{k,\mathcal{M}}^{(n-2)}(\tau,z)
 &=&
 \sum_{\begin{smallmatrix} d \\ d | f \end{smallmatrix}}
 E_{k,\mathcal{M}[{^t W_d^{-1}}]}^{(n-2)}
    (\tau,z W_d)
 \sum_{\begin{smallmatrix}a \\ a | \frac{f}{d}\end{smallmatrix}}
  \mu(a)\, h_{k-\frac12}(\frac{m}{a^2d^2}).
\end{eqnarray*}
Therefore the proposition follows.
\end{proof}

\section{Relation between Jacobi forms of integer index and of matrix index}
In this section we fix a positive definite half-integral symmetric matrix
$\mathcal{M} \in \mbox{Sym}_2^+$,
and we assume that $\mathcal{M}$ has a form
$\mathcal{M} = \begin{pmatrix} l & \frac12 r \\ \frac12 r & 1 \end{pmatrix}$
with integers $l$ and $r$.

The purpose of this section is to give a map $\iota_{\mathcal{M}}$,
which is a map from certain holomorphic functions on $\mathfrak{H}_n \times \C^{(n,2)}$
to holomorphic functions on $\mathfrak{H}_n \times \C^{(n,1)}$.

The restriction of $\iota_{\mathcal{M}}$ gives a map from a certain subspaces
$J_{k,\, \mathcal{M}}^{(n)*}$ of $J_{k,\, \mathcal{M}}^{(n)}$
to a certain subspace
$J_{k-\frac12,\, \det(2\mathcal{M})}^{(n)*}$ of $J_{k-\frac12,\, \det(2\mathcal{M})}^{(n)}$
(cf. Lemma~\ref{lemma:iota}).
Moreover, we shall show the compatibility of the restriction of this map $\iota_{\mathcal{M}}$
with index-shift maps which shift indices of spaces of Jacobi forms.
(cf. Proposition~\ref{prop:iota_U} and Proposition~\ref{prop:iota_hecke}).
Furthermore we define index-shift maps for $J_{k-\frac12,\, \det(2\mathcal{M})}^{(n)*}$ at $p=2$
through the map $\iota_{\mathcal{M}}$ (cf. \S\ref{ss:hecke_p2}).

\subsection{An expansion of Jacobi forms of integer index}\label{ss:fj_expansion}

In this subsection we shall introduce certain spaces $M_k^*(\Gamma_n)$,
$J_{k,1}^{(n-1)*}$ and $J_{k,\mathcal{M}}^{(n-2)*}$.
Moreover, we consider an expansion of Jacobi forms of integer index.

The $\C$-vector subspace $M_k^*(\Gamma_n)$ of $M_k(\Gamma_n)$ denotes
the image of the Ikeda lifts in $M_k(\Gamma_n)$.
We remark that this subspace $M_k^*(\Gamma_n)$ contains the Siegel-Eisenstein series $E_k^{(n)}$.
In the case $n=2$, the space $M_k^*(\Gamma_2)$ coincides the Maass space. 

We denote by $J_{k,1}^{(n)}$ the space of Jacobi forms of weight $k$ of index $1$ of degree $n$
(cf. \S\ref{ss:jacobi_forms_of_matrix_index}).
The $\C$-vector subspace $J_{k,1}^{(n-1)*}$ of $J_{k,1}^{(n-1)}$ denotes the image of $M_k^*(\Gamma_n)$
through the Fourier-Jacobi expansion with index $1$.
Moreover, the $\C$-vector subspace $J_{k,\mathcal{M}}^{(n-2)*}$ of $J_{k,\mathcal{M}}^{(n-2)}$
denotes the image of $M_k^*(\Gamma_n)$ in $J_{k,\mathcal{M}}^{(n-2)}$
through the Fourier-Jacobi expansion with index $\mathcal{M}$,
where $\mathcal{M}$ is a $2 \times 2$ matrix.

Let $\phi_1(\tau,z) \in J_{k,1}^{(n-1)}$ be a Jacobi form of index 1.
We regard $\phi_1(\tau,z)\, e(\omega)$ as a holomorphic function on $\H_n$,
where $\tau \in \H_{n-1}$, $z \in \C^{(n-1,1)}$ and $\omega \in \H_1$
such that $\smat{\tau}{z}{^t z}{\omega} \in \H_n$.
We have an expansion
\begin{eqnarray*}
 \phi_1(\tau,z) e(w)
 &=&
 \sum_{\begin{smallmatrix}
        S \in Sym_2^+ \\
        S = \smat{ *}{ *}{ *}{ 1}
       \end{smallmatrix}}
      \phi_{\mathcal{S}}(\tau',z') e(\mathcal{S} \omega'),
\end{eqnarray*}
where $\tau' \in \H_{n-2}$, $z' \in \C^{(n-2,2)}$ and $\omega' \in \H_2$ such that
$\smat{\tau}{z}{^t z}{\omega} = \smat{\tau'}{z'}{^t z'}{\omega'} \in \H_n$.
Because the group $\Gamma_{n-2,2}^J$ is a subgroup of $\Gamma_{n-1,1}^J$,
the form $\phi_{\mathcal{S}}$ belongs to $J_{k,\mathcal{S}}^{(n-2)}$.
We denote this map by $\mbox{FJ}_{1,\mathcal{S}}$, namely we have a map
\begin{eqnarray*}
 \mbox{FJ}_{1,\mathcal{S}}: J_{k,1}^{(n-1)} \rightarrow J_{k,\mathcal{S}}^{(n-2)}.
\end{eqnarray*}

\subsection{Fourier-Jacobi expansion of Siegel modular forms of half-integral weight}
The purpose of this subsection is to show the following lemma.
\begin{lemma} \label{lemma:fj_half}
Let $F\smat{\tau}{z}{^t z}{\omega} = \sum_{m \in \Z}\phi_m(\tau,z) e(m\omega)$
be a Fourier-Jacobi expansion of $F \in M_{k-\frac12}(\Gamma_0^{(n+1)}(4))$,
where $\tau \in \H_n$, $\omega \in \H_1$ and $z \in \C^{(n,1)}$.
Then $\phi_m \in J_{k-\frac12,m}^{(n)}$ for any natural number $m$.
\end{lemma}
\begin{proof}
Due to the definition of $J_{k-\frac12,m}^{(n)}$, it is enough to show the identity
\begin{eqnarray*}
 \theta^{(n+1)}(\gamma \cdot \smat{\tau}{z}{^t z}{\omega})\, \theta^{(n+1)}(\smat{\tau}{z}{^t z}{\omega})^{-1}
 &=&
 \theta^{(n)}(\smat{A}{B}{C}{D} \cdot \tau)\, \theta^{(n)}(\tau)^{-1}
\end{eqnarray*}
for any $\gamma = (\smat{A}{B}{C}{D}, [(\lambda,\mu),\kappa]) \in \Gamma_{n,1}^J$, $\smat{\tau}{z}{^t z}{\omega} \in \H_{n+1}$
such that $\tau \in \H_n$, $\omega \in \H_1$.
Here $\theta^{(n+1)}$ and $\theta^{(n)}$ are the theta constants (cf. \S\ref{ss:double_covering_groups}.)

For any $M = \smat{A'}{B'}{C'}{D'} \in \Gamma_0^{(n+1)}(4)$,
it is known that 
\begin{eqnarray*}
 \left(\theta^{(n+1)}(M \cdot Z)\, \theta^{(n+1)}(Z)^{-1}\right)^2
 &=&
 \det(C' Z + D') \left(\frac{-4}{\det D'}\right),
\end{eqnarray*}
where $Z \in \H_{n+1}$.
Here $\left(\frac{-4}{\det D'}\right) \left( = (-1)^{\frac{\det D' - 1}{2}} \right)$ is the quadratic symbol.
Hence for any $\gamma = (\smat{A}{B}{C}{D}, [(\lambda,\mu),\kappa]) \in \Gamma_{n,1}^J$
we obtain
\begin{eqnarray*}
 \left(\theta^{(n+1)}(\gamma \cdot Z)\, \theta^{(n+1)}(Z)^{-1}\right)^2
 &=&
 \det(C \tau + D) \left(\frac{-4}{\det D}\right),
\end{eqnarray*}
where $Z = \smat{\tau}{z}{^t z}{\omega} \in \H_{n+1}$ with $\tau \in \H_n$.
In particular, the holomorphic function $\frac{\theta^{(n+1)}(\gamma\cdot Z)}{ \theta^{(n+1)}(Z)}$
does not depend on the choice of
$z \in \C^{(n,1)}$ and of $\omega \in \H_1$.
We substitute $z = 0$ into $\frac{\theta^{(n+1)}(\gamma\cdot Z)}{ \theta^{(n+1)}(Z)}$ and a straightforward calculation shows
\begin{eqnarray*}
 \frac{\theta^{(n+1)}(\gamma \cdot \smat{\tau}{0}{0}{\omega})}{\theta^{(n+1)}(\smat{\tau}{0}{0}{\omega})}
 &=&
 \frac{\theta^{(n)}(\smat{A}{B}{C}{D}\cdot \tau)}{\theta^{(n)}(\tau)}.
\end{eqnarray*}
Hence we conclude this lemma.
\end{proof}

\subsection{The map $\sigma$ and the subspace $J_{k-\frac12,m}^{(n)*}$}\label{ss:fj_expansion_half}
In this subsection we introduce generalized Cohen-Eisenstein series $\mathcal{H}_{k-\frac12}^{(n+1)}$
and consider the Fourier-Jacobi expansion of $\mathcal{H}_{k-\frac12}^{(n+1)}$.
Moreover, we will introduce a subspace $J_{k-\frac12,m}^{(n)*}$ of $J_{k-\frac12,m}^{(n)}$.

Let $M_{k-\frac12}^+(\Gamma_0^{(n+1)}(4))$ be the generalized plus-space introduced in~\cite[page 112]{Ib},
which is a generalization of the Kohnen plus-space for higher degrees:
\begin{eqnarray*}
 M_{k-\frac12}^+(\Gamma_0^{(n+1)}(4)) 
 &:=& \left\{ F \in M_{k-\frac12}(\Gamma_0^{(n+1)}(4))  \left| 
 \begin{matrix} \mbox{the coefficients }
 A(N) = 0 \mbox{ unless } \\
 N+ (-1)^k R {^t R} \in 4\, \mbox{Sym}_{n+1}^* \\
 \mbox{ for some } R \in \Z^{(n+1,1)}
 \end{matrix}
 \right\}
 \right.
 .
\end{eqnarray*}

For any even integer $k$,
the isomorphism between $J_{k,1}^{(n+1)}$, which is the space of Jacobi forms of index $1$,
and $M_{k-\frac12}^+(\Gamma_0^{(n+1)}(4))$ is shown in~\cite[Theorem 1]{Ib}.
We denote this linear map by $\sigma$ which is a bijection from $J_{k,1}^{(n+1)}$ to $M_{k-\frac12}^+(\Gamma_0^{(n+1)}(4))$
as modules over the ring of Hecke operators.
The map $\sigma$ is given via
\begin{eqnarray*}
 &&
 \sum_{\begin{smallmatrix} 
        N \in Sym_n^*, \, R \in \Z^{(n,1)}\\
        4N - R{^t R} \geq 0
       \end{smallmatrix}}
       C(N,R)\, e(N\tau + R{^t z}) \\
 &\mapsto& 
 \sum_{\begin{smallmatrix}
         R\!\!\! \mod  (2\Z)^{(n,1)} \\
         R \in \Z^{(n,1)}
       \end{smallmatrix}}
 \sum_{\begin{smallmatrix}
         N \in Sym_n^* \\
         4N - R{^t R} \geq 0
       \end{smallmatrix}}
    C(N,R)\, e( (4N - R {^t R}) \tau).
\end{eqnarray*}

The $\C$-vector subspace $M_{k-\frac12}^*(\Gamma_0^{(n+1)}(4))$ of $M_{k-\frac12}^+(\Gamma_0^{(n+1)}(4))$ denotes
the image of $J_{k,1}^{(n+1)*}$ by the map $\sigma$, where $J_{k,1}^{(n+1)*}$ was defined in \S\ref{ss:fj_expansion}.

Let $E_{k,1}^{(n+1)}$ be the first Fourier-Jacobi coefficient of Siegel-Eisenstein series $E_k^{(n+2)}$.
It is known that $E_{k,1}^{(n+1)}$ coincides the Jacobi-Eisenstein series of weight $k$ of index $1$ of degree $n+1$
(cf. \cite[Satz 7]{Bo}.)
We define the \textit{generalized Cohen-Eisenstein series} $\mathcal{H}_{k-\frac12}^{(n+1)}$ of weight $k-\frac12$ of degree $n+1$ by
\begin{eqnarray*}
 \mathcal{H}_{k-\frac12}^{(n+1)}
 &:=&
 \sigma(E_{k,1}^{(n+1)}).
\end{eqnarray*}
Because $E_{k,1}^{(n+1)} \in J_{k,1}^{(n+1) * }$, we have $\mathcal{H}_{k-\frac12}^{(n+1)} \in M_{k-\frac12}^*(\Gamma_0^{(n+1)}(4))$.

For any integer $m$ we denote by $\widetilde{\mbox{FJ}}_m$ the linear map from 
$M_{k-\frac12}(\Gamma_0^{(n+1)}(4))$ to $J_{k-\frac12,m}^{(n)}$
obtained by the Fourier-Jacobi expansion with respect to the index $m$.

We denote by $J_{k-\frac12,m}^{(n)*}$ the image of
$M_{k-\frac12}^*(\Gamma_0^{(n+1)}(4))$
through the map $\widetilde{\mbox{FJ}}_m$.

We recall that the form $e_{k,m}^{(n)}$ was defined as the $m$-th Fourier-Jacobi coefficient of
the generalized Cohen-Eisenstein series $\mathcal{H}_{k-\frac12}^{(n+1)}$ (cf. \S\ref{s:introduction}).
Thus $e_{k,m}^{(n)} \in J_{k-\frac12,m}^{(n)*}$.

\subsection{The map $\iota_{\mathcal{M}}$ }\label{ss:iota}
We recall
$\mathcal{M} = \left(\begin{smallmatrix} l & \frac{r}{2} \\ \frac{r}{2} & 1 \end{smallmatrix}\right) \in \mbox{Sym}_2^+$.
In this subsection we shall introduce a map
\begin{eqnarray*}
\iota_{\mathcal{M}} : H_{\mathcal{M}}^{(n)} \rightarrow \mbox{Hol}(\H_n\times \C^{(n,1)} \rightarrow \C),
\end{eqnarray*}
where $H_{\mathcal{M}}^{(n)}$ is a certain subspace of holomorphic functions on $\mathfrak{H}_n\times \C^{(n,2)}$,
which will be defined below,
and where $\mbox{Hol}(\H_n\times \C^{(n,1)} \rightarrow \C)$ denotes the space of all holomorphic functions
on $\H_n\times \C^{(n,1)}$.
We will show that
the restriction of $\iota_{\mathcal{M}}$ gives a linear isomorphism
between $J_{k,\mathcal{M}}^{(n)*}$ and  $J_{k-\frac12,m}^{(n)*}$ (cf. Lemma~\ref{lemma:iota}.)

Let $\phi$ be a holomorphic function on $\H_n \times \C^{(n,2)}$.
We assume that $\phi$ has a Fourier expansion
\begin{eqnarray*}
 \phi(\tau,z) 
 &=& 
 \sum_{\begin{smallmatrix} N \in Sym_n^* , R \in \Z^{(n,1)}\\
                          4 N - R \mathcal{M}^{-1} {^tR} \geq 0 \end{smallmatrix}}A(N,R)\, e(N\tau + ^t R z)
\end{eqnarray*}
for $(\tau,z) \in \H_n \times \C^{(n,2)}$,
and assume that 
$\phi$ satisfies the following condition on the Fourier coefficients:
if 
\begin{eqnarray*}
 \begin{pmatrix}
  N & \frac12 R \\
  \frac12 ^t R & \mathcal{M}
 \end{pmatrix}
 &=&
 \begin{pmatrix}
  N' & \frac12 R' \\
  \frac12 ^t R' & \mathcal{M}
 \end{pmatrix}
 \left[\begin{pmatrix}
  1_n &  \\
  ^tT   & 1_2
 \end{pmatrix}\right]
\end{eqnarray*}
with some $T = \begin{pmatrix} 0 , \lambda \end{pmatrix} \in \Z^{(n,2)}$,
$\lambda \in \Z^{(n,1)}$,
then $A(N,R) = A(N',R')$.

The symbol $H_{\mathcal{M}}^{(n)}$ denotes the $\C$-vector space consists of
all holomorphic functions which satisfy the above condition.

We remark $J_{k,\mathcal{M}}^{(n)*} \subset J_{k,\mathcal{M}}^{(n)} \subset H_{\mathcal{M}}^{(n)}$ for any even integer $k$.

Now we shall define a map $\iota_{\mathcal{M}}$.
For $\phi(\tau',z') = \sum A(N,R) e(N\tau' + R {^t z'}) \in H_{\mathcal{M}}^{(n)}$
we define a holomorphic function $\iota_{\mathcal{M}}(\phi)$
on $\H_n \times \C^{(n,1)}$ by
\begin{eqnarray*}
 \iota_{\mathcal{M}}(\phi)(\tau,z)
 &:=&
 \sum_{\begin{smallmatrix}
        M \in Sym_n^*,\ S \in \Z^{(n,1)} \\
        4 M m - S {^t S} \geq 0 
       \end{smallmatrix}}
       C(M,S) e(M\tau + S {^t z}),
\end{eqnarray*}
for $(\tau,z)\in \H_n \times \C^{(n,1)}$,
where we define $C(M,S) := A(N,R)$ if there exists matrices $N \in \mbox{Sym}_2^*$ and $R = (R_1,R_2) \in \Z^{(n,2)}$
$(R_1,R_2 \in \Z^{(n,1)})$
which satisfy 
\begin{eqnarray*}
 \begin{pmatrix}
  M & \frac12 S \\ \frac12 ^t S & \det(2\mathcal{M})
 \end{pmatrix}
 &=&
 4 \begin{pmatrix}
    N & \frac12 R_1 \\ \frac12 ^t R_1 & l
   \end{pmatrix}
 - \begin{pmatrix}
    R_2 \\ r 
   \end{pmatrix}
   \begin{pmatrix}
     ^t R_2 , r
    \end{pmatrix} ,
\end{eqnarray*}
$C(M,S) := 0 $ otherwise.
We remark that the coefficient $C(M,S)$ does not depend on the choice of the matrices $N$ and $R$,
because if
\begin{eqnarray*}
 4 \begin{pmatrix}
    N & \frac12 R_1 \\ \frac12 ^t R_1 & l
   \end{pmatrix}
 - \begin{pmatrix}
    R_2 \\ r 
   \end{pmatrix}
   \begin{pmatrix}
     ^t R_2 , r
    \end{pmatrix} 
&=&
 4 \begin{pmatrix}
    N' & \frac12 R'_1 \\ \frac12 ^t R'_1 & l
   \end{pmatrix}
 - \begin{pmatrix}
    R'_2 \\ r 
   \end{pmatrix}
   \begin{pmatrix}
     ^t R'_2 , r
    \end{pmatrix} ,
\end{eqnarray*}
then $4N - R_2 {^t R_2} = 4N' - {R'_2} ^t {R'_2}$. Hence $R_2 {^t R_2} \equiv  {R'_2} ^t {R'_2} \!\! \mod 4$.
Thus there exists a matrix $\lambda \in \Z^{(n,1)}$ such that ${R'_2} = R_2 + 2 \lambda$.
Therefore, by straightforward calculation we have
\begin{eqnarray*}
 \begin{pmatrix}
  N & \frac12 R \\
  \frac12 ^t R & \mathcal{M}
 \end{pmatrix}
 &=&
 \begin{pmatrix}
  N' & \frac12 R' \\
  \frac12 ^t R' & \mathcal{M}
 \end{pmatrix}
 \left[
 \begin{pmatrix}
 1_n & 0 \\
  ^t T & 1_2 
 \end{pmatrix}
 \right]
\end{eqnarray*}
with $T = (0, \lambda)$, $R = (R_1, R_2)$ and $R' = ({R'}_1, {R'}_2)$.
Because $\phi$ belongs to $H_{\mathcal{M}}^{(n)}$, the above definition of $C(M,S)$ is well-defined.

Now, we restrict the above map $\iota_{\mathcal{M}}$ to the subspace $J_{k,1}^{(n+1)*} \subset J_{k,1}^{(n+1)}$,
then we obtain the following lemma.

\begin{lemma}\label{lemma:iota}
Let $k$ be an even integer. We put $m = \det(2\mathcal{M})$.
Then we have the following commutative diagram:
$$
\begin{CD}
 J_{k,1}^{(n+1)*} @>\sigma|_{J_{k,1}^{(n+1)*}} >> M_{k-\frac12}^*(\Gamma_0^{(n+1)}(4)) \\
 @V{\mbox{FJ}_{1,\mathcal{M}}|_{J_{k,1}^{(n)*}}}VV @VV{\widetilde{\mbox{FJ}}_m|_{M_{k-\frac12}^*(\Gamma_0^{(n)}(4))}}V  \\
 J_{k,\mathcal{M}}^{(n)*} @>\iota_{\mathcal{M}}|_{J_{k,\mathcal{M}}^{(n)*}} >> J_{k-\frac12,m}^{(n)*} .
\end{CD}
$$
Moreover, the restriction of the linear map $\iota_{\mathcal{M}}$ on $J_{k,\mathcal{M}}^{(n)*}$ is
a bijection between $J_{k,\mathcal{M}}^{(n)*}$ and $J_{k-\frac12,m}^{(n)*}$.
\end{lemma}
\begin{proof}
Let $\psi \in J_{k,1}^{(n+1)*}$.
Due to the definition of $\sigma$ (cf. \S\ref{ss:fj_expansion_half}) and $\iota_{\mathcal{M}}$,
it is not difficult to see $\iota_{M}(FJ_{1,\mathcal{M}}(\psi)) = \widetilde{\mbox{FJ}}_m (\sigma(\psi))$.
Namely, we have the above commutative diagram.

Because the map
$\widetilde{\mbox{FJ}}_m|_{M_{k-\frac12}^*(\Gamma_0^{(n+1)}(4))} : M_{k-\frac12}^*(\Gamma_0^{(n+1)}(4)) \rightarrow J_{k-\frac12,m}^{(n)*}$
is surjective
and because $\sigma$ is an isomorphism,
the map
$\iota_{\mathcal{M}}|_{J_{k,\mathcal{M}}^{(n)*}} : J_{k,\mathcal{M}}^{(n)*} \rightarrow J_{k-\frac12,m}^{(n)*}$ is surjective.
The injectivity of the map
$\iota_{\mathcal{M}}|_{J_{k,\mathcal{M}}^{(n)*}} : J_{k,\mathcal{M}}^{(n)*} \rightarrow J_{k-\frac12,m}^{(n)*}$
follows directly from the definition of the map $\iota_{\mathcal{M}}$.
\end{proof}

\subsection{Compatibility between index-shift maps and $\iota_{\mathcal{M}}$}
In this subsection we shall show compatibility between some index-shift maps and the map $\iota_{\mathcal{M}}$.

For function $\psi$ on $\H_n \times \C^{(n,2)}$ and for $L \in \Z^{(2,2)}$
we define the function $\psi|U_L$ on $\H_n\times\C^{(n,2)}$ by
\begin{eqnarray*}
  (\psi|U_L)(\tau,z) &:=& \psi(\tau, z {^t L}) .
\end{eqnarray*}
For function $\phi$ on $\H_n \times \C^{(n,1)}$ and for integer $a$
we define the function $\phi|U_a$ on $\H_n \times \C^{(n,1)}$ by
\begin{eqnarray*}
   (\phi|U_a)(\tau,z) &:=& \phi(\tau,a z).
\end{eqnarray*}
\begin{prop}\label{prop:iota_U}
For any $\phi \in J_{k,\mathcal{M}}^{(n)*}$ and for any $L = \smat{a}{ }{b}{1} \in \Z^{(2,2)}$ we obtain
\begin{eqnarray*}
 \iota_{\mathcal{M}[L]}(\phi|U_L)
 &=&
 \iota_{\mathcal{M}} (\phi)|U_a.
\end{eqnarray*}
In particular, for any prime $p$ we have
$\iota_{\mathcal{M}[\smat{p}{ }{ }{1}]} \left( \phi\left| U_{\smat{p}{ }{ }{1}}\right) \right. = 
 \iota_{\mathcal{M}} (\phi)|U_p$.
\end{prop}
\begin{proof}
We put $m = \det(2 \mathcal{M})$.
Let $\phi(\tau,z') = \displaystyle{
  \!\!\!\!\!\!\!\!\!\!
  \sum_{\begin{smallmatrix}
          N \in Sym_n^*,\  R\in \Z^{(n,2)} \\
          4 N - R \mathcal{M}^{-1} {^t R} \geq 0 
  \end{smallmatrix}}
  \!\!\!\!\!\!\!\!\!\!
  A(N,R) e(N\tau + R { ^t z'})
 }$ be a Fourier expansion of $\phi$.
And let
\begin{eqnarray*}
 \iota_{\mathcal{M}}(\phi)(\tau,z) 
 &=&
 \!\!\!\!\!\!\!\!\!\!
 \sum_{\begin{smallmatrix}
            M \in Sym_n^*,\ S\in \Z^{(n,1)} 
         \\ 4 M m - S {^t S} \geq 0 
       \end{smallmatrix}}
 \!\!\!\!\!\!\!\!\!\!
 C(M,S)\, e(M\tau + S{^t z}),
\\
 \iota_{\mathcal{M}[L]}(\phi|U_L)(\tau,z)
 &=&
 \!\!\!\!\!\!\!\!\!\!
 \sum_{\begin{smallmatrix}
            M \in Sym_n^*,\ S\in \Z^{(n,1)} 
         \\ 4 M m a^2 - S {^t S} \geq 0 
       \end{smallmatrix}}
 \!\!\!\!\!\!\!\!\!\!
 C_1(M,S)\, e(M\tau + S{^t z}) 
\end{eqnarray*}
and
\begin{eqnarray*}
 (\iota_{\mathcal{M}} (\phi)|U_a)(\tau,z)
 &=&
 \!\!\!\!\!\!\!\!\!\!
 \sum_{\begin{smallmatrix}
            M \in Sym_n^*,\ S\in \Z^{(n,1)} 
         \\ 4 M m a^2 - S {^t S} \geq 0 
       \end{smallmatrix}}
 \!\!\!\!\!\!\!\!\!\!
 C_2(M,S)\, e(M\tau + S{^t z})
\end{eqnarray*}
be Fourier expansions.

We have $C_2(M,S) = C(M,a^{-1}S)$.
Moreover, we obtain
$C_1(M,S) = A(N,R L^{-1})$ with some $N \in \mbox{Sym}_n^*$ and $R \in \Z^{(n,2)}$ such that
\begin{eqnarray*}
 \begin{pmatrix}
  M & \frac12 S \\
 \frac12 {^t S} & m a^2
 \end{pmatrix}
 &=&
 4 \begin{pmatrix}
  N & \frac12 R\\
 \frac12 {^t R} & \mathcal{M}[L]
 \end{pmatrix}
 \left[
 \begin{pmatrix}
  1_n & \begin{matrix} 0 \\ \vdots \\ 0 \end{matrix}\\
  \begin{matrix} 0 \cdots 0 \end{matrix} & 1 \\
  -\frac12 {^t (R \left(\begin{smallmatrix}0\\ 1 \end{smallmatrix}\right))} & -\frac12 r a - b
 \end{pmatrix}
 \right] .
\end{eqnarray*}
With the above matrices $N$, $R$, $M$ and $S$ we have
\begin{eqnarray*}
 \begin{pmatrix}
  M & \frac{1}{2} a^{-1} S \\
 \frac{1}{2}a^{-1}{^t S} & m
 \end{pmatrix}
 &=&
 4 \begin{pmatrix}
  N & \frac12 R\\
 \frac12 {^t R} & \mathcal{M}[L]
 \end{pmatrix}
 \left[
 \begin{pmatrix}
  1_n & \begin{matrix} 0 \\ \vdots \\ 0 \end{matrix} \\
  \begin{matrix} 0 \cdots 0 \end{matrix} & 1 \\
  -\frac12 {^t (R \left(\begin{smallmatrix}0\\ 1 \end{smallmatrix}\right))} & -\frac12 r a - b
  \end{pmatrix}
 \begin{pmatrix}
  1_n & \begin{matrix} 0 \\ \vdots \\ 0 \end{matrix} \\
  \begin{matrix} 0  \cdots  0 \end{matrix} & a^{-1}
 \end{pmatrix}
 \right]
\\
 &=&
 4 \begin{pmatrix}
  N & \frac12 R  L^{-1}\\
 \frac12 {^t (R L^{-1})} & \mathcal{M}
 \end{pmatrix}
 \left[
 \begin{pmatrix}
  1_{n+1} & \begin{matrix} 0 \\ \vdots \\ 0 \end{matrix} \\
  -\frac12 {^t (R L^{-1}\left(\begin{smallmatrix}0\\ 1 \end{smallmatrix}\right))} & -\frac12 r 
  \end{pmatrix}
 \right] .
\end{eqnarray*}
Thus $C_2(M,S) = C(M,a^{-1}S) = A(N,R L^{-1}) = C_1(M,S)$.
\end{proof}

\begin{prop} \label{prop:iota_hecke}
For odd prime $p$ and for $0\leq \alpha \leq n$,
let $\tilde{V}_{\alpha,n-\alpha}(p^2)$ and $V_{\alpha,n-\alpha}(p^2)$ be index-shift maps
defined in \S\ref{ss:hecke_operators}.
Then, for any $\phi \in J_{k,\mathcal{M}}^{(n)*}$ we have
\begin{eqnarray}\label{id:iota_hecke}
 \iota_{\mathcal{M}}(\phi)| \tilde{V}_{\alpha,n-\alpha}(p^2)
 &=&
 p^{k(2n+1) - n (n+\frac72) + \frac12 \alpha }\
 \iota_{\mathcal{M}[\smat{p}{ }{ }{1}]}(\phi| V_{\alpha,n-\alpha}(p^2)).
\end{eqnarray}
\end{prop}
\begin{proof}
We compare the Fourier coefficients of the both sides of (\ref{id:iota_hecke}).
Let 
\begin{eqnarray*}
 \phi(\tau,z') 
 &=& 
 \sum_{N,R}A_1(N,R) e(N\tau + R{^t z'}),
\\
 (\phi|V_{\alpha,n-\alpha}(p^2))(\tau,z')
 &=&
 \sum_{N,R}A_2(N,R) e(N\tau + R{^t z'}),
\\
 (\iota_{\mathcal{M}}(\phi))(\tau,z)
 &=&
 \sum_{M,S}C_1(M,S) e(M\tau + S{^t z})
\end{eqnarray*}
and
\begin{eqnarray*}
 (\iota_{\mathcal{M}}(\phi)|\tilde{V}_{\alpha,n-\alpha}(p^2))(\tau,z)
 &=&
 \sum_{M,S}C_2(M,S) e(M\tau + S{^t z})
\end{eqnarray*}
be Fourier expansions,
where $\tau \in \H_n$, $z' \in \C^{(n,2)}$ and $z \in \C^{(n,1)}$.
For the sake of simplicity
we put $U = \smat{p^2}{ }{ }{p}$.
Then
\begin{eqnarray*}
 &&
 \phi| V_{\alpha,n-\alpha}(p^2) \\
 &=&
 \sum_{\smat{p^2 {^t D}^{-1}}{B}{0_n}{D}}
 \sum_{\lambda_2,\mu_2 \in (\Z/p\Z)^{(n,1)}} 
\\
 &&
 \times \phi|_{k,\mathcal{M}}
  \left(\smat{p^2 {^t D}^{-1}}{B}{0_n}{D} \times \smat{U}{ }{ }{p^2 U^{-1}}, [((0,\lambda_2),(0,\mu_2)),0_2]\right)
\\
 &=&
 \sum_{\smat{p^2 {^t D}^{-1}}{B}{0_n}{D}}
 \sum_{\lambda_2,\mu_2 \in (\Z/p\Z)^{(n,1)}} 
 \sum_{N,R}
  A(N,R)
\\
 &&
 \times e(N\tau + {R ^t z})|_{k,\mathcal{M}}
  \left(\smat{p^2 {^t D}^{-1}}{B}{0_n}{D} \times \smat{U}{ }{ }{p^2 U^{-1}}, [((0,\lambda_2),(0,\mu_2)),0_2]\right),
\end{eqnarray*}
where, in the above summations, $\smat{p^2 {^t D}^{-1}}{B}{0_n}{D}$ run over a set of all representatives
of $\Gamma_n \backslash \Gamma_n \mbox{diag}(1_{\alpha},p 1_{n-\alpha},p^2 1_{\alpha}, p 1_{n-\alpha})\Gamma_n$,
and where the slash operator $|_{k,\mathcal{M}}$ is defined in \S\ref{ss:factors_automorphy}.

We put $\lambda = (0,\lambda_2)$, $\mu = (0, \mu_2)$ $\in \Z^{(n,2)}$, then
we obtain
\begin{eqnarray*}
 &&
 e(N\tau + R{^t z})
 |_{k,\mathcal{M}} \left(\smat{p^2 {^t D}^{-1}}{B}{0_n}{D}  \times \smat{U}{ }{ }{p^2 U^{-1}},
                   [(\lambda,\mu),0_2] \right) 
\\
 &=&
 p^{-k}\det(D)^{-k} e(\hat{N} \tau + \hat{R} {^t z} + NBD^{-1} + R U {^t \mu} D^{-1}),
\end{eqnarray*}
where
\begin{eqnarray*}
 \hat{N} &=& p^2 D^{-1}N{^t D}^{-1} + D^{-1}R U {^t \lambda} +  \frac{1}{p^2}\lambda U \mathcal{M} U {^t \lambda}
\end{eqnarray*}
and
\begin{eqnarray*}
 \hat{R} &=& D^{-1} R U + \frac{2}{p^2} \lambda U \mathcal{M} U .
\end{eqnarray*}
Thus
\begin{eqnarray*}
 N
 &=& 
 \frac{1}{p^2} D\left(\left(\hat{N} - \frac{1}{4}\hat{R}_2 {^t \hat{R}_2}\right) 
   + \frac{1}{4}(\hat{R}_2 - 2 \lambda_2) {^t(\hat{R}_2 - 2 \lambda_2)}\right){^t D}
\end{eqnarray*}
and
\begin{eqnarray*}
 R &=& D\left(\hat{R} - \frac{2}{p^2} \lambda U \mathcal{M} U \right) U^{-1},
\end{eqnarray*}
where $\hat{R}_2 = \hat{R}\left(\begin{smallmatrix}0 \\ 1 \end{smallmatrix}\right)$.

Hence, for any $\hat{N} \in \mbox{Sym}_n^*$ and for any $\hat{R} \in \Z^{(n,2)}$ we have
\begin{eqnarray*}
&&
 A_2(\hat{N},\hat{R})
\\
 &=&
 p^{-k} \!\!\!\!\!\!\! \sum_{\smat{p^2 {^t D}^{-1}}{B}{0_n}{D}}
   \!\!\!\!\!\!\!\!
   \det(D)^{-k}
   \!\!\!\!\!\!\!\!
   \sum_{\lambda_2 \in (\Z/p\Z)^{(n,1)}}\sum_{\mu_2 \in (\Z/p\Z)^{(n,1)}}
   \!\!\!\!\!\!\!
    A_1(N,R)\, e(NBD^{-1} + RU{^t(0, \mu_2)}D^{-1})
\\
 &=&
 p^{-k+n} \sum_{\smat{p^2 {^t D}^{-1}}{B}{0_n}{D}} 
   \det(D)^{-k} \sum_{\lambda_2 \in (\Z/p\Z)^{(n,1)}}
   A_1(N,R)\, e(NBD^{-1}),
\end{eqnarray*}
where $N$ and $R$ are the same symbols as the above, which are determined by $\hat{N}$, $\hat{R}$ and $\lambda_2$,
and where, in the above summations, $\smat{p^2 {^t D}^{-1}}{B}{0_n}{D}$ runs over a complete set of representatives 
of $\Gamma_n \backslash \Gamma_n \mbox{diag}(1_{\alpha},p 1_{n-\alpha},p^2 1_{\alpha}, p 1_{n-\alpha}) \Gamma_n$.
We remark that $A_1(N,R) = 0$ unless $N \in \mbox{Sym}_n^*$ and $R \in \Z^{(n,2)}$.

Due to the definition of $\iota_{\mathcal{M}}$,
for $N \in \mbox{Sym}_n^*$ and $R \in \Z^{(n,2)}$ 
we have the identity 
\begin{eqnarray*}
 A_1(N,R) 
 &=&
 C_1(4N-R\left(\begin{smallmatrix}0\\1\end{smallmatrix}\right)^t(R\left(\begin{smallmatrix}0\\1\end{smallmatrix}\right)),
    4 R\left(\begin{smallmatrix}1\\0\end{smallmatrix}\right) - 2 r R\left(\begin{smallmatrix}0\\1\end{smallmatrix}\right)).
\end{eqnarray*}
Here
\begin{eqnarray*}
 4N - R\left(\begin{smallmatrix}0\\1\end{smallmatrix}\right) {^t(R\left(\begin{smallmatrix}0\\1\end{smallmatrix}\right)) }
 &=&
 \frac{1}{p^2} D \left( 4\hat{N} - \hat{R}_2 {^t \hat{R}_2} \right) {^t D}
\end{eqnarray*}
and
\begin{eqnarray*}
 4 R\left(\begin{smallmatrix}1\\0\end{smallmatrix}\right) - 2 r R\left(\begin{smallmatrix}0\\1\end{smallmatrix}\right)
 &=&
 \frac{1}{p^2}D(4 \hat{R}\left(\begin{smallmatrix}1\\0\end{smallmatrix}\right) - 2 r p \hat{R}_2).
\end{eqnarray*}
Hence we have
\begin{eqnarray}\label{id:A2}
&&
 A_2(\hat{N},\hat{R})
\\
 &=&
 p^{-k+n} \!\!\! \sum_{\smat{p^2 {^t D}^{-1}}{B}{0_n}{D}} \!\!\! \det(D)^{-k} 
 C_1(\frac{1}{p^2} D \left( 4\hat{N} - \hat{R}_2 {^t \hat{R}_2} \right) {^t D},
     \frac{1}{p^2}D(4 \hat{R}\left(\begin{smallmatrix}1\\0\end{smallmatrix}\right) - 2 r p \hat{R}_2))
\notag \\
 && \times
  e(\frac{1}{p^2} \left(\hat{N} - \frac{1}{4}\hat{R}_2 {^t \hat{R}_2}\right){^t D}B)
 \sum_{\lambda_2}
  e(\frac{1}{4 p^2} 
    (\hat{R}_2 - 2 \lambda_2) {^t(\hat{R}_2 - 2 \lambda_2)}{^t D}B), \notag
\end{eqnarray}
where $\lambda_2$ runs over a complete set of representatives of $(\Z/p\Z)^{(n,1)}$ such that
\begin{eqnarray*}
 D \left(\hat{R} - \frac{2}{p^2}(0, \lambda_2) U \mathcal{M} U\right) U^{-1} \in \Z^{(n,2)}.
\end{eqnarray*}

Let $\mathfrak{S}_{\alpha}$ be a complete set of representative of
$\Gamma_n \backslash
  \Gamma_n \left(\begin{smallmatrix}
            1_{\alpha}&&&\\
            &p 1_{n-\alpha}&&\\
            &&p^2 1_{\alpha}&\\
            &&&p 1_{n-\alpha}
           \end{smallmatrix}\right) \Gamma_n$
Now we quote a complete set of representatives $\mathfrak{S}_{\alpha}$
from \cite{Zhu:euler}.
We put
\[
 \delta_{i,j} := \mbox{diag}(1_i,p 1_{j-i},p^2 1_{n-j})
\]
and
\begin{eqnarray*}
 \mathfrak{S}_{\alpha}
 &:=&
 \left.
 \left\{
  \begin{pmatrix}p^2{\delta_{i,j}}^{-1}&b_0\\0_n & \delta_{i,j} \end{pmatrix}
  \begin{pmatrix}{^t u}^{-1} & 0_n \\ 0_n & u \end{pmatrix}
  \, \right| \,
  i,j , b_0, u
 \right\},
\end{eqnarray*}
where, in the above set, $i$ and $j$ run over all non-negative integers such that $j-i-n+\alpha \geq 0$,
and where $u$ runs over a complete set of representatives of
$(\delta_{i,j}^{-1}\mbox{GL}_n(\Z) \delta_{i,j} \cap \mbox{GL}_n(\Z)) \backslash \mbox{GL}_n(\Z)$,
and $b_0$ runs over all matrices in the set
\begin{eqnarray*}
 \mathfrak{T}
 &:=&
 \left. \left\{
  \begin{pmatrix}
   0_i & 0 & 0 \\
   0 & a_1 & p b_1 \\
   0 & ^t b_1 & b_2 
  \end{pmatrix}
  \right|
  \begin{matrix}
  b_1 \in (\Z/p\Z)^{(j-i,n-j)},
  b_2 = {^t b_2} \in (\Z/p^2 \Z)^{(n-j,n-j)},\\
  a_1 = {^t a_1} \in (\Z/p\Z)^{(j-i,j-i)},
  \mbox{rank}_p(a_1) = j-i-n+\alpha
  \end{matrix}
    \right\} .
\end{eqnarray*}

For a matrix
 $g = \begin{pmatrix}p^2 {^t D}^{-1}&B \\ 0_n& D \end{pmatrix}
  =
  \begin{pmatrix}p^2{\delta_{i,j}}^{-1}&b_0\\0_n & \delta_{i,j} \end{pmatrix}
  \begin{pmatrix}{^t u}^{-1} & 0_n \\ 0_n & u \end{pmatrix}$
  $\in$ $\mathfrak{S}_{\alpha}$ 
with a matrix
$b_0 =  \begin{pmatrix}
   0_i & 0 & 0 \\
   0 & a_1 & p b_1 \\
   0 & ^t b_1 & b_2 
  \end{pmatrix} \in \mathfrak{T}$,
we define $\varepsilon(g) := \left(\frac{-4}{p}\right)^{rank_p(a_1)/2}\left(\frac{\det a_1'}{p}\right)$,
where $a_1' \in \mbox{GL}_{j-i-n+\alpha}(\Z/p\Z)$ is a matrix,
such that
$a_1 \equiv \smat{a_1'}{0}{0}{0_{n-\alpha}}[v] \mod p$ with some $v \in \mbox{GL}_{j-i}(\Z)$.
Under the assumption
\begin{eqnarray*}
\frac{1}{p^2}D(4 \hat{R}\left(\begin{smallmatrix}1\\0\end{smallmatrix}\right) - 2 r p \hat{R}_2) \in \Z^{(n,2)}
\end{eqnarray*}
the condition $D (\hat{R} - p^{-2}(0, \lambda_2) U \mathcal{M} U) U^{-1} \in \Z^{(n,2)}$
is equivalent to the condition
\begin{eqnarray*}
 u(\hat{R}_2 - 2 \lambda_2) \in \smat{p 1_i}{0}{0}{1_{n-i}} \Z^{(n,1)}.
\end{eqnarray*}
Hence the last summation in~(\ref{id:A2}) is
\begin{eqnarray*}
 &&
 \sum_{\lambda_2}
  e(\frac{1}{4 p^2} 
    (\hat{R}_2 - 2 \lambda_2) {^t(\hat{R}_2 - 2 \lambda_2)}{^t D}B)
\\
 &=&
 p^{n-j} \sum_{\lambda' \in (\Z/p\Z)^{(j-i,1)}} e\left(\frac{1}{p} {^t \lambda'} a_1 \lambda'\right)
\\
 &=&
 p^{n - i - rank_p(a_1)} \left(\left(\frac{-4}{p}\right) p\right)^{rank_p(a_1)/2} \left(\frac{\det a_1'}{p}\right)
\\
 &=&
 p^{n - i -\frac{rank_p(a_1)}{2}} \varepsilon(g)
\\
 &=&
 p^{n + (n-i-j-\alpha)/2} \varepsilon(g).
\end{eqnarray*}
Thus (\ref{id:A2}) is
\begin{eqnarray*}
 A_2(\hat{N},\hat{R})
 &=&
 p^{-k+2n} \sum_{g } p^{-k(2n-i-j) + (n-i-j-\alpha)/2} \varepsilon(g)\,
 e\! \left(p^{-2} \left(4 \hat{N} - \hat{R}_2 {^t \hat{R}_2}\right){^t D}B\right)
\\
 && \times
  C_1\left(p^{-2} D ( 4\hat{N} - \hat{R}_2 {^t \hat{R}_2} ) {^t D},\
     p^{-2}D(4 \hat{R}\left(\begin{smallmatrix}1\\0\end{smallmatrix}\right) - 2 r p \hat{R}_2)\right),
\end{eqnarray*}
where 
 $g = \begin{pmatrix}p^2 {^t D}^{-1}&B \\ 0_n& D \end{pmatrix} =
  \begin{pmatrix}p^2{\delta_{i,j}}^{-1}&b_0\\0_n & \delta_{i,j} \end{pmatrix}
  \begin{pmatrix}{^t u}^{-1} & 0_n \\ 0_n & u \end{pmatrix} $ runs over all elements in the set $\mathfrak{S}_{\alpha}$.

Now we shall express $C_2(M,S)$ as a linear combination of Fourier coefficients $C_1(M,S)$ of $\iota_{M}(\phi)$.
For $Y = (\mbox{diag}(1_{\alpha},p 1_{n-\alpha},p^2 1_{\alpha}, p 1_{n-\alpha}),p^{\alpha/2}) \in \widetilde{\mbox{GSp}_n^+(\Z)}$
a complete set of representatives of
$\Gamma_0^{(n)}(4)^* \backslash \Gamma_0^{(n)}(4)^* Y \Gamma_0^{(n)}(4)^*$ is given by elements
\begin{eqnarray*}
 \widetilde{g}
 &=&
 (g,\varepsilon(g) p^{(n-i-j)/2}) \in \widetilde{\mbox{GSp}_n^+(\Z)},
\end{eqnarray*}
where $g$ runs over all elements in the set $\mathfrak{S}_{\alpha}$,
and $\varepsilon(g)$ is defined as the above (cf. \cite[Lemma 3.2]{Zhu:euler}).
Hence
\begin{eqnarray*}
 &&
 (\iota_{\mathcal{M}}(\phi)|\tilde{V}_{\alpha,n-\alpha}(p^2))(\tau,z)
\\
 &=&
 p^{n(2k-1)/2 - n(n+1)}\sum_{M,S} \sum_{\widetilde{g}} p^{(-k+1/2)(n-i-j)} \varepsilon(g)\,
  C_1(M,S) \\
 && \qquad \qquad \qquad \qquad \qquad \qquad
  \times e(M(p^2 {^t D}^{-1}\tau + B) D^{-1} + p^2 S {^t z} D^{-1})
\\
 &=&
 p^{n(2k-1)/2 - n(n+1)}\sum_{\hat{M},\hat{S}} \sum_{g \in \mathfrak{S}_{\alpha}} p^{(-k+1/2)(n-i-j)} \varepsilon(g)\,
  C_1(p^{-2} D \hat{M}{^t D}, p^{-2} D \hat{S})
\\
 && \qquad \qquad \qquad \qquad \qquad \qquad
  \times e(\hat{M} \tau + \hat{S} {^t z} + p^{-2} \hat{M} {^t D} B).
\end{eqnarray*}
Thus 
\begin{eqnarray*}
 C_2(\hat{M},\hat{S})
 &=&
  \sum_{g} p^{-n(n+1)+(k-1/2)(i+j)} \varepsilon(g)\,
  C_1(p^{-2} D \hat{M}{^t D}, p^{-2} D \hat{S})\, e(p^{-2} \hat{M} {^t D} B).
\end{eqnarray*}

Now we put
$\hat{M} = 4 \hat{N} - \hat{R}_2 {^t \hat{R}_2 }$ and
$\hat{S} = 4 \hat{R} \left(\begin{smallmatrix}1\\0 \end{smallmatrix}\right) - 2 r p \hat{R}_2$,
then
\begin{eqnarray*}
 C_2(4 \hat{N} - \hat{R}_2 {^t \hat{R}_2 },4 \hat{R} \left(\begin{smallmatrix}1\\0 \end{smallmatrix}\right) - 2 r p \hat{R}_2)
 &=&
 p^{2nk+k - n^2 -\frac{7}{2}n  + \frac12 \alpha} A_2(\hat{N},\hat{R}).
\end{eqnarray*}
The proposition follows from this identity.
\end{proof}

\subsection{Index-shift maps at $p=2$}\label{ss:hecke_p2}
For $p=2$ we define the index-shift map $\tilde{T}_{\alpha,n-\alpha}(p^2)$ on $J_{k-\frac12,m}^{(n)*}$
through the identity~\ref{id:iota_hecke}, namely we define
\begin{eqnarray*}
 \phi|\tilde{V}_{\alpha,n-\alpha}(4)
 &:=&
 2^{k(2n+1) - n (n+\frac72) + \frac12 \alpha }\
 \iota_{\mathcal{M}[\smat{2}{ }{ }{1}]}(\psi| V_{\alpha,n-\alpha}(4))
\end{eqnarray*}
for any $\phi \in J_{k-\frac12,m}^{(n)*}$, and where $\psi \in J_{k,\mathcal{M}}^{(n)*}$
is the Jacobi form which satisfies $\iota_{\mathcal{M}}(\psi) = \phi$.
We have $\phi|\tilde{V}_{\alpha,n-\alpha}(4) \in J_{k-\frac12,4m}^{(n)}$.
The reader is referred to \S\ref{ss:iota} for the definition of $\iota_{\mathcal{M}}$
and is referred to \S\ref{ss:hecke_operators} for the definition of $V_{\alpha,n-\alpha}(4)$.

\subsection{Index-shift maps $V_{p}^{(1)}$, $V_{1,p}^{(2)}$ and $V_{2,p}^{(2)}$}\label{ss:hecke_V}
In the case $n = 1$, $2$, we write simply
\begin{eqnarray*}
 \varphi|V_{p}^{(1)}
 &:=&
 \varphi|\tilde{V}_{1,0}(p^2) ,\\ 
 \phi|V_{1,p}^{(2)}
 &:=&
 p^{-k+\frac72} \phi|\tilde{V}_{1,1}(p^2) ,
\\
 \phi|V_{2,p}^{(2)}
 &:=&
 \phi|\tilde{V}_{2,0}(p^2)
\end{eqnarray*}
for any $\varphi \in J_{k-\frac12,m}^{(1)*}$, $\phi \in J_{k-\frac12,m}^{(2)*}$ and for any prime $p$.
The reader is referred to \S\ref{ss:fj_expansion_half}
for the definition of the subspace $J_{k-\frac12,m}^{(n)*}$ of $J_{k-\frac12,m}^{(n)}$,
and is referred to \S\ref{ss:hecke_operators} and \S\ref{ss:hecke_p2}
for the definition of the index-shift maps $\tilde{V}_{1,0}(p^2)$, $\tilde{V}_{1,1}(p^2)$ and $\tilde{V}_{2,0}(p^2)$.
We remark that $\varphi|V_{p}^{(1)} \in J_{k-\frac12,mp^2}^{(1)}$ and $\phi|V_{i,p}^{(2)} \in J_{k-\frac12,mp^2}^{(2)}$ $(i=1,2)$.

\section{Action of the index-shift maps on Jacobi-Eisenstein series}
In this section we fix a positive definite half-integral symmetric matrix
$\mathcal{M} = \left( \begin{smallmatrix} * & * \\ * & 1 \end{smallmatrix} \right) \in \mbox{Sym}_2^+$.
The purpose of this section is to express the function $E_{k,\mathcal{M}}^{(n)}|V_{\alpha,n-\alpha}(p^2)$
as a summation of certain exponential functions with generalized Gauss sums,
where $E_{k,\mathcal{M}}^{(n)}$ is the Jacobi-Eisenstein series of index $\mathcal{M}$ (cf. \S\ref{s:fourier_matrix}),
and where $V_{\alpha,n-\alpha}(p^2)$ is an index-shift map (cf. \S\ref{ss:hecke_operators}).
We remark that $E_{k,\mathcal{M}}^{(n)}|V_{\alpha,n-\alpha}(p^2) \in J_{k,\mathcal{M}[\smat{p}{0}{0}{1}]}^{(n)}$ for any
$0 \leq \alpha \leq n$.

First we will express $E_{k,\mathcal{M}}^{(n)}|V_{\alpha,n-\alpha}(p^2)$ as a summation of certain
functions $K_{i,j}^{\beta}$ (cf. Lemma~\ref{lemma:kij}),
and after that, we will express $E_{k,\mathcal{M}}^{(n)}|V_{\alpha,n-\alpha}(p^2)$ as a summation of
certain functions $\tilde{K}_{i,j}^{\beta}$ (cf. Proposition~\ref{prop:tilde_kij}).

The calculation in this section is an analogue to the one given in \cite{Ya2} in the case of index $1$.
However, we need to generalize his calculation for Jacobi-Eisenstein series $E_{k,1}^{(n)}$ of index $1$ to our case
for $E_{k,\mathcal{M}}^{(n)}$ with $\mathcal{M} = \smat{*}{*}{ *}{1} \in \mbox{Sym}_2^+$.
This generalization is not obvious, because we need to treat the action of
the Heisenberg group parts $[((0,u_2),(0,v_2)),0]$,
which plays an important rule in this generalization.

\subsection{The function $K_{i,j}^{\beta}$}\label{ss:kij}
The purpose of this subsection is to introduce functions $K_{i,j}^{\beta}$ and to express
$E_{k,\mathcal{M}}^{(n)}|V_{\alpha,n-\alpha}(p^2)$ as a summation of $K_{i,j}^{\beta}$.
Moreover, we shall calculate $K_{i,j}^{\beta}$ more precisely (cf. Lemma~\ref{lemma:kij}).

We put $\delta_{i,j} := \mbox{diag}(1_i,p 1_{j-i},p^2 1_{n-j})$.
For $x = \mbox{diag}(0_i,x',0_{n-i-j})$ with $x' = {^t x'} \in \Z^{(j-i,j-i)}$,
we set $\delta_{i,j}(x) := \left(\begin{matrix}p^2 \delta_{i,j}^{-1} & x \\ 0 & \delta_{i,j} \end{matrix}\right)$
and 
$\Gamma(\delta_{i,j}(x)) := \Gamma_n \cap \delta_{i,j}(x)^{-1}\Gamma_{\infty}^{(n)} \delta_{i,j}(x)$.

For $x = \mbox{diag}(0_i,x',0_{n-i-j})$, $y = \mbox{diag}(0_i,y',0_{n-i-j})$ with $x'={^t x'}, y' = {^t y'} \in \Z^{(j-i,j-i)}$,
we say that $x$ and $y$ are \textit{equivalent},
if there exists a matrix
$u \in \mbox{GL}_n(\Z) \cap \delta_{i,j}\mbox{GL}_n(\Z)\delta_{i,j}^{-1}$
which has a form
$u = 
 \left(\begin{smallmatrix}u_1& * & * \\
                          * &u_2& * \\ 
                          * & * &u_3 \end{smallmatrix}\right)$
satisfying
$x' \equiv u_2\, y'\, {^t u_2} \mod p$,
where $u_2 \in \Z^{(j-i,j-i)}$, $u_1 \in \Z^{(i,i)}$ and $u_3\in \Z^{(n-j,n-j)}$.

We denote by $[x]$ the equivalent class of $x$.
We quote the following lemma from \cite{Ya2}.
\begin{lemma}\label{lemma:hecke_decom}
The double coset $\Gamma_n \mbox{diag}(1_{\alpha},p 1_{n-\alpha}, p^2 1_{\alpha}, p 1_{n-\alpha})\Gamma_n$ is 
written as a disjoint union 
 \begin{eqnarray*}
  \Gamma_n \left( \begin{smallmatrix}1_{\alpha} &&&\\
                                &p 1_{n-\alpha}&&\\
                                &&p^2 1_{\alpha}&\\
                                &&&p 1_{n-\alpha}  \end{smallmatrix}\right)\Gamma_n
  &=&
  \bigcup_{\begin{smallmatrix} i,j \\ 0\leq i \leq j \leq n \end{smallmatrix}}
  \bigcup_{[x]}
   \Gamma_{\infty}^{(n)} \delta_{i,j}(x) \Gamma_n ,
 \end{eqnarray*}
where $[x]$ runs over all equivalent classes which satisfy $\mbox{rank}_p(x) = j - i - n + \alpha \geq 0$.
\end{lemma}
\begin{proof}
The reader is referred to~\cite[Corollary~2.2]{Ya2}.
\end{proof}

We put $U := \begin{pmatrix} p^2&0\\0&p \end{pmatrix}$.

By the definition of the index-shift map $V_{\alpha,n-\alpha}(p^2)$
and of the Jacobi-Eisenstein series $E_{k,\mathcal{M}}^{(n)}$, we have
\begin{eqnarray*}
 &&
 E_{k,\mathcal{M}}^{(n)}|V_{\alpha,n-\alpha}(p^2) \\
 &=&
 \sum_{u,v \in \Z^{(n,1)}}
 \sum_{M' \in \Gamma_n \backslash \Gamma_n diag(1_{\alpha},p1_{n-\alpha},p^2 1_{\alpha}, p 1_{n-\alpha})\Gamma_n}
 \sum_{M \in \Gamma_{\infty}^{(n)}\backslash \Gamma_n}
 \sum_{\lambda \in \Z^{(n,1)}}
\\
 &&
  1 |_{k,\mathcal{M}} ([(\lambda,0),0],MM'\times \smat{U}{0}{0}{p^2 U^{-1}})
    |_{k,\mathcal{M}[\smat{p}{0}{0}{1}]}[((0, u),(0, v)),0]
\\
 &=&
 \sum_{u,v \in \Z^{(n,1)}}
 \sum_{M \in \Gamma_{\infty}^{(n)} \backslash \Gamma_n
              diag(1_{\alpha},p1_{n-\alpha},p^2 1_{\alpha}, p 1_{n-\alpha})
             \Gamma_n}
 \sum_{\lambda \in \Z^{(n,1)}}
\\
 &&
  1 |_{k,\mathcal{M}} ([(\lambda,0),0],M\times \smat{U}{0}{0}{p^2 U^{-1}})
    |_{k,\mathcal{M}[\smat{p}{0}{0}{1}]}[((0, u),(0, v)),0] .
\end{eqnarray*}
Hence, due to Lemma~\ref{lemma:hecke_decom}, we have
\begin{eqnarray*}
 &&
 E_{k,\mathcal{M}}^{(n)}|V_{\alpha,n-\alpha}(p^2)
\\
 &=&
 \sum_{u,v \in \Z^{(n,1)}}
 \sum_{\begin{smallmatrix} i,j \\ 0\leq i \leq j \leq n \end{smallmatrix}}
 \sum_{\begin{smallmatrix} [x] \\ rank_p(x) = j-i-n+\alpha\end{smallmatrix}}
 \sum_{M \in \Gamma_{\infty}^{(n)}\backslash \delta_{i,j}(x)\Gamma_n}
 \sum_{\lambda \in \Z^{(n,1)}}
\\
 &&
  1 |_{k,\mathcal{M}} ([(\lambda,0),0],M\times\smat{U}{0}{0}{p^2 U^{-1}})
    |_{k,\mathcal{M}[\smat{p}{0}{0}{1}]}[((0, u),(0, v)),0]
\\
 &=&
 \sum_{u,v \in \Z^{(n,1)}}
 \sum_{\begin{smallmatrix} i,j \\ 0\leq i \leq j \leq n \end{smallmatrix}}
 \sum_{\begin{smallmatrix} [x] \\ rank_p(x) = j-i-n+\alpha \end{smallmatrix}}
 \sum_{M \in \Gamma(\delta_{i,j}(x))\backslash \Gamma_n}
 \sum_{\lambda \in \Z^{(n,1)}}
\\
 &&
  1 |_{k,\mathcal{M}} ([(\lambda,0),0],\delta_{i,j}(x)M\times\smat{U}{0}{0}{p^2 U^{-1}})
    |_{k,\mathcal{M}[\smat{p}{0}{0}{1}]}[((0, u),(0, v)),0] .
\end{eqnarray*}
For $\beta \leq j-i $ we define a function
\begin{eqnarray*}
 &&
 K_{i,j}^{\beta}(\tau,z)
\\
 &:=&
 K_{i,j,\mathcal{M},p}^{\beta}(\tau,z)
\\
 &=& \!\!
 \sum_{\begin{smallmatrix}[x] \\ rank_p(x) = \beta \end{smallmatrix}}
 \sum_{M \in \Gamma(\delta_{i,j}(x))\backslash \Gamma_n}\sum_{\lambda \in \Z^n}
 \left\{
  1|_{k,\mathcal{M}}(
   [(\lambda,0),0],
   \delta_{i,j}(x) M 
    \times\smat{U}{0}{0}{p^2 U^{-1}}
   )
 \right\}
  (\tau,z) .
\end{eqnarray*}
Then
\begin{eqnarray*}
 E_{k,\mathcal{M}}^{(n)}|V_{\alpha,n-\alpha}(p^2)
 &=&
 \sum_{\begin{smallmatrix} i,j \\  0\leq i \leq j \leq n \end{smallmatrix}}
 \sum_{u,v \in \Z^{(n,1)}}
  K_{i,j}^{\alpha-i-n+j}(\tau,z)
  |_{k,\mathcal{M}[\smat{p}{0}{0}{1}]}[((0, u),(0, v)),0] .
\end{eqnarray*}

We define
\begin{eqnarray*}
 L_{i,j}
 &:=&
 L_{i,j,\mathcal{M},p}
 =
 \left.
 \left\{ \left(\begin{smallmatrix}\lambda_1\\ \lambda_2\\ \lambda_3 \end{smallmatrix}\right)
   \ \right| \
 \begin{matrix}
  \lambda_1 \in (p\Z)^{(i,2)}\, , \, \lambda_2 \in \Z^{(j-i,2)}\, , \, \lambda_3 \in (p^{-1}\Z)^{(n-j,2)}
 \\
  2 \lambda_2 \mathcal{M} {^t \lambda_3} \in \Z^{(j-i,n-j)} \, , \, 
  \lambda_3 \mathcal{M} {^t \lambda_3} \in \Z^{(n-j,n-j)}
 \end{matrix}
 \right\}.
\end{eqnarray*}
Moreover, we define a subgroup $\Gamma(\delta_{i,j})$ of $\Gamma_{\infty}^{(n)}$ by
\begin{eqnarray*}
 \Gamma(\delta_{i,j})
 &:=&
 \left.
 \left\{
  \begin{pmatrix}
   A & B \\ 0_n & {^t A}^{-1}
  \end{pmatrix}
  \in
  \Gamma_{\infty}^{(n)}
 \, \right| \,
   A \in \delta_{i,j}\mbox{GL}_n(\Z) \delta_{i,j}^{-1}
 \right\} .
\end{eqnarray*}

\begin{lemma} \label{lemma:kij}
Let $K_{i,j}^{\beta}$ be as the above.
We obtain
\begin{eqnarray*}
 K_{i,j}^{\beta}(\tau,z)
 &=&
 p^{-k(2n-i-j+1)+(n-j)(n-i+1)}
 \sum_{M \in \Gamma(\delta_{i,j})\backslash \Gamma_n}
\\
 && \times \sum_{\lambda \in L_{i,j}}
 1|_{k,\mathcal{M}}([(\lambda,0),0_n],M)(\tau,z\smat{p}{0}{0}{1}) \!\!\!
 \sum_{\begin{smallmatrix} x = ^t x \in (\Z/p\Z)^{(n,n)} \\
                           x = diag(0_i,x',0_{n-j}) \\
                           rank_p(x') = \beta \end{smallmatrix}} \!\!\!
 e\left(\frac{1}{p}\mathcal{M}{^t\lambda} x \lambda\right),
\end{eqnarray*}
where, in the last summation, $x$ runs over
a complete set of representatives of $(\Z/p\Z)^{(n,n)}$ such that $x = {^t x}$,
$\mbox{rank}_p(x) = \beta$ and
$x = \mbox{diag}(0_i,x',0_{n-j})$ with some $x' \in (\Z/p\Z)^{(j-i,j-i)}$.
\end{lemma}
\begin{proof}
We proceed as in~\cite[Proposition~3.2]{Ya2}.
The inside of the last summation of $K_{i,j}^{\beta}(\tau,z)$ is
\begin{eqnarray*}
 &&
 \left(1|_{k,\mathcal{M}}(
   [(\lambda,0),0],
   \delta_{i,j}(x) M 
    \times\smat{U}{0}{0}{p^2 U^{-1}})\right)(\tau,z)
\\
 &=&
 \det(p^2U^{-1})^{-k}\det(\delta_{i,j})^{-k}
\\
 &&
 \times
 \left(e(\mathcal{M}(^t\lambda(p^2 \delta_{ij}^{-1}\tau+ x)\delta_{ij}^{-1}\lambda 
 + 2 ^t \lambda \delta_{ij}^{-1} z \begin{pmatrix} p^2 & \\ & p \end{pmatrix}))
 |_{k,\mathcal{M}[\left(\begin{smallmatrix} p&0\\0&1\end{smallmatrix} \right)]} M\right)(\tau,z)
\\
 &=&
 p^{-k(2n-i-j+1)} 
\\
 && \times
 \left(\left((1|_{k,\mathcal{M}}([(p\delta_{i,j}^{-1}\lambda,0),0],\begin{pmatrix}1&p^{-1}x\\0&1\end{pmatrix}))
  (\tau,z\smat{p}{0}{0}{1})\right)|_{k,\mathcal{M}[\left(\begin{smallmatrix} p&0\\0&1\end{smallmatrix} \right)]} M\right)(\tau,z)
\\
 &=&
 p^{-k(2n-i-j+1)}
 \left(1|_{k,\mathcal{M}}([(p\delta_{i,j}^{-1}\lambda,0),0],\begin{pmatrix}1&p^{-1}x\\0&1\end{pmatrix}M)\right)
  (\tau,z\smat{p}{0}{0}{1}) .
\end{eqnarray*}
Here we used the identity $\delta_{i,j} x = \delta_{i,j} \mbox{diag}(0_i,x',0_{n-j}) = p x$.
Thus
\begin{eqnarray*}
 K_{i,j}^{\beta}(\tau,z)
 &=&
 p^{-k(2n-i-j+1)}
 \sum_{\begin{smallmatrix}[x] \\ rank_p(x) = \beta \end{smallmatrix}}
 \sum_{M \in \Gamma(\delta_{i,j}(x))\backslash \Gamma_n}
\\
 &&
 \times \sum_{\lambda \in \Z^n}
  1|_{k,\mathcal{M}}\left([(p\delta_{i,j}^{-1}\lambda,0),0_n],\begin{pmatrix}1&p^{-1}x\\0&1\end{pmatrix}M\right)
   (\tau,z\smat{p}{0}{0}{1}) .
\end{eqnarray*}
We put
\begin{eqnarray*}
 \mathcal{U}
 &:=&
 \left.
 \left\{\begin{pmatrix}1_n& s \\ 0_n & 1_n \end{pmatrix} \, \right| \, 
  s = {^t s} \in \Z^{(n,n)}
 \right\} .
\end{eqnarray*}
Then the set
\begin{eqnarray*}
 \left.
 \left\{
  \begin{pmatrix}1_n& s \\ 0_n & 1_n \end{pmatrix}
 \, \right| \,
  s = \left(\begin{smallmatrix}0&0&0\\0&0&s_2\\0&^t s_2&s_3\end{smallmatrix}\right),
  s_2 \in (\Z/p\Z)^{(j-i,n-j)},
  s_3 = {^t s_3} \in (\Z/p \Z)^{(n-j,n-j)}
 \right\}
\end{eqnarray*}
is a complete set of representatives of $\Gamma(\delta_{i,j}(x))\backslash \Gamma(\delta_{i,j}(x))\mathcal{U}$.
Therefore
\begin{eqnarray*}
 &&
 K_{i,j}^{\beta}(\tau,z)
\\
 &=&
 p^{-k(2n-i-j+1)}
 \sum_{\begin{smallmatrix}[x] \\ rank_p(x) = \beta \end{smallmatrix}}
 \sum_{M \in (\Gamma(\delta_{i,j}(x))\mathcal{U})\backslash \Gamma_n}\sum_{\lambda \in \Z^n}
 \sum_{\smat{1_n}{s}{0}{1_n} \in \Gamma(\delta_{i,j}(x))\backslash(\Gamma(\delta_{i,j}(x))\mathcal{U})}
\\
 &&
 \times\quad 1|_{k,\mathcal{M}}([(p\delta_{i,j}^{-1}\lambda,0),0_n],\begin{pmatrix}1_n&p^{-1}x\\0&1_n\end{pmatrix}
                            \begin{pmatrix}1_n&s\\0&1_n\end{pmatrix}M)(\tau,z\smat{p}{0}{0}{1})
\end{eqnarray*}
Hence
\begin{eqnarray*}
 K_{i,j}^{\beta}(\tau,z)
 &=&
 p^{-k(2n-i-j+1)}
 \sum_{\begin{smallmatrix}[x] \\ rank_p(x) = \beta \end{smallmatrix}}
 \sum_{M \in (\Gamma(\delta_{i,j}(x))\mathcal{U})\backslash \Gamma_n}\sum_{\lambda \in \Z^n} 
\\
 &&
 \times\quad
 1|_{k,\mathcal{M}}
  ([(p\delta_{i,j}^{-1}\lambda,0),0_n],\begin{pmatrix}1_n&p^{-1}x\\0&1_n\end{pmatrix}M)(\tau,z\smat{p}{0}{0}{1})
\\
 &&
 \times
 \sum_{\smat{1_n}{s}{0}{1_n} \in \Gamma(\delta_{i,j}(x))\backslash(\Gamma(\delta_{i,j}(x))\mathcal{U})}
 e\left(p^2 \mathcal{M} {^t\lambda} \delta_{i,j}^{-1} s \delta_{i,j}^{-1} \lambda \right) .
\end{eqnarray*}
The last summation of the above identity is
\begin{eqnarray*}
&&
 \sum_{s} e\left(p^2 \mathcal{M} {^t\lambda} \delta_{i,j}^{-1} s \delta_{i,j}^{-1} \lambda \right)
\\
&=&
 \begin{cases}
   p^{(n-j)(n-i+1)} & \mbox{ if } \lambda_3 \mathcal{M} {^t\lambda_3} \equiv 0 \!\! \mod p^2  
                                  \mbox{ and } 
                                  2 \lambda_3 \mathcal{M} {^t\lambda_2} \equiv 0 \!\! \mod p ,
                                  \\
   0 & \mbox{ otherwise, }
 \end{cases}
\end{eqnarray*}
where 
$\lambda = \left(\begin{smallmatrix} \lambda_1 \\ \lambda_2 \\ \lambda_3 \end{smallmatrix} \right) \in \Z^{(n,2)}$
with $\lambda_1 \in \Z^{(i,2)}$, $\lambda_2 \in \Z^{(j-i,2)}$ and $\lambda_3 \in \Z^{(n-j,2)}$.

Thus
\begin{eqnarray*}
 K_{i,j}^{\beta}(\tau,z)
 &=&
 p^{-k(2n-i-j+1)+(n-j)(n-i+1)}\sum_{\begin{smallmatrix}[x] \\ rank_p(x) = \beta \end{smallmatrix}}
 \sum_{M \in (\Gamma(\delta_{i,j}(x))\mathcal{U})\backslash \Gamma_n}
\\
 && \times \sum_{\lambda \in L_{i,j}}
  1|_{k,\mathcal{M}}
  ([(\lambda,0),0_n],\smat{1_n}{p^{-1}x}{0}{1_n}M)(\tau,z\smat{p}{0}{0}{1}) .
\end{eqnarray*}
Now
$\Gamma(\delta_{i,j}(x))\mathcal{U}$ is a subgroup of $\Gamma(\delta_{i,j})$.
For any $\smat{A}{B}{0_n}{{^t A}^{-1}} \in \Gamma(\delta_{i,j})$ we have
\begin{eqnarray*}
 &&
 1|_{k,\mathcal{M}}([(\lambda,0),0_n],\smat{1_n}{p^{-1}x}{0_n}{1_n} \smat{A}{B}{0_n}{^t A^{-1}} M)
 \\
 &=&
 1|_{k,\mathcal{M}}([(\lambda,0),0_n],\smat{A}{B}{0_n}{{^t A}^{-1}}\smat{1_n}{p^{-1}A^{-1}x{{^t A}^{-1}}}{0_n}{1_n}M)
 \\
 &=&
 1|_{k,\mathcal{M}}([({^t A}\lambda,{^t B}\lambda),0_n],\smat{1_n}{p^{-1}A^{-1}x{{^t A}^{-1}}}{0_n}{1_n}M)
 \\
 &=&
 1|_{k,\mathcal{M}}([({^t A}\lambda,0),0_n],\smat{1_n}{p^{-1}A^{-1}x{{^t A}^{-1}}}{0_n}{1_n}M),
\end{eqnarray*}
and ${^tA}L_{i,j} = L_{i,j}$.
Moreover, when $\smat{A}{B}{0_n}{{^t A}^{-1}}$ runs over all elements in a complete set of representatives
of $\Gamma(\delta_{i,j}(x))\mathcal{U} \backslash \Gamma(\delta_{i,j})$, then  $A^{-1}x {^t A}^{-1}$ runs over all elements in the equivalent class $[x]$ (cf.~\cite[proof of Proposition~3.2]{Ya2}).
Therefore, we have
\begin{eqnarray*}
 &&
 K_{i,j}^{\beta}(\tau,z) \\
 &=&
 p^{-k(2n-i-j+1)+(n-j)(n-i+1)}
 \sum_{\begin{smallmatrix} x = ^t x \in (\Z/p\Z)^{(n,n)} \\ 
                           x = diag(0_i,x',0_{n-j}) \\
                           rank_p(x') = \beta 
       \end{smallmatrix}}
 \sum_{M \in \Gamma(\delta_{i,j})\backslash \Gamma_n}
\\
 && \times \sum_{\lambda \in L_{i,j}}
 1|_{k,\mathcal{M}}([(\lambda,0),0_n],\smat{1_n}{p^{-1}x}{0}{1_n}M)(\tau,z\smat{p}{0}{0}{1})
\\
 &=&
 p^{-k(2n-i-j+1)+(n-j)(n-i+1)}
 \sum_{M \in \Gamma(\delta_{i,j})\backslash \Gamma_n}
\\
 && \times \sum_{\lambda \in L_{i,j}}
 1|_{k,\mathcal{M}}([(\lambda,0),0_n],M)(\tau,z\smat{p}{0}{0}{1})
 \sum_{\begin{smallmatrix} x = ^t x \in (\Z/p\Z)^{(n,n)} \\
                           x = diag(0_i,x',0_{n-j}) \\
                           rank_p(x') = \beta \end{smallmatrix}}
 e\left(\frac{1}{p}\mathcal{M}{^t\lambda} x \lambda\right).
\end{eqnarray*}
\end{proof}

\subsection{The function $\tilde{K}_{i,j}^{\beta}$}\label{ss:tilde_kij}
The purpose of this subsection is to introduce functions $\tilde{K}_{i,j}^{\beta}$ and to express
$E_{k,\mathcal{M}}^{(n)}|V_{\alpha,n-\alpha}(p^2)$ as a summation of $\tilde{K}_{i,j}^{\beta}$.
Moreover, we shall show that $\tilde{K}_{i,j}^{\beta}$ is a summation of certain exponential functions
with generalized Gauss sums (cf. Proposition~\ref{prop:tilde_kij}).

We define
\begin{eqnarray*}
 L_{i,j}^*
 &:=&
 L_{i,j,\mathcal{M},p}^* \\
 &=&
 \left\{
  \left( \begin{matrix} \lambda_1\\ \lambda_2\\ \lambda_3 \end{matrix} \right)
  \in (p^{-1}\Z)^{(n,2)}
 \, \left| \,
  \begin{matrix}
  \lambda_1\smat{p}{0}{0}{1}^{-1} \in \Z^{(i,2)}, \
  \lambda_2 \in \Z^{(j-i,2)} \\
  \lambda_3 \in (p^{-1}\Z)^{(n-j,2)},\
  2 \lambda_2 \mathcal{M} {^t \lambda_3} \in \Z^{(j-i,n-j)} \\
  \lambda_3 \mathcal{M} {^t\lambda_3} \in \Z^{(n-j,n-j)}, \
  2 \lambda_3 \mathcal{M} \left(\begin{smallmatrix}0\\1 \end{smallmatrix}\right) \in \Z^{(n-j,1)}
  \end{matrix}
 \right\}
 \right.
\end{eqnarray*}
and define a generalized Gauss sum
\begin{eqnarray*}
 G_{\mathcal{M}}^{j-i,l}(\lambda_2) 
 &:=& 
 \sum_{\begin{smallmatrix} x' = ^t x' \in (\Z/p\Z)^{(j-i,j-i)} \\
                           rank_p(x') = j-i-l
       \end{smallmatrix}}
  e\left(\frac{1}{p}\mathcal{M} {^t \lambda_2 } x'  \lambda_2\right)
\end{eqnarray*}
for $\lambda_2 \in \Z^{(j-i,2)}$.
We remark that a formula for this generalized Gauss sum is given
by Saito \cite{Sa}. 
We define
\begin{eqnarray*}
 \tilde{K}_{i,j}^{\beta}(\tau,z)
 &:=&
 \tilde{K}_{i,j,\mathcal{M},p}^{\beta}(\tau,z)
 = \!\!\!\!\!\!\!\!
 \sum_{u,v \in (\Z/p\Z)^{(n,1)}}\left(K_{i,j}^{\beta}
  |_{k,\mathcal{M}[\smat{p}{0}{0}{1}]}
    [\left((0,u),(0,v)\right), 0_n ] \right)(\tau,z).
\end{eqnarray*}

\begin{prop}\label{prop:tilde_kij}
Let the notation be as above. Then
\begin{eqnarray*} 
 E_{k,\mathcal{M}}^{(n)}|V_{\alpha,n-\alpha}(p^2)
 &=&
 \sum_{0\leq i \leq j \leq n}
  \tilde{K}_{i,j}^{\alpha-i-n+j}(\tau,z) ,
\end{eqnarray*}
where
\begin{eqnarray*}
 \tilde{K}_{i,j}^{\alpha-i-n+j}(\tau,z)
 &=&
 p^{-k(2n-i-j+1)+(n-j)(n-i+1)+2n-j }
\\
 &&
 \times
 \sum_{M \in \Gamma(\delta_{i,j})\backslash \Gamma_n}
 \sum_{ \lambda = \left(\begin{smallmatrix}\lambda_1\\ \lambda_2\\ \lambda_3 \end{smallmatrix}\right)\in L_{i,j}^*}
  \left\{1|_{k,\mathcal{M}}([(\lambda,0),0_n],M)\right\}(\tau,z\smat{p}{0}{0}{1})
\\
 &&
 \times
  \sum_{u_2 \in (\Z/p\Z)^{(j-i,1)}}G_{\mathcal{M}}^{j-i,n-\alpha}(\lambda_2 + (0,u_2)) .
\end{eqnarray*}
\end{prop}
\begin{proof}
From the definition of $\tilde{K}_{i,j}^{\beta}$ and Lemma~\ref{lemma:kij} we obtain
\begin{eqnarray*}
 &&
 \tilde{K}_{i,j}^{\alpha-i-n+j}(\tau,z)
\\
 &=&
 p^{-k(2n-i-j+1)+(n-j)(n-i+1)}
 \sum_{M \in \Gamma(\delta_{i,j})\backslash \Gamma_n}
 \sum_{\lambda = \left(\begin{smallmatrix}\lambda_1\\ \lambda_2\\ \lambda_3 \end{smallmatrix}\right)\in L_{i,j}}
 G_{\mathcal{M}}^{j-i,n-\alpha}(\lambda_2)
\\
 &&
 \times \!\!\!
 \sum_{u,v \in (\Z/p\Z)^{(n,1)}}
 \left(1|_{k,\mathcal{M}}([(\lambda,0),0_n],M)(\tau,z\smat{p}{0}{0}{1})\right)
  |_{k,\mathcal{M}[\smat{p}{0}{0}{1}]}[\left((0,u),(0,v)\right), 0_n],
\end{eqnarray*}
where, in the second summation, $\lambda_1 \in \Z^{(i,2)}$, $\lambda_2 \in \Z^{(j-i,2)}$, $\lambda_3 \in \Z^{(n-j,2)}$ and
$\lambda \in L_{i,j}$.

By a straightforward calculation we have
\begin{eqnarray*}
 1|_{k,\mathcal{M}}([(\lambda,0),0_n],M)(\tau,z\smat{p}{0}{0}{1})
 &=&
 1|_{k,\mathcal{M}[\smat{p}{0}{0}{1}]}([(\lambda\smat{p}{0}{0}{1}^{-1},0),0_n],M)(\tau,z).
\end{eqnarray*}
Hence the last summation of the above identity is~
\begin{eqnarray*}
 && \!\!\!\!\!\!\!
 \sum_{u,v \in (\Z/p\Z)^{(n,1)}}
 \!\!\!
 \left\{
  1|_{k,\mathcal{M}}([(\lambda,0),0_n],M)(\tau,z\smat{p}{0}{0}{1})
   |_{k,\mathcal{M}[\smat{p}{0}{0}{1}]}[\left((0,u),(0,v)\right), 0_n]
 \! \right\}\! (\tau,z)
\\
 &=&
 \sum_{u',v' \in (\Z/p\Z)^{(n,1)}}
 \left\{1|_{k,\mathcal{M}[\smat{p}{0}{0}{1}]}([(\lambda\smat{p}{0}{0}{1}^{-1}+(0,u'),(0,v')),0_n],M)\right\}
 (\tau,z)
\\
 &=&
 \sum_{u',v' \in (\Z/p\Z)^{(n,1)}}
 \left\{1|_{k,\mathcal{M}}([(\lambda+(0,u'),(0,v')),0_n],M)\right\}
 (\tau,z\smat{p}{0}{0}{1})
\\
 &=&
 \sum_{u' \in (\Z/p\Z)^{(n,1)}} \!\!\!
 \left\{1|_{k,\mathcal{M}}([(\lambda+(0,u'),0),0_n],M)\right\}
 (\tau,z\smat{p}{0}{0}{1}) \!\!\!\!\!
 \sum_{v' \in (\Z/p\Z)^{(n,1)}} \!\!\!\!\!\! e(2 \mathcal{M} {^t\lambda}(0, v')),
\end{eqnarray*}
where, in the first identity, we used
\begin{eqnarray*}
 (M,[((0,u),(0,v)),0_n])
 &=&
 ([((0,u'),(0,v')),0_n],M)
\end{eqnarray*}
with
$\left(\begin{smallmatrix} u' \\ v'\end{smallmatrix}\right)
  =
 \smat{D}{-C}{-B}{A}
 \left(\begin{smallmatrix} u \\ v \end{smallmatrix} \right)$
for any $M = \smat{A}{B}{C}{D} \in \Gamma_n$.
Now, for $\lambda = \left(\begin{smallmatrix}\lambda_1\\ \lambda_2 \\ \lambda_3 \end{smallmatrix}\right) \in L_{i,j}$
we have
\begin{eqnarray*}
 \sum_{v' \in (\Z/p\Z)^{(n,1)}} e(2 \mathcal{M} {^t\lambda} (0,v'))
 &=&
 \begin{cases}
  p^n & \mbox{if } 2 \lambda_3 \mathcal{M}\left(\begin{smallmatrix}0 \\ 1 \end{smallmatrix}\right) \in \Z^{(n-j,1)},
 \\
  0   & \mbox{otherwise}.
 \end{cases}
\end{eqnarray*}
Therefore
\begin{eqnarray*}
 &&
 \tilde{K}_{i,j}^{\alpha-i-n+j}(\tau,z)
\\
 &=&
 p^{-k(2n-i-j+1)+(n-j)(n-i+1)+n}
 \sum_{M \in \Gamma(\delta_{i,j})\backslash \Gamma_n}
 \sum_{\begin{smallmatrix}
        \lambda = \left(\begin{smallmatrix}\lambda_1\\ \lambda_2\\ \lambda_3 \end{smallmatrix}\right)\in L_{i,j} \\
        2 \lambda_3 \mathcal{M} \left(\begin{smallmatrix} 0\\1 \end{smallmatrix}\right) \in \Z^{(n-j,1)}
       \end{smallmatrix}}
 G_{\mathcal{M}}^{j-i,n-\alpha}(\lambda_2)
\\
 &&
 \times
 \sum_{u \in (\Z/p\Z)^{(n,1)}}
  \left\{1|_{k,\mathcal{M}}([(\lambda+(0,u),0),0_n],M)\right\}(\tau,z\smat{p}{0}{0}{1}).
\end{eqnarray*}
Thus
\begin{eqnarray*}
 &&
 \tilde{K}_{i,j}^{\alpha-i-n+j}(\tau,z)
\\
 &=&
 p^{-k(2n-i-j+1)+(n-j)(n-i+1)+n }
 \sum_{M \in \Gamma(\delta_{i,j})\backslash \Gamma_n}
 \sum_{ \lambda = \left(\begin{smallmatrix}\lambda_1\\ \lambda_2\\ \lambda_3 \end{smallmatrix}\right)\in L_{i,j}^*}
\\
 &&
 \times
  \left\{1|_{k,\mathcal{M}}([(\lambda,0),0_n],M)\right\}(\tau,z\smat{p}{0}{0}{1})
\\
 &&
 \times
  p^{n-j}\sum_{u_2 \in (\Z/p\Z)^{(j-i,1)}}G_{\mathcal{M}}^{j-i,n-\alpha}(\lambda_2 + (0,u_2)),
\end{eqnarray*}
where $L_{i,j}^*$ is defined as before.
\end{proof}

\section{Proof of Theorem~\ref{thm:deg2}}\label{s:thm:deg2}
The purpose of this section is to give the proof of Theorem~\ref{thm:deg2}.
From Proposition \ref{prop:tilde_kij} we recall
\begin{eqnarray*}
 E_{k,\mathcal{M}}^{(1)}|V_{\alpha,1-\alpha}(p^2)
 &=&
 \sum_{0\leq i \leq j \leq 1}
  \tilde{K}_{i,j}^{\alpha-i-1+j}.
\end{eqnarray*}
Hence
\begin{eqnarray*}
 E_{k,\mathcal{M}}^{(1)}|V_{1,0}(p^2)
 &=&
 \tilde{K}_{1,1}^{0}
 +
 \tilde{K}_{0,1}^{1}
 +
 \tilde{K}_{0,0}^{0}.
\end{eqnarray*}

\subsection{Calculation of $\tilde{K}_{i,j}^{\beta}$ for $n=1$}
In this subsection we calculate $\tilde{K}_{1,1}^{0}$, $\tilde{K}_{0,1}^{1}$ and $\tilde{K}_{0,0}^{0}$.
\begin{lemma}\label{lemma:E_V1}
For $\mathcal{M} \in \mbox{Sym}_2^+$ let
$D_0$ be the discriminant of $\Qq(\sqrt{-\det(2\mathcal{M})})$.
Let $f$ be the positive integer such that $-\det(2\mathcal{M}) = D_0 f^2$.
Then, for any prime $p$ we obtain
\begin{eqnarray*}
 \tilde{K}_{1,1}^{0}(\tau,z) &=&
 p^{-k+1}E_{k,\mathcal{M}[\smat{p}{ }{ }{1} ]}^{(1)}(\tau,z) ,\\
 \tilde{K}_{0,1}^{1}(\tau,z) &=&
 -p^{-2k+2} \left(\frac{D_0 f^2}{p}\right) E_{k,\mathcal{M}[\smat{p}{ }{ }{1}]}^{(1)}(\tau,z) \\
 &&
 +p^{-2k+2} \left(\frac{D_0 f^2}{p}\right) E_{k,\mathcal{M}}^{(1)}(\tau,z\smat{p}{ }{ }{1}) ,\\
 \tilde{K}_{0,0}^{0}(\tau,z) &=&
 \begin{cases}
  p^{-3k+4} E_{k,\mathcal{M}[X^{-1}{\smat{p}{ }{ }{1}}^{-1} ]}^{(1)}(\tau,z\smat{p}{ }{ }{1} {^t X} \smat{p}{ }{ }{1}) & \mbox{ if } p|f, \\
  p^{-3k+4} E_{k,\mathcal{M}}^{(1)}(\tau,z\smat{p}{ }{ }{1}) & \mbox{ if } p {\not|} f,
 \end{cases}
\end{eqnarray*}
where, in the case $p|f$, $X = \smat{1}{0}{x}{1} $ is a matrix such that $\mathcal{M}[X^{-1}\smat{p}{ }{ }{1}^{-1}] \in \mbox{Sym}_2^+$.
\end{lemma}
\begin{proof}
By the definition of $\tilde{K}_{i,j}^{\alpha-i-n+j}$ we have
\begin{eqnarray*}
 \tilde{K}_{1,1}^{0} &=&
 p^{-k+1}\sum_{M \in \Gamma(\delta_{1,1})\backslash \Gamma_1}
 \sum_{ \lambda_1 \in L_{1,1}^*}
 \left\{1|_{k,\mathcal{M}}([(\lambda_1,0),0],M)\right\}(\tau,z\smat{p}{0}{0}{1})\\
 &=&
 p^{-k+1}\sum_{M \in \Gamma_{\infty}^{(1)}\backslash \Gamma_1}
 \sum_{ \lambda_1 \in p\Z \times \Z}
 \left\{1|_{k,\mathcal{M}}([(\lambda_1,0),0],M)\right\}(\tau,z\smat{p}{0}{0}{1})\\
 &=&
 p^{-k+1}\sum_{M \in \Gamma_{\infty}^{(1)}\backslash \Gamma_1}
 \sum_{ \lambda \in \Z^{(1,2)}}
 \left\{1|_{k,\mathcal{M}}([(\lambda \smat{p}{ }{ }{1} ,0),0],M)\right\}(\tau,z\smat{p}{0}{0}{1})\\
 &=&
 p^{-k+1}\sum_{M \in \Gamma_{\infty}^{(1)}\backslash \Gamma_1}
 \sum_{ \lambda \in \Z^{(1,2)}}
 \left\{1|_{k,\mathcal{M}[\smat{p}{ }{ }{1} ]}([(\lambda,0),0],M)\right\}(\tau,z)\\
 &=&
 p^{-k+1} E_{k,\mathcal{M}[\smat{p}{ }{ }{1} ]}^{(1)}(\tau,z).
\end{eqnarray*}
Now we shall calculate $\tilde{K}_{0,1}^{1}$. First, for any $\lambda \in \Z^{(1,2)}$ we have
\begin{eqnarray*}
 G_{\mathcal{M}}^{1,0}(\lambda)
 &=&
 \sum_{\begin{smallmatrix} x \in \Z/p\Z \\ rank_p(x) = 1 \end{smallmatrix}} e\left(\frac{1}{p} \mathcal{M} {^t \lambda}x \lambda\right) \\
 &=&
 \begin{cases}
  p-1 & \mbox{ if } \lambda \mathcal{M} {^t \lambda} \equiv 0 \mod p , \\
  -1  & \mbox{ if } \lambda \mathcal{M} {^t \lambda} \not\equiv 0 \mod p .
 \end{cases}
\end{eqnarray*}
Hence we have
\begin{eqnarray*}
 \sum_{u \in \Z/p\Z} G_{\mathcal{M}}^{1,0}(\lambda + (0,u))
 &=&
 \begin{cases}
  0 & \mbox{ if } \lambda \in p\Z \times \Z ,\\
  \left(\frac{ D_0 f^2 }{p}\right) p & \mbox{ if } \lambda \not\in p\Z \times \Z.
 \end{cases}
\end{eqnarray*}
Thus
\begin{eqnarray*}
 \tilde{K}_{0,1}^{1}
 &=&
 p^{-2k+2}\sum_{M \in \Gamma(\delta_{0,1})\backslash \Gamma_1}
 \sum_{ \lambda_2 \in L_{0,1}^*}
 \left\{1|_{k,\mathcal{M}}([(\lambda_2,0),0],M)\right\}(\tau,z\smat{p}{0}{0}{1})\\
 && \times
 \sum_{u_2 \in \Z/p\Z}G_{\mathcal{M}}^{1,0}(\lambda_2 + (0,u_2)) \\
 &=&
 -p^{-2k+2} \left(\frac{D_0 f^2}{p} \right)\sum_{M \in \Gamma_{\infty}^{(1)}\backslash \Gamma_1}
 \sum_{ \lambda_2 \in p\Z \times \Z}
 \left\{1|_{k,\mathcal{M}}([(\lambda_2,0),0],M)\right\}(\tau,z\smat{p}{0}{0}{1}) \\
 &&
 +p^{-2k+2} \left(\frac{D_0 f^2}{p} \right)\sum_{M \in \Gamma_{\infty}^{(1)}\backslash \Gamma_1}
 \sum_{ \lambda_2 \in \Z \times \Z}
 \left\{1|_{k,\mathcal{M}}([(\lambda_2,0),0],M)\right\}(\tau,z\smat{p}{0}{0}{1}) .
\end{eqnarray*}
Finally, we shall calculate $\tilde{K}_{0,0}^{0}$.
We have
\begin{eqnarray*}
 L_{0,0}^*
 &=&
 \left\{
  \lambda_3 \in (p^{-1}\Z)^{(1,2)} \, | \,
   \lambda_3 \mathcal{M} {^t \lambda_3} \in \Z,\
   2 \lambda_3 \mathcal{M} \left(\begin{smallmatrix}0\\1 \end{smallmatrix}\right) \in \Z
 \right\}.
\end{eqnarray*}
We need to consider two cases: the case $p$ is an odd prime and the case $p=2$.
When $p$ is an odd prime, there exists a matrix $X = \smat{1}{ }{x}{1} \in \Z^{(2,2)}$ such that
the matrix $\mathcal{M}$ has an expression $\mathcal{M} \equiv {^t X}\smat{4^{-1}|D_0|f^2}{ }{ }{1} X \mod p$.
We have
\begin{eqnarray*}
 L_{0,0}^*
 &=&
 \begin{cases}
 \left\{\lambda_3 \in (p^{-1}\Z)^{(1,2)} \, | \, \lambda_3 {^t X} \in \frac{1}{p}\Z \times \Z \right\}
 & \mbox{ if } p|f ,\\
 \Z^{(1,2)}
 & \mbox{ if } p {\not|} f .
 \end{cases}
\end{eqnarray*}
Thus, if $p|f$, then
\begin{eqnarray*}
 \tilde{K}_{0,0}^{0}
 &=&
 p^{-3k+4} \sum_{M \in \Gamma(\delta_{0,0})\backslash \Gamma_1}
 \sum_{\begin{smallmatrix} \lambda_3 \\ \lambda_3 {^t X} \in \frac{1}{p}\Z \times \Z \end{smallmatrix}}
 \left\{1|_{k,\mathcal{M}}([(\lambda_3,0),0],M)\right\}(\tau,z\smat{p}{0}{0}{1}) \\
 &=&
 p^{-3k+4} \sum_{M \in \Gamma_{\infty}^{(1)}\backslash \Gamma_1}
 \sum_{\begin{smallmatrix} \lambda_3 \\ \lambda_3 {^t X} \smat{p}{ }{ }{1} \in \Z^{(1,2)} \end{smallmatrix}} \\
 && \times
 \left\{1|_{k,\mathcal{M}[X^{-1}\smat{p}{ }{ }{1}^{-1}]}([(\lambda_3 {^t X} \smat{p}{ }{ }{1},0),0],M)\right\}(\tau,z\smat{p}{ }{ }{1} {^t X} \smat{p}{ }{ }{1}) \\
 &=&
 p^{-3k+4} E_{k,\mathcal{M}[X^{-1}\smat{p}{ }{ }{1}^{-1}]}^{(1)}(\tau,z \smat{p}{ }{ }{1} {^t X} \smat{p}{ }{ }{1}),
\end{eqnarray*}
and, if $p {\not|} f$, then
\begin{eqnarray*}
  \tilde{K}_{0,0}^{0}
 &=&
 p^{-3k+4} \sum_{M \in \Gamma(\delta_{0,0})\backslash \Gamma_1}
 \sum_{\lambda_3  \in \Z^{(1,2)} }
 \left\{1|_{k,\mathcal{M}}([(\lambda_3,0),0],M)\right\}(\tau,z\smat{p}{0}{0}{1})\\
 &=&
 p^{-3k+4} E_{k,\mathcal{M}}^{(1)}(\tau,z \smat{p}{ }{ }{1}).
\end{eqnarray*}
When $p=2$, there exist a matrix $X = \smat{1}{ }{x}{1} \in \Z^{(2,2)} $ and an integer $u$
such that the matrix $\mathcal{M}$ equals one of the following three forms:
\begin{eqnarray*}
 && {^t X} \smat{u}{0}{0}{1} X \mbox{ with } u \equiv 0,1,2 \!\!\! \mod 4, \\
 && {^t X} \smat{u}{1}{1}{1} X \mbox{ with } u \equiv 0 \!\!\! \mod 4, \mbox{ or } \\
 && {^t X} \smat{u}{\frac12}{\frac12}{1} X.
\end{eqnarray*}
If $2|f$, then $\mathcal{M} = {^t X} \smat{u}{0}{0}{1} X$ with $u \equiv 0 \mod 4$ or
$\mathcal{M} = {^t X} \smat{u}{1}{1}{1} X$ with $u \equiv 0 \mod 4$.
By a straightforward calculation we have
\begin{eqnarray*}
 L_{0,0}^*
 &=&
 \begin{cases}
 \left\{\lambda_3 \in (2^{-1}\Z)^{(1,2)} \, | \, \lambda_3 {^t X} \in \frac{1}{2}\Z \times \Z \right\}
 & \mbox{ if } 2|f ,\\
 \Z^{(1,2)} & \mbox{ if } 2 {\not|} f,
 \end{cases}
\end{eqnarray*}
where $X = \smat{1}{ }{x}{1}$ is a matrix such that $\mathcal{M} = {^t X} \smat{u}{0}{0}{1} X$
with $u \equiv 0 \mod 4$
or $ = {^t X} \smat{u}{1}{1}{1} X$ with $u \equiv 0 \mod 4$.
Thus, if $2|f$, then
\begin{eqnarray*}
 \tilde{K}_{0,0}^{0}
 &=&
 2^{-3k+4} \sum_{M \in \Gamma_{\infty}^{(1)}\backslash \Gamma_1}
 \sum_{\begin{smallmatrix} \lambda_3 \\ \lambda_3 {^t X} \smat{2}{ }{ }{1} \in \Z^{(1,2)} \end{smallmatrix}} \\
 && \times
 \left\{1|_{k,\mathcal{M}[X^{-1}\smat{2}{ }{ }{1}^{-1}]}([(\lambda_3 {^t X} \smat{2}{ }{ }{1},0),0],M)\right\}(\tau,z\smat{2}{ }{ }{1} {^t X} \smat{2}{ }{ }{1}) \\
 &=&
 2^{-3k+4} E_{k,\mathcal{M}[X^{-1}\smat{2}{ }{ }{1}^{-1}]}^{(1)}(\tau,z \smat{2}{ }{ }{1} {^t X} \smat{2}{ }{ }{1}).
\end{eqnarray*}
And, if $2 {\not|} f$, then
\begin{eqnarray*}
  \tilde{K}_{0,0}^{0}
 &=&
 2^{-3k+4} \sum_{M \in \Gamma(\delta_{0,0})\backslash \Gamma_1}
 \sum_{\lambda_3  \in \Z^{(1,2)} }
 \left\{1|_{k,\mathcal{M}}([(\lambda_3,0),0],M)\right\}(\tau,z\smat{2}{0}{0}{1})\\
 &=&
 2^{-3k+4} E_{k,\mathcal{M}}^{(1)}(\tau,z \smat{2}{ }{ }{1}).
\end{eqnarray*}
Hence we obtain the formula for $\tilde{K}_{0,0}^{0}$.

Therefore we conclude the lemma.
\end{proof}

\subsection{Proof of Theorem \ref{thm:deg2}}\label{ss:pf:deg2}
In this subsection we conclude the proof of Theorem \ref{thm:deg2}.
We recall $\mathcal{M} = \smat{*}{*}{ * }{1} \in \mbox{Sym}_2^+$.
We put $m = \det(2\mathcal{M})$.

We define $E_{k,m}^{(1)} := \iota_{\mathcal{M}}(E_{k,\mathcal{M}}^{(1)})$,
where the map $\iota_{\mathcal{M}}$ is defined in \S\ref{ss:iota}.
We remark that $E_{k,m}^{(1)}$ is well-defined,
namely, if $\mathcal{N} = \smat{*}{*}{ * }{1} \in \mbox{Sym}_2^+$ and $\det(2\mathcal{N}) = m$,
then $\iota_{\mathcal{N}}(E_{k,\mathcal{N}}^{(1)}) = E_{k,m}^{(1)}$.
This fact follows from Proposition \ref{prop:iota_U} and from the fact that
there exits a matrix $X = \smat{1}{ }{x}{1}$ such that $\mathcal{N} = \mathcal{M}[X]$.

The form $e_{k,m}^{(1)} \in J_{k-\frac12,m}^{(1)*}$ was defined as the Fourier-Jacobi coefficient of
generalized Cohen-Eisenstein series of degree $2$ (cf. \S\ref{s:introduction}), and 
due to Lemma \ref{lemma:iota}, we have
$e_{k,m}^{(1)} = \iota_{M}(e_{k,\mathcal{M}}^{(1)})$.
For the definition of $e_{k,\mathcal{M}}^{(1)}$,
see \S\ref{s:fourier_matrix}.

Now, by Proposition \ref{prop:fourier_jacobi} and Proposition \ref{prop:iota_hecke}, we have
\begin{eqnarray*}
 e_{k,m}^{(1)}|V_{p}^{(1)}
 &=&
 e_{k,m}^{(1)}|\tilde{V}_{1,0}(p^2) \\
 &=&
 p^{3k-4}\iota_{\mathcal{M}[\smat{p}{ }{ }{1}]}(e_{k,\mathcal{M}}^{(1)}|V_{1,0}(p^2))\\
 &=&
 p^{3k-4} \sum_{d|f}g_k\!\left(\frac{f}{d^2}\right)\ \iota_{\mathcal{M}[\smat{p}{ }{ }{1}]}\left(E_{k,\mathcal{M}[{^t W_d}^{-1}]}^{(1)}(\tau,z W_d)|V_{1,0}(p^2)\right),
\end{eqnarray*}
where the symbols $f$ and $W_d$ are the same ones in Proposition \ref{prop:fourier_jacobi}.
By the definition of index-shift maps we have
\begin{eqnarray*}
 E_{k,\mathcal{M}[{^t W_d}^{-1}]}^{(1)}(\tau,z W_d)|V_{1,0}(p^2)
 &=&
 \left(E_{k,\mathcal{M}[{^t W_d}^{-1}]}^{(1)}|V_{1,0}(p^2)\right) (\tau,z W_d).
\end{eqnarray*}
The form $E_{k,\mathcal{M}[{^t W_d}^{-1}]}^{(1)}|V_{1,0}(p^2)$ is a linear combination of Jacobi-Eisenstein series
of matrix index (cf. Lemma \ref{lemma:E_V1}.)

Due to Proposition \ref{prop:iota_U} we have
$
 \iota_{\mathcal{M}[\smat{p}{0}{0}{1}]}(E_{k,\mathcal{M}}^{(n)}(*,*\smat{p}{0}{0}{1}))(\tau,z)
 =
 E_{k,m}^{(n)}(\tau,pz)
$
and
$
 \iota_{\mathcal{M}[\smat{p}{0}{0}{1}]}
 \left(E_{k,\mathcal{M}\left[X^{-1}\smat{p}{0}{0}{1}^{-1}\right]}^{(n)}(*,* \smat{p}{0}{0}{1} {^t X} \smat{p}{0}{0}{1})\right)(\tau,z)
 =
 E_{k,\frac{m}{p^2}}^{(n)}(\tau,p^2 z).
$
By using these identities and due to Lemma \ref{lemma:E_V1}, we obtain
\begin{eqnarray*}
 &&
 \iota_{\mathcal{M}[{^t W_d}^{-1} \smat{p}{ }{ }{1}]}\left(E_{k,\mathcal{M}[{^t W_d}^{-1}]}^{(1)}|V_{1,0}(p^2) \right) \\
 &=&
 p^{-k+1} E_{k,\frac{mp^2}{d^2}}^{(1)}(\tau,dz)
 - p^{-2k+2} \left(\frac{D_0 f^2/d^2}{p}\right) E_{k,\frac{mp^2}{d^2}}^{(1)}(\tau,dz)\\
 &&
 + p^{-2k+2} \left(\frac{D_0 f^2/d^2}{p}\right) E_{k,\frac{m}{d^2}}^{(1)}(\tau,pdz)\\
 &&
 + \delta\!\left(p|\frac{f}{d}\right) p^{-3k+4} E_{k,\frac{m}{p^2d^2}}^{(1)}(\tau,p^2 d z) 
 + \delta\!\left(p {\not|} \frac{f}{d}\right) p^{-3k+4} E_{k,\frac{m}{d^2}}^{(1)}(\tau,pdz),
\end{eqnarray*}
where $\delta(\mathcal{S}) = 1$ or $0$ accordingly as the statement $\mathcal{S}$ is true or false,
and where $D_0$ is the discriminant of $\mathbb{Q}(\sqrt{-m})$.
Hence
\begin{eqnarray*}
 && e_{k,m}^{(1)}|V_{p}^{(1)} \\
 &=& p^{2k-3} \sum_{d|f}g_{k}\!\left(\frac{m}{d^2}\right)E_{k,\frac{mp^2}{d^2}}^{(1)}(\tau,dz)
     - p^{k-2} \sum_{d|f}g_{k}\!\left(\frac{m}{d^2}\right)\left(\frac{D_0 f^2/d^2}{p}\right) E_{k,\frac{mp^2}{d^2}}^{(1)}(\tau,dz) \\
 && + p^{k-2} \sum_{d|f}g_{k}\!\left(\frac{m}{d^2}\right) \left(\frac{D_0 f^2/d^2}{p}\right) E_{k,\frac{m}{d^2}}^{(1)}(\tau,pdz)\\
 && + \sum_{d|f} \delta\!\left(p|\frac{f}{d}\right) g_{k}\!\left(\frac{m}{d^2}\right) E_{k,\frac{m}{p^2d^2}}^{(1)}(\tau,p^2 dz)
    + \sum_{d|f}\delta\!\left(p {\not|} \frac{f}{d}\right) g_{k}\!\left(\frac{m}{d^2}\right) E_{k,\frac{m}{d^2}}^{(1)}(\tau,pdz).
\end{eqnarray*}
Because of Lemma \ref{lemma:gk} we obtain
\begin{eqnarray*}
 &&
 e_{k,m}^{(1)}|V_{p}^{(1)}  \\
 &=&
 \sum_{d|f} g_k\!\left(\frac{mp^2}{d^2}\right) E_{k,\frac{mp^2}{d^2}}^{(1)}(\tau,dz)
 + p^{k-2} \left(\frac{D_0}{p}\right)
      \sum_{\begin{smallmatrix}d|f \\ \frac{f}{d}\not\equiv 0\!\! \mod p\end{smallmatrix}} g_k\!\left(\frac{m}{d^2}\right) E_{k,\frac{m}{d^2}}^{(1)}(\tau,pdz) \\
 &&
 + \delta(p|f) \!\! 
    \sum_{\begin{smallmatrix} d|f \\ \frac{f}{d} \equiv 0 \!\! \mod p \end{smallmatrix}} \!\! g_k\!\left(\frac{m}{d^2}\right) E_{k,\frac{m}{d^2p^2}}^{(1)}(\tau,p^2dz) 
 +  \!\!\!\! \sum_{\begin{smallmatrix} d|f \\ \frac{f}{d}\not\equiv 0\!\! \mod p \end{smallmatrix}} \!\! g_k\!\left(\frac{m}{d^2}\right) E_{k,\frac{m}{d^2}}^{(1)}(\tau,pdz).
\end{eqnarray*}
By using Lemma \ref{lemma:gk} again, we have
\begin{eqnarray*}
 &&
 e_{k,m}^{(1)}|V_{p}^{(1)} \\
 &=&
 \sum_{d|f} g_k\!\left(\frac{mp^2}{d^2}\right) E_{k,\frac{mp^2}{d^2}}^{(1)}(\tau,dz)
 + p^{k-2} \left(\frac{D_0}{p}\right)
      \sum_{\begin{smallmatrix}d|f \\ \frac{f}{d}\not\equiv 0\!\! \mod p \end{smallmatrix}} g_k\!\left(\frac{m}{d^2}\right) E_{k,\frac{m}{d^2}}^{(1)}(\tau,pdz) \\
 &&
 + \delta(p|f)\, p^{2k-3}
    \sum_{d|\frac{f}{p}} g_k\!\left(\frac{m}{d^2p^2}\right) E_{k,\frac{m}{d^2p^2}}^{(1)}(\tau,p^2dz) \\
 &&
 - \delta(p|f)\, p^{k-2}
    \sum_{\begin{smallmatrix} d > 0 \\ pd|f \end{smallmatrix}}
     \left(\frac{m/(dp)^2}{p}\right) g_k\!\left(\frac{m}{d^2p^2}\right) E_{k,\frac{m}{d^2p^2}}^{(1)}(\tau,p^2dz) \\
 &&
 + \sum_{\begin{smallmatrix} d | f \\ \frac{f}{d} \not \equiv 0 \!\! \mod p \end{smallmatrix}}
    g_k\!\left(\frac{mp^2}{(pd)^2}\right) E_{k,\frac{mp^2}{(pd)^2}}^{(1)}(\tau,pdz) \\
 &=&
 \sum_{d|f} g_k\!\left(\frac{mp^2}{d^2}\right) E_{k,\frac{mp^2}{d^2}}^{(1)}(\tau,dz)
 + p^{k-2} \left(\frac{D_0}{p}\right)
      \sum_{\begin{smallmatrix}d|f \\ \frac{f}{d}\not\equiv 0\!\! \mod p \end{smallmatrix}}
 \!\!  g_k\!\left(\frac{m}{d^2}\right) E_{k,\frac{m}{d^2}}^{(1)}(\tau,pdz) \\
 &&
 + \delta(p|f)\, p^{2k-3}
    \sum_{d|\frac{f}{p}} g_k\!\left(\frac{m}{d^2p^2}\right) E_{k,\frac{m}{d^2p^2}}^{(1)}(\tau,p^2dz) \\
 &&
 - \delta(p|f)\, p^{k-2} \left(\frac{D_0}{p}\right)
    \sum_{\begin{smallmatrix} d'|f \\ \frac{f}{d'}\not\equiv 0\!\!\mod p \end{smallmatrix}}
      g_k\!\left(\frac{m}{d'^2}\right) E_{k,\frac{m}{d'^2}}^{(1)}(\tau,pd'z) \\
 &&
 + \sum_{\begin{smallmatrix} d' | pf \\ \frac{pf}{d'} \not\equiv 0\!\! \mod p  \end{smallmatrix}}
    g_k\!\left(\frac{mp^2}{d'^2}\right) E_{k,\frac{mp^2}{d'^2}}^{(1)}(\tau,d'z)  \\
 &=&
 \sum_{d|fp} g_k\!\left(\frac{mp^2}{d^2}\right) E_{k,\frac{mp^2}{d^2}}^{(1)}(\tau,dz)
 + p^{k-2} \left(\frac{D_0 f^2}{p}\right)
      \sum_{d|f} g_k\!\left(\frac{m}{d^2}\right) E_{k,\frac{m}{d^2}}^{(1)}(\tau,pdz) \\
 &&
 + \delta(p|f)\, p^{2k-3}
    \sum_{d|\frac{f}{p}} g_k\!\left(\frac{m}{d^2p^2}\right) E_{k,\frac{m}{d^2p^2}}^{(1)}(\tau,p^2dz).
\end{eqnarray*}
Because
$ e_{k,m}^{(1)}(\tau,z)
 =
 \sum_{d|f} g_k\!\left(\frac{m}{d^2}\right) E_{k,\frac{m}{d^2}}^{(1)}(\tau,dz),
$
we conclude Theorem \ref{thm:deg2}.
\qed

\subsection{Proof of Corollary~\ref{cor:deg2}}\label{ss:cor:deg2}
In this subsection we shall show Corollary~\ref{cor:deg2}.
Let $M_{k-\frac12}^+\left(\Gamma_0(4)\right)$ be the Kohnen plus-space of weight $k-\frac12$.
Let $g(\tau) = \sum_m c(m) e^{2\pi \sqrt{-1} m \tau}$ be the Fourier expansion of an element $g$ in
$M_{k-\frac12}^+\left(\Gamma_0(4)\right)$.
For any prime $p$ the Hecke operator $T_1(p^2)$ is defined by
\begin{eqnarray*}
 &&
  \left(g|T_1(p^2)\right)(\tau) \\
 &:=&
 \sum_m \left(c(p^2m) + p^{k-2}\left(\frac{(-1)^{k-1}m}{p}\right) c(m) + p^{2k-3} c\! \left(\frac{m}{p^2}\right) \right) e^{2\pi \sqrt{-1} m\tau}.
\end{eqnarray*}
Hence, by the definition of $V_p^{(1)}$ and of $S_p^{(1)}$ and by substituting $z=0$ to $e_{k,m}^{(1)}(\tau,z)$, we obtain
\begin{eqnarray*}
 \left(e_{k,m}^{(1)}( * ,0)|T_1(p^2)\right)(\tau)
 &=&
 \left(e_{k,m}^{(1)}|V_p^{(1)}\right)(\tau,0) \\
 &=&
 \left(e_{k,m}^{(1)}|S_p^{(1)} \right)(\tau,0) \\
 &=&
 e_{k,p^2m}^{(1)}(\tau,0) + p^{k-2} \left(\frac{-m}{p}\right) e_{k,m}^{(1)}(\tau,0) + p^{2k-3} e_{k,\frac{m}{p^2}}^{(1)}(\tau,0).
\end{eqnarray*}
Therefore
\begin{eqnarray*}
 &&
 \mathcal{H}_{k-\frac12}^{(2)}\!\!\left.\left(\begin{pmatrix} \tau & 0 \\ 0 & \omega \end{pmatrix}\right)\right|_{\tau}T_1(p^2)
 \\ &=&
 \sum_m \left(\left(e_{k,m}^{(1)}( * ,0)|T_1(p^2)\right)(\tau)\right) e^{2\pi \sqrt{-1} m \omega}
 \\ &=&
 \sum_m \left(
   e_{k,p^2m}^{(1)}(\tau,0) + p^{k-2} \left(\frac{-m}{p}\right) e_{k,m}^{(1)}(\tau,0) + p^{2k-3} e_{k,\frac{m}{p^2}}^{(1)}(\tau,0) \right)
   e^{2\pi \sqrt{-1} m \omega}
 \\ &=&
 \mathcal{H}_{k-\frac12}^{(2)}\!\! \left.\left(\begin{pmatrix} \tau & 0 \\ 0 & \omega \end{pmatrix}\right)\right|_{\omega}T_1(p^2).
\end{eqnarray*}
\qed

\section{Proof of Theorem~\ref{thm:deg3}}
In this section we shall give the proof of Theorem \ref{thm:deg3}.
We treat the case degree $n = 2$.
For the sake of simplicity we abbreviate $E_{k,\mathcal{M}}^{(2)}$ (resp. $e_{k,\mathcal{M}}^{(2)}$) as $E_{k,\mathcal{M}}$ 
(resp. $e_{k,\mathcal{M}}$).

\subsection{Calculation of $\tilde{K}_{i,j}^{\beta}$}\label{ss:tilde_kij_n2}
In this subsection we shall express $\tilde{K}_{i,j}^{\beta}(\tau,z)$ (cf. Proposition~\ref{prop:tilde_kij}) as a
linear combination of three Jacobi-Eisenstein series $E_{k,\mathcal{M}[\smat{p}{ }{ }{1}]}(\tau,z)$,
$E_{k,\mathcal{M}}(\tau,z\smat{p}{ }{ }{1})$ and
$E_{k,\mathcal{M}[X^{-1}\smat{p}{ }{ }{1}^{-1}]}(\tau,z\smat{p}{ }{ }{1} {^t X} \smat{p}{ }{ }{1})$,
where $X = \smat{1}{ }{x}{1}$ is a certain matrix depending on the choice of $\mathcal{M}$ and $p$.

From Proposition \ref{prop:tilde_kij} we recall
\begin{eqnarray*}
 E_{k,\mathcal{M}}|V_{\alpha,2-\alpha}(p^2)
 &=&
 \sum_{0\leq i \leq j \leq 2}
  \tilde{K}_{i,j}^{\alpha-i-2+j}.
\end{eqnarray*}
Hence,
\begin{eqnarray*}
 E_{k,\mathcal{M}}|V_{1,1}(p^2)
 &=&
 \tilde{K}_{1,2}^{0}
 +
 \tilde{K}_{0,1}^{0}
 +
 \tilde{K}_{0,2}^{1}
\end{eqnarray*}
and
\begin{eqnarray*}
 E_{k,\mathcal{M}}|V_{2,0}(p^2)
 &=&
 \tilde{K}_{2,2}^{0}
 +
 \tilde{K}_{1,1}^{0}
 +
 \tilde{K}_{0,0}^{0}
 +
 \tilde{K}_{1,2}^{1}
 +
 \tilde{K}_{0,1}^{1}
 +
 \tilde{K}_{0,2}^{2}.
\end{eqnarray*}
\begin{lemma}\label{lemma:E_lambda}
For $\mathcal{M} \in \mbox{Sym}_2^+$ let
$D_0$ be the discriminant of $\Qq(\sqrt{-\det(2\mathcal{M})})$.
Let $f$ be the positive integer such that $-\det(2\mathcal{M}) = D_0 f^2$.
Then, for any prime $p$ we obtain
\begin{eqnarray*}
 \tilde{K}_{1,2}^{0}(\tau,z) = 
 p^{-2k+3}E_{k,\mathcal{M}}(\tau,z\left(\begin{smallmatrix}p&\\ & 1 \end{smallmatrix}\right)) 
 + p^{-2k+4}E_{k,\mathcal{M}[\left(\begin{smallmatrix}p&\\ & 1 \end{smallmatrix}\right)]}(\tau,z) ,
\end{eqnarray*}

\begin{eqnarray*}
 && \tilde{K}_{0,1}^{0}(\tau,z) =  \\
 && \begin{cases}
  p^{-4k+7}E_{k,\mathcal{M}[X^{-1}\left(\begin{smallmatrix}p&\\ & 1 \end{smallmatrix}\right)^{-1}]}
 (\tau,z\left(\begin{smallmatrix}p&\\ & 1 \end{smallmatrix}\right) ^t X \smat{p}{ }{ }{1}) 
  + p^{-4k+8}E_{k,\mathcal{M}}(\tau,z\left(\begin{smallmatrix}p&\\ & 1 \end{smallmatrix}\right)) 
 & \mbox{ if } p| f ,
 \\
 p^{-4k+7}(p+1)E_{k,\mathcal{M}}(\tau,z\left(\begin{smallmatrix}p&\\ & 1 \end{smallmatrix}\right))
 & \mbox{ if } p {\not|} f ,
\end{cases}
\end{eqnarray*}
where $X = \smat{1}{0}{x}{1}$ is a matrix such that $\mathcal{M}[X^{-1}\smat{p}{ }{ }{1}^{-1}] \in \mbox{Sym}_2^+$,
\begin{eqnarray*}
 &&
 \tilde{K}_{0,2}^{1}(\tau,z) \\
 &=&
 -p^{-3k+5} \left(\frac{D_0 f^2}{p}\right) E_{k,\mathcal{M}[\left(\begin{smallmatrix}p& \\ &1 \end{smallmatrix}\right)]}(\tau,z) 
 + p^{-3k+5} \left(\frac{D_0 f^2}{p}\right) 
 E_{k,\mathcal{M}}(\tau, z\left(\begin{smallmatrix} p & \\ & 1 \end{smallmatrix}\right)).
\end{eqnarray*}
\end{lemma}
\begin{proof}
Let $G_{\mathcal{M}}^{j-i,2-\alpha}(\lambda)$ and $L_{i,j}^*$
be the symbols defined in \S\ref{ss:tilde_kij}
and $\Gamma(\delta_{i,j})$ be the symbol defined in \S\ref{ss:kij}.

For $i = 1$, $j = 2$ we have
\begin{eqnarray*}
 L_{1,2}^*
 &=&
 \left\{\left(\begin{smallmatrix}\lambda_1 \\ \lambda_2 \end{smallmatrix}\right) \in \Z^{(2,2)}
 \, | \, 
 \lambda_1 \in p \Z \times \Z, \, \lambda_2 \in \Z \times \Z \right\}.
\end{eqnarray*}
Now we remark that the set $\left\{\smat{1}{ }{x}{1}, \smat{ }{-1}{1}{p} \, | \, x \mod p \right\}$ is
a complete set of representatives of $\delta_{1,2} GL_2(\Z) \delta_{1,2}^{-1}\cap GL_2(\Z) \backslash GL_2(\Z)$.
Hence, for any function $F$ on $\Z^{(2,2)}$ we obtain
\begin{eqnarray*}
 \sum_{\smat{A}{B}{0_2}{^t A ^{-1}} \in \Gamma(\delta_{1,2})\backslash \Gamma_{\infty}^{(2)}} \sum_{\lambda\in L_{1,2}^*}
 F({^t A} \lambda)
 &=&
 \sum_{A \in \delta_{1,2} GL_2(\Z) \delta_{1,2}^{-1}\cap GL_2(\Z) \backslash GL_2(\Z)}
 \sum_{\lambda\in L_{1,2}^*}
 F({^t A} \lambda)
\\
 &=&
 \sum_{\lambda \in \Z^{(2,2)}} F(\lambda) + p \sum_{\lambda \in \Z^{(2,2)}} F(\lambda \smat{p}{0}{0}{1}),
\end{eqnarray*}
if the above summations are absolutely convergent (cf. Lemma~\ref{lemma:F1} in the appendix).
Due to Proposition~\ref{prop:tilde_kij} we therefore obtain
\begin{eqnarray*}
 \tilde{K}_{1,2}^{0}(\tau,z)
 &=&
 p^{-2k+2 }
 \sum_{M \in \Gamma(\delta_{1,2})\backslash \Gamma_2}
 \sum_{ \lambda = \left(\begin{smallmatrix}\lambda_1\\ \lambda_2 \end{smallmatrix}\right)\in L_{1,2}^*}
 \left\{1|_{k,\mathcal{M}}([(\lambda,0),0_n],M)\right\}(\tau,z\smat{p}{0}{0}{1})
\\
 &&
 \times
  \sum_{u_2 \in (\Z/p\Z)^{(1,1)}}G_{\mathcal{M}}^{1,1}(\lambda_2 + (0,u_2)).
\end{eqnarray*}
Because
 $G_{\mathcal{M}}^{1,1}(\lambda_2 + (0,u_2)) = 1$ for any $\lambda_2 \in \Z^{(1,2)}$, $u_2 \in \Z$,
and because
\begin{eqnarray*}
 \left\{1|_{k,\mathcal{M}}([(\lambda\smat{p}{0}{0}{1},0),0_2],M)\right\}(\tau,z\smat{p}{0}{0}{1})
 &=&
 \left\{1|_{k,\mathcal{M}[\smat{p}{0}{0}{1}]}([(\lambda,0),0_2],M)\right\}(\tau,z),
\end{eqnarray*}
we obtain
\begin{eqnarray*}
 &&
 \tilde{K}_{1,2}^{0}(\tau,z) \\
 &=&
 p^{-2k+3 }
 \sum_{M \in \Gamma_{\infty}^{(2)} \backslash \Gamma_2}
 \sum_{M_0 \in \Gamma(\delta_{1,2})\backslash \Gamma_{\infty}^{(2)}}
 \sum_{ \lambda \in L_{1,2}^*}
 \left\{1|_{k,\mathcal{M}}([(\lambda,0),0_n],M_0 M)\right\}(\tau,z\smat{p}{0}{0}{1})
\\
&=&
 p^{-2k+3} \left\{ 
  \sum_{M \in \Gamma_{\infty}^{(2)} \backslash \Gamma_2} \sum_{\lambda \in \Z^{(2,2)}}
               \left\{1|_{k,\mathcal{M}}([(\lambda,0),0_2],M)\right\}(\tau,z\smat{p}{0}{0}{1})
 \right. \\
&&
 \left. + 
  p \sum_{M \in \Gamma_{\infty}^{(2)} \backslash \Gamma_2} \sum_{\lambda \in \Z^{(2,2)}}
               \left\{1|_{k,\mathcal{M}}([(\lambda\smat{p}{0}{0}{1},0),0_2],M)\right\}(\tau,z\smat{p}{0}{0}{1}) 
 \right\} \\
&=& 
 p^{-2k+3}E_{k,\mathcal{M}}(\tau,z\left(\begin{smallmatrix}p&\\ & 1 \end{smallmatrix}\right)) 
 + p^{-2k+4}E_{k,\mathcal{M}[\left(\begin{smallmatrix}p&\\ & 1 \end{smallmatrix}\right)]}(\tau,z) .
\end{eqnarray*}
Thus we have the formula for $\tilde{K}_{1,2}^{0}(\tau,z)$.

Now we shall calculate $\tilde{K}_{0,1}^{0}(\tau,z)$.
If $p|f$, then we can take matrices $X = \smat{1}{0}{x}{1} \in \Z^{(2,2)}$ and $\mathcal{M}' \in \mbox{Sym}_2^+$
which satisfy
$\mathcal{M} = \mathcal{M}'[\smat{p}{0}{0}{1} X]$.
Then
\begin{eqnarray*}
 L_{0,1}^* 
  &=& 
  \left. \left\{ \left(\begin{matrix}\lambda_2 \\ \lambda_3 \end{matrix} \right)
  \, \right| \,
  \lambda_2 \in \Z^{(1,2)}, \, 
  \lambda_3 {^t X} \in \Z^{(1,2)}\smat{p}{0}{0}{1}^{-1}  \right\}
\end{eqnarray*}
and
\begin{eqnarray*}
 \tilde{K}_{0,1}^{0}(\tau,z) 
&=& 
 p^{-4k+6} \sum_{M \in \Gamma(\delta_{0,1})\backslash \Gamma_2} 
      \sum_{\lambda = \left(\begin{smallmatrix}\lambda_2 \\ \lambda_3 \end{smallmatrix}\right)
                \in L_{0,1}^*}
    \left\{1|_{k,\mathcal{M}}([(\lambda,0),0_2],M)\right\}(\tau,z\smat{p}{0}{0}{1}) \\
&&
    \times \sum_{u_2 \in \Z/p\Z}G_{\mathcal{M}}^{1,1}(\lambda_2 + (0,u_2)) \\
&=& 
 p^{-4k+7} \sum_{M \in \Gamma(\delta_{0,1})\backslash \Gamma_2} 
      \sum_{\lambda \in L_{0,1}^* }
    \left\{1|_{k,\mathcal{M}}([(\lambda,0),0_2],M)\right\}(\tau,z\smat{p}{0}{0}{1}) .
\end{eqnarray*}
Because
\begin{eqnarray*}
 &&
 \left\{1|_{k,\mathcal{M}}([(\lambda,0),0_2],M)\right\}(\tau,z\smat{p}{0}{0}{1}) \\
 &=&
 \left\{1|_{k,\mathcal{M}'}([(\lambda {^t X} \smat{p}{0}{0}{1},0),0_2],M)\right\}(\tau,z\smat{p}{0}{0}{1}{^t X}\smat{p}{0}{0}{1}) ,
\end{eqnarray*}
we obtain
\begin{eqnarray*}
 &&
 \tilde{K}_{0,1}^{0}(\tau,z) \\
 &=&
 p^{-4k+7} \sum_{M \in \Gamma(\delta_{0,1})\backslash \Gamma_2}
      \sum_{\begin{smallmatrix}
               \lambda \\ 
               \lambda {^t X} \in \left(\begin{smallmatrix} p \Z & \Z \\ \Z & \Z \end{smallmatrix} \right)
            \end{smallmatrix}}
      \left\{1|_{k,\mathcal{M}'}([(\lambda {^t X},0),0_2],M)\right\}(\tau,z\smat{p}{0}{0}{1}{^t X}\smat{p}{0}{0}{1})
\\
 &=&
 p^{-4k+7}\left\{
  \sum_{M \in \Gamma_{\infty}^{(2)} \backslash \Gamma_2}
  \sum_{\begin{smallmatrix}
         \lambda \\
         \lambda {^t X} \in \Z^{(2,2)}
        \end{smallmatrix}}
      \left\{1|_{k,\mathcal{M}'}([(\lambda {^t X},0),0_2],M)\right\}(\tau,z\smat{p}{0}{0}{1}{^t X}\smat{p}{0}{0}{1})
 \right. 
\\
 &&
 + \left. 
  p \sum_{M \in \Gamma_{\infty}^{(2)} \backslash \Gamma_2}
    \sum_{\begin{smallmatrix}
           \lambda \\
           \lambda {^t X} \in \Z^{(2,2)}
          \end{smallmatrix}}
      \left\{1|_{k,\mathcal{M}'}([(\lambda {^t X} \smat{p}{0}{0}{1},0),0_2],M)\right\}(\tau,z\smat{p}{0}{0}{1}{^t X}\smat{p}{0}{0}{1})
 \right\}
\\
 &=&
 p^{-4k+7}E_{k,\mathcal{M}[X^{-1}\left(\begin{smallmatrix}p&\\ & 1 \end{smallmatrix}\right)^{-1}]}
 (\tau,z\left(\begin{smallmatrix}p&\\ & 1 \end{smallmatrix}\right) ^t X \smat{p}{ }{ }{1}) 
  + p^{-4k+8}E_{k,\mathcal{M}}(\tau,z\left(\begin{smallmatrix}p&\\ & 1 \end{smallmatrix}\right)).
\end{eqnarray*}
Thus we obtain the formula of $\tilde{K}_{0,1}^{0}(\tau,z)$ for the case $p | f$.
If $p {\not|} f$,
then by a straightforward calculation, we have
$L_{0,1}^* = \Z^{(2,2)}$.
Hence
\begin{eqnarray*}
 &&
 \tilde{K}_{0,1}^{0}(\tau,z)  \\
&=& 
 p^{-4k+6} \sum_{M \in \Gamma(\delta_{0,1})\backslash \Gamma_2} 
      \sum_{\lambda \in \Z^{(2,2)}}
    \left\{1|_{k,\mathcal{M}}([(\lambda,0),0_2],M)\right\}(\tau,z\smat{p}{0}{0}{1})
 \\
 &&
    \times \sum_{u_2 \in \Z/p\Z}G_{\mathcal{M}}^{1,1}(\lambda_2 + (0,u_2))
 \\
 &=&
 p^{-4k+7} [\Gamma_{\infty}^{(2)} : \Gamma(\delta_{0,1})] \sum_{M \in \Gamma_{\infty}^{(2)} \backslash \Gamma_2} 
      \sum_{\lambda \in \Z^{(2,2)}}
      \left\{1|_{k,\mathcal{M}}([(\lambda,0),0_2],M)\right\}(\tau,z\smat{p}{0}{0}{1})
\\
 &=&
 p^{-4k+7}(p+1)E_{k,\mathcal{M}}(\tau,z\left(\begin{smallmatrix}p&\\ & 1 \end{smallmatrix}\right)).
\end{eqnarray*}
Here we used
$[\Gamma_{\infty}^{(2)} : \Gamma(\delta_{0,1})] 
 = [\mbox{GL}_2(\Z) : \mbox{GL}_2(\Z)\cap \delta_{0,1}\mbox{GL}_2(\Z)\, \delta_{0,1}^{-1} ] 
 = p + 1$.
Thus we obtain the formula of $\tilde{K}_{0,1}^{0}(\tau,z)$ also for the case $p {\not|} f$.

Finally we shall calculate $\tilde{K}_{0,2}^{1}(\tau,z)$.
We remark that for a matrix $X = \smat{1}{0}{x}{1}$ and for a $\lambda \in \Z^{(2,2)}$,
the condition $\lambda {^t X} \in \Z^{(2,2)}\smat{p}{0}{0}{1}$
is equivalent to the condition $\lambda \in \Z^{(2,2)}\smat{p}{0}{0}{1}$.
Thus, due to Lemma~\ref{lemma:gauss21} in the appendix, we have
\begin{eqnarray*}
 \sum_{u_2 \in (\Z/p\Z)^{(2,1)}}G_{\mathcal{M}}^{2,1}(\lambda + (0,u_2))
 &=&
 \begin{cases}
  0 & \mbox{if } \lambda \in \Z^{(2,2)}\smat{p}{0}{0}{1}, \\
  p^3 \left(\frac{D_0 f^2}{p}\right)
    & \mbox{if } \lambda \not \in \Z^{(2,2)}\smat{p}{0}{0}{1}.
 \end{cases}
\end{eqnarray*}

Therefore
\begin{eqnarray*}
 &&
 \tilde{K}_{0,2}^{1}(\tau,z) \\
 &=&
 p^{-3k+2}\sum_{M \in \Gamma_{\infty}^{(2)}\backslash \Gamma_2} \sum_{\lambda \in \Z^{(2,2)}}
 \left\{1|_{k,\mathcal{M}}([(\lambda,0),0_2],M)\right\}(\tau,z\smat{p}{0}{0}{1})
\\
 &&
 \times \sum_{u_2 \in (\Z/p\Z)^{(2,1)}}G_{\mathcal{M}}^{2,1}(\lambda + (0,u_2))
\\
 &=&
 p^{-3k+2}\sum_{M \in \Gamma_{\infty}^{(2)}\backslash \Gamma_2} 
 \sum_{\begin{smallmatrix}\lambda \in \Z^{(2,2)}\\ 
                          \lambda \not \in \Z^{(2,2)}\smat{p}{0}{0}{1} \end{smallmatrix}}
 p^3\left(\frac{D_0 f^2}{p}\right) \left\{1|_{k,\mathcal{M}}([(\lambda,0),0_2],M)\right\}(\tau,z\smat{p}{0}{0}{1})
\\
 &=&
 p^{-3k+5} \left(\frac{D_0 f^2}{p}\right)
 \left\{ 
   - \sum_{M \in \Gamma_{\infty}^{(2)} \backslash \Gamma_2} \sum_{\lambda \in \Z^{(2,2)}}
   \left\{1|_{k,\mathcal{M}}([(\lambda\smat{p}{0}{0}{1},0),0_2],M)\right\}(\tau,z\smat{p}{0}{0}{1})
 \right.
\\
 && \left. 
  + \sum_{M \in \Gamma_{\infty}^{(2)} \backslash \Gamma_2} \sum_{\lambda \in \Z^{(2,2)}}
   \left\{1|_{k,\mathcal{M}}([(\lambda,0),0_2],M)\right\}(\tau,z\smat{p}{0}{0}{1})
 \right\} \\
 &=& 
 -p^{-3k+5}\left(\frac{D_0 f^2}{p}\right) 
   E_{k,\mathcal{M}\left[\left(\begin{smallmatrix} p&0\\ 0&1 \end{smallmatrix} \right)\right]}(\tau,z)
 +p^{-3k+5}\left(\frac{D_0 f^2}{p}\right) 
   E_{k,\mathcal{M}}(\tau,z\left(\begin{smallmatrix} p&0\\ 0&1 \end{smallmatrix} \right)) .
\end{eqnarray*}
Hence we conclude the lemma.
\end{proof}

\begin{lemma}\label{lemma:E_omega}
For $\mathcal{M} \in \mbox{Sym}_2^+$ let
$D_0$ and $f$ be as in Lemma~\ref{lemma:E_lambda}.
For any prime $p$
we obtain
\begin{eqnarray*}
 \tilde{K}_{2,2}^{0}(\tau,z) 
 &=& 
  p^{-k+2} E_{k,\mathcal{M}[\left(\begin{smallmatrix}p&0\\0&1\end{smallmatrix} \right)]}(\tau,z),
\\
 \tilde{K}_{1,1}^{0}(\tau,z)
 &=& 
 \begin{cases}
 p^{-3k+6} E_{k,\mathcal{M}}(\tau,z\left(\begin{smallmatrix}p&0\\0&1\end{smallmatrix} \right)) + \!
 p^{-3k+7} E_{k,\mathcal{M}[\left(\begin{smallmatrix}p&0\\0&1\end{smallmatrix} \right)]}(\tau,z)
 & \mbox{if } p {\not|}f ,\\
 p^{-3k+5} E_{k,\mathcal{M}'}(\tau,z\left(\begin{smallmatrix}p^2&0\\0&1\end{smallmatrix} \right))
 + \! p^{-3k+5}(p-1) E_{k,\mathcal{M}}(\tau,z\left(\begin{smallmatrix}p&0\\0&1\end{smallmatrix} \right)) & \\
 +  p^{-3k+7} E_{k,\mathcal{M}[\left(\begin{smallmatrix}p&0\\0&1\end{smallmatrix} \right)]}(\tau,z)
 & \mbox{if } p | f,
 \end{cases}
\\
 \tilde{K}_{0,0}^{0}(\tau,z) 
 &=& 
 \begin{cases}
 p^{-5k+10} E_{k,\mathcal{M}}(\tau , z\left(\begin{smallmatrix}p&0\\0&1\end{smallmatrix} \right))
 & \mbox{if } p {\not|} f ,\\
 p^{-5k+10} E_{k,\mathcal{M}'}(\tau , z\left(\begin{smallmatrix}p^2&0\\0&1\end{smallmatrix} \right))
 & \mbox{if } p | f ,
 \end{cases}
\\
 \tilde{K}_{1,2}^{1}(\tau,z) 
 &=& 
 p^{-2k+3}\left(\frac{D_0 f^2}{p}\right) E_{k,\mathcal{M}}(\tau , z\left(\begin{smallmatrix}p&0\\0&1\end{smallmatrix} \right)) \\
 &&
 - p^{-2k+3}\left(\frac{D_0 f^2}{p}\right) 
   E_{k,\mathcal{M} \left[\left(\begin{smallmatrix}p&0\\0&1\end{smallmatrix} \right)\right]} (\tau,z),
\\
 \tilde{K}_{0,1}^{1}(\tau,z) 
 &=&
 p^{-4k+8}\left(\frac{D_0 f^2}{p}\right) E_{k,\mathcal{M}}(\tau , z\left(\begin{smallmatrix}p&0\\0&1\end{smallmatrix} \right)) \\
 &&
 - p^{-4k+8}\left(\frac{D_0 f^2}{p}\right) 
   E_{k,\mathcal{M} \left[\left(\begin{smallmatrix}p&0\\0&1\end{smallmatrix} \right)\right]} (\tau,z),
\\
 \tilde{K}_{0,2}^{2}(\tau,z) 
 &=& 
 p^{-3k+4}(p-1) E_{k,\mathcal{M}}(\tau , z\left(\begin{smallmatrix}p&0\\0&1\end{smallmatrix} \right)) .
\end{eqnarray*}
\end{lemma}
\begin{proof}
For the calculation of $\tilde{K}_{i,j}^{j-i}$ we need to determine the set $L_{i,j}^*$, the value
of the summation $\displaystyle{\sum_{u_2\in \left(\Z/p\Z\right)^{(j-i,1)}}G_{\mathcal{M}}^{(j-i,0)}(\lambda + (0,u_2))}$
 and a complete set of the representatives of $\Gamma(\delta_{i,j})\backslash \Gamma_{\infty}^{(2)}$.
The table of these are given in \S\ref{ss:table_n2_2} in the appendix.

For the calculation of $\tilde{K}_{1,1}^{0}$ we use the identity
\begin{eqnarray*}
 &&
 \sum_{\smat{A}{B}{0_2}{^t A ^{-1}} \in \Gamma(\delta_{1,1})\backslash \Gamma_{\infty}^{(2)}} \sum_{\lambda \in \smat{p^2\Z}{\Z}{\Z}{\Z}}
   F({^tA} \lambda)
\\
 &=&
 \sum_{\lambda \in \Z^{(2,2)}}F(\lambda)
 + (p-1) \sum_{\lambda \in \Z^{(2,2)}}F(\lambda \smat{p}{0}{0}{1})
 + p^2  \sum_{\lambda \in \Z^{(2,2)}}F(\lambda \smat{p^2}{0}{0}{1}),
\end{eqnarray*}
where $F$ is a function on $\Z^{(2,2)}$ such that the above summations are absolutely convergent.
The proof of this identity will be given in Lemma \ref{lemma:F2} in the appendix.

The rest of the calculation is an analogue to Lemma~\ref{lemma:E_lambda},
hence we omit the detail.
\end{proof}

\subsection{Proof of Theorem \ref{thm:deg3}}
In this subsection we conclude the proof of Theorem \ref{thm:deg3}.

Let $\mathcal{M} = \smat{*}{*}{ * }{1} \in \mbox{Sym}_2^+$ be a matrix which satisfies $\det(2\mathcal{M}) = m$.
Let $D_0$, $f$ be as before, namely $D_0$ is the discriminant of $\Qq(\sqrt{-m})$ and $f$ is the positive integer 
which satisfies $-m = D_0 f^2$.

We define $E_{k,m} = E_{k,m}^{(2)} := \iota_{\mathcal{M}}(E_{k,\mathcal{M}}^{(2)})$,
where the map $\iota_{\mathcal{M}}$ is defined in \S\ref{ss:iota}.
We remark that $E_{k,m}$ is well-defined, namely it does not depend on the choice of $\mathcal{M}$ (cf. \S\ref{ss:pf:deg2}.)

The form $e_{k,m} := e_{k,m}^{(2)} \in J_{k-\frac12,m}^{(2)*}$ was defined as the Fourier-Jacobi coefficient of the
generalized Cohen-Eisenstein series $\mathcal{H}_{k-\frac12}^{(3)}$ of degree $3$ (cf. \S\ref{s:introduction}).
And, due to Lemma~\ref{lemma:iota}, we have
$e_{k,m} = \iota_{M}(e_{k,\mathcal{M}})$.

By Proposition~\ref{prop:iota_hecke} and by Proposition~\ref{prop:fourier_jacobi}, we have
\begin{eqnarray*}
 &&
 e_{k,m}|\widetilde{V}_{\alpha,2-\alpha}(p^2) \\
 &=&
 p^{5k-11+\frac12 \alpha}\,
 \iota_{\mathcal{M}[\smat{p}{0}{0}{1}]}\left(
 e_{k,\mathcal{M}}|V_{\alpha,2-\alpha}(p^2)
 \right)
 \\
 &=&
 p^{5k-11+\frac12 \alpha}\,
 \iota_{\mathcal{M}[\smat{p}{0}{0}{1}]}\left(
  \sum_{d|f} g_k\!\left(\frac{m}{d^2}\right)\left(E_{k,\mathcal{M}[{^t W_d}^{-1}]}|V_{\alpha,2-\alpha}(p^2)\right)(\tau,z W_d)\right).
\end{eqnarray*}
The forms $E_{k,\mathcal{M}}|V_{\alpha,2-\alpha}(p^2)$ $(\alpha = 1,2)$ have been calculated in \S\ref{ss:tilde_kij_n2},
and are written as linear combinations of three forms
 $E_{k,\mathcal{M}[X^{-1} \smat{p}{ }{ }{1}^{-1}]}(\tau,z\smat{p}{ }{ }{1} {^t X} \smat{p}{ }{ }{1})$,
 $E_{k,\mathcal{M}}(\tau,z\smat{p}{ }{ }{1})$
 and
 $E_{k,\mathcal{M}[\smat{p}{ }{ }{1}]}(\tau,z)$.
We recall the definitions of $V_{1,p}^{(2)}$ and $V_{2,p}^{(2)}$:
\begin{eqnarray*}
 e_{k,m}|V_{1,p}^{(2)}
 &=&
 p^{-k+\frac72}\, e_{k,m}|\tilde{V}_{1,1}(p^2),
\\
 e_{k,m}|V_{2,p}^{(2)}
 &=&
 e_{k,m}|\tilde{V}_{2,0}(p^2).
\end{eqnarray*}
Because we defined $E_{k,m} = \iota_{\mathcal{M}}(E_{k,\mathcal{M}})$,
we have
\begin{eqnarray*}
 E_{k,\frac{m p^2}{d^2}}(\tau,dz)
 &=&
 \iota_{\mathcal{M}[\smat{p}{ }{ }{1}]}(E_{k,\mathcal{M}[{^t W_d}^{-1} \smat{p}{ }{ }{1}]}( * , * W_d))(\tau,z),
 \\
 E_{k,\frac{m}{d^2}}(\tau,pdz)
 &=&
 \iota_{\mathcal{M}[\smat{p}{ }{ }{1}]}(E_{k,\mathcal{M}[{^t W_d}^{-1} ]}( * , * W_d \smat{p}{ }{ }{1}))(\tau,z)
\end{eqnarray*}
and
\begin{eqnarray*}
 \\
 E_{k,\frac{m}{p^2 d^2}}(\tau,p^2 dz)
 &=&
 \iota_{\mathcal{M}[\smat{p}{ }{ }{1}]}(E_{k,\mathcal{M}[{^t W_d}^{-1} X^{-1}\smat{p}{ }{ }{1}^{-1}]}( * , * W_d \smat{p}{ }{ }{1} {^t X} \smat{p}{ }{ }{1}))(\tau,z).
\end{eqnarray*}

Now we shall calculate $e_{k,m}|V_{1,p}^{(2)}$.
Due to Lemma~\ref{lemma:E_lambda} and due to the above identities, we have
\begin{eqnarray*}
 e_{k,m}|V_{1,p}^{(2)}
 &=&
 A_1 + A_2 + A_3,
\end{eqnarray*}
where
\begin{eqnarray*}
 A_1
 &:=&
 \sum_{d|f}
 \{
  p^{2k-4} E_{k,\frac{m}{d^2}}(\tau,pdz)
  + p^{2k-3} E_{k,\frac{m p^2}{d^2}}(\tau,d z)
 \}\, g_k\!\left(\frac{m}{d^2} \right), \\
 A_2
 &:=&
  \sum_{\begin{smallmatrix} d|f \\  \frac{f}{d} \equiv 0\!\! \mod p \end{smallmatrix}}
 \{
  E_{k,\frac{m}{p^2d^2}}(\tau,p^2dz)
  + p E_{k,\frac{m}{d^2}}(\tau,p d z)
 \} \, g_k\!\left(\frac{m}{d^2} \right) \\
 &&
 + \sum_{\begin{smallmatrix} d|f \\ \frac{f}{d} \not\equiv 0\!\! \mod p \end{smallmatrix}} 
  (p+1) E_{k,\frac{m}{d^2}}(\tau,p d z) \, g_k\!\left(\frac{m}{d^2} \right)
\end{eqnarray*}
and
\begin{eqnarray*} 
 A_3
 &:=&
 \sum_{ d|f } \left\{-p^{k-2} \left(\frac{D_0f^2/d^2}{p}\right) E_{k,\frac{m p^2}{d^2}}(\tau,d z) \right. \\
 &&
 \left. + p^{k-2} \left(\frac{D_0 f^2/d^2}{p}\right) E_{k,\frac{m}{d^2}}(\tau,p d z)
 \right\} 
 \, g_k\!\left(\frac{m}{d^2}\right).
\end{eqnarray*}

By using Lemma~\ref{lemma:gk} we have
\begin{eqnarray*}
 A_1
&=&
 p^{2k-4} \sum_{d|f} E_{k,\frac{m}{d^2}}(\tau,pdz)\, g_k\!\left(\frac{m}{d^2} \right)
 + \sum_{d|f} 
 E_{k,\frac{m p^2}{d^2}}(\tau,d z)\, g_k\!\left(\frac{mp^2}{d^2} \right) \\
 && 
 + \left(\frac{D_0}{p} \right)p^{k-2} 
 \sum_{\begin{smallmatrix} d|f \\ \frac{f}{d} \not\equiv 0\!\! \mod p \end{smallmatrix}} 
 E_{k,\frac{m p^2}{d^2}}(\tau,d z)\, g_k\!\left(\frac{m}{d^2} \right),
\end{eqnarray*}
\begin{eqnarray*}
 A_2
&=&
 \delta(p|f) p^{2k-3} \sum_{d|\frac{f}{p}}
 E_{k,\frac{m}{p^2d^2}}(\tau,p^2dz) \, g_k\!\left(\frac{m}{d^2p^2} \right) \\
 &&
 - \delta(p|f)
  \sum_{ d|\frac{f}{p} }
   p^{k-2} \left(\frac{D_0 f^2/d^2p^2}{p}\right)
   E_{k,\frac{m}{p^2d^2}}(\tau,p^2dz) \, g_k\!\left(\frac{m}{d^2p^2} \right) \\
 &&
 + \delta(p|f) p \sum_{\begin{smallmatrix} d|f \\ \frac{f}{d}\equiv 0 \!\! \mod p \end{smallmatrix}}
 E_{k,\frac{m}{d^2}}(\tau,p d z) \, g_k\!\left(\frac{m}{d^2} \right) \\
 &&
 + p \sum_{\begin{smallmatrix} d|f \\ \frac{f}{d}\not\equiv 0\!\! \mod p \end{smallmatrix}}
 E_{k,\frac{m}{d^2}}(\tau,p d z) \, g_k\!\left(\frac{m}{d^2} \right)
 + \sum_{\begin{smallmatrix} d|f \\ \frac{f}{d}\not\equiv 0\!\! \mod p \end{smallmatrix}}
  E_{k,\frac{m}{d^2}}(\tau,p d z) \, g_k\!\left(\frac{m}{d^2} \right) 
\\
&=&
 \delta(p|f) p^{2k-3} \sum_{d|\frac{f}{p}}
 E_{k,\frac{m}{p^2d^2}}(\tau,p^2dz) \, g_k\!\left(\frac{m}{d^2p^2} \right) \\
 &&
 - \delta(p|f) \left(\frac{D_0}{p}\right) p^{k-2}
  \sum_{\begin{smallmatrix} d|\frac{f}{p} \\ \frac{f}{dp}\not\equiv 0\!\! \mod p  \end{smallmatrix}}
   E_{k,\frac{m}{p^2d^2}}(\tau,p^2dz) \, g_k\!\left(\frac{m}{d^2p^2} \right) \\
 &&
 + p \sum_{ d|f }
 E_{k,\frac{m}{d^2}}(\tau,p d z) \, g_k\!\left(\frac{m}{d^2} \right)
 + \sum_{\begin{smallmatrix} d|pf \\ \frac{pf}{d}\not\equiv 0\!\! \mod p \end{smallmatrix}} 
   E_{k,\frac{p^2m}{d^2}}(\tau, d z) \, g_k\!\left(\frac{p^2m}{d^2} \right),
\end{eqnarray*}
where we used the identities
\begin{eqnarray*}
 \left(\frac{D_0 f^2/d^2p^2}{p}\right)
 &=&
 \left(\frac{D_0}{p}\right) \delta\!\left(p {\not|}  \frac{f}{pd}\right)
\end{eqnarray*}
and
\begin{eqnarray*}
 \delta(p|f)\, p \!\!\!\! \sum_{\begin{smallmatrix} d|f \\ \frac{f}{d}\equiv 0 \!\! \mod p \end{smallmatrix}}
 +\
 p \!\!\!\! \sum_{\begin{smallmatrix} d|f \\ \frac{f}{d}\not\equiv 0\!\! \mod p \end{smallmatrix}}
 &=&
 p  \sum_{ d|f } \ ,
\end{eqnarray*}
and we have
\begin{eqnarray*}
 A_3 &=&
 - \left(\frac{D_0}{p}\right) p^{k-2} 
 \sum_{\begin{smallmatrix} d|f \\ \frac{f}{d}\not\equiv 0\!\! \mod p \end{smallmatrix} }
  E_{k,\frac{m p^2}{d^2}}(\tau,d z)  \, g_k\!\left(\frac{m}{d^2}\right)
 \\
 &&
 + \left(\frac{D_0}{p}\right) p^{k-2} 
 \sum_{\begin{smallmatrix} d|f \\ \frac{f}{d}\not\equiv 0\!\! \mod p \end{smallmatrix} }
  E_{k,\frac{m}{d^2}}(\tau,p d z) \, g_k\!\left(\frac{m}{d^2}\right) .
\end{eqnarray*}
Thus, due to Proposition~\ref{prop:fourier_jacobi}, we obtain
\begin{eqnarray*}
 e_{k,m}|V_{1,p}^{(2)}
&=&
  p^{2k-4} \sum_{d|f} E_{k,\frac{m}{d^2}}(\tau,pdz) \, g_k\!\left(\frac{m}{d^2} \right)
 + p \sum_{ d|f }
 E_{k,\frac{m}{d^2}}(\tau,p d z) \, g_k\!\left(\frac{m}{d^2} \right) \\
 && 
 + \sum_{d|f} 
 E_{k,\frac{m p^2}{d^2}}(\tau,d z) \, g_k\!\left(\frac{mp^2}{d^2} \right)
 + \sum_{\begin{smallmatrix} d|pf \\ \frac{pf}{d}\not\equiv 0\!\! \mod p  \end{smallmatrix}} 
   E_{k,\frac{p^2m}{d^2}}(\tau, d z) \, g_k\!\left(\frac{p^2m}{d^2} \right)
\\
 &&
 + \delta(p|f)\, p^{2k-3} \sum_{d|\frac{f}{p}}
 E_{k,\frac{m}{p^2d^2}}(\tau,p^2dz) \, g_k\!\left(\frac{m}{d^2p^2} \right) \\
 &&
 - \delta(p|f)\, \left(\frac{D_0}{p}\right) p^{k-2}
  \sum_{\begin{smallmatrix} d|\frac{f}{p} \\ \frac{f}{dp}\not\equiv 0\!\! \mod p \end{smallmatrix}}
   E_{k,\frac{m}{p^2d^2}}(\tau,p^2dz) \, g_k\!\left(\frac{m}{d^2p^2} \right) \\
\\
 &&
 + \left(\frac{D_0}{p}\right) p^{k-2} 
 \sum_{\begin{smallmatrix} d|f \\ \frac{f}{d}\not\equiv 0\!\! \mod p \end{smallmatrix} }
  E_{k,\frac{m}{d^2}}(\tau,p d z) \, g_k\!\left(\frac{m}{d^2}\right) \\
 &=&
  p^{2k-4} \sum_{d|f} E_{k,\frac{m}{d^2}}(\tau,pdz) \, g_k\!\left(\frac{m}{d^2} \right)
 + p \sum_{ d|f }
 E_{k,\frac{m}{d^2}}(\tau,p d z) \, g_k\!\left(\frac{m}{d^2} \right) \\
 && 
 + \sum_{d|fp}
 E_{k,\frac{m p^2}{d^2}}(\tau,d z) \, g_k\!\left(\frac{mp^2}{d^2} \right) \\
 &&
 + \delta(p|f)\, p^{2k-3} \sum_{d|\frac{f}{p}}
 E_{k,\frac{m}{p^2d^2}}(\tau,p^2dz) \, g_k\!\left(\frac{m}{d^2p^2} \right) \\
 &&
 + \left(\frac{D_0 f^2}{p}\right) p^{k-2} 
 \sum_{d|f}
  E_{k,\frac{m}{d^2}}(\tau,p d z) \, g_k\!\left(\frac{m}{d^2}\right)
\\
 &=&
 p \left( p^{2k-5} + 1 \right) e_{k,m}(\tau,pz)
\\ 
 && + e_{k,mp^2}(\tau,z) 
 + p^{2k-3} e_{k,\frac{m}{p^2}}(\tau,p^2 z)
 + \left(\frac{-m}{p}\right) p^{k-2} e_{k,m}(\tau,pz).
\end{eqnarray*}
Hence we obtain the identity for $e_{k,m}|V_{1,p}^{(2)}$.

Because the calculation of $e_{k,m}|V_{2,p}^{(2)}$ is an analogue to the case of $e_{k,m}|V_{1,p}^{(2)}$,
we omit the detail.
\qed

\subsection{Proof of Corollary \ref{cor:deg3}}\label{ss:cor:deg3}
In this subsection we shall show Corollary~\ref{cor:deg3}.

By the definition of $V_{1,p}^{(2)}$, $T_{2,1}(p^2)$ and of $S_p^{(2)}$
and by substituting $z=0$ to $e_{k,m}^{(2)}(\tau,z)$, we obtain
\begin{eqnarray*}
 \left(e_{k,m}^{(2)}( * ,0)|T_{2,1}(p^2)\right)(\tau) 
 &=&
 \left(e_{k,m}^{(2)}|V_{1,p}^{(2)}\right)(\tau,0) \\
 &=&
 \left(e_{k,m}^{(2)}|\left(p\left(p^{2k-5}+1\right) + S_p^{(2)} \right)\right)(\tau,0) \\
 &=&
 p(p^{2k-5}+1) e_{k,m}^{(2)}(\tau,0) + 
 e_{k,p^2m}^{(2)}(\tau,0) \\
 &&
 + p^{k-2} \left(\frac{-m}{p}\right) e_{k,m}^{(2)}(\tau,0) + p^{2k-3} e_{k,\frac{m}{p^2}}^{(2)}(\tau,0).
\end{eqnarray*}
Therefore
\begin{eqnarray*}
 &&
 \mathcal{H}_{k-\frac12}^{(3)}\!\!\left.\left(\begin{pmatrix} \tau & 0 \\ 0 & \omega \end{pmatrix}\right)\right|_{\tau}T_{2,1}(p^2)
 \\ &=&
 \sum_m \left(\left(e_{k,m}^{(2)}( * ,0)|T_{2,1}(p^2)\right)(\tau)\right) e^{2\pi \sqrt{-1} m \omega}
 \\ &=&
 \sum_m \left(
   p(p^{2k-5}+1) e_{k,m}^{(2)}(\tau,0) \right.\\
   &&
   \quad \quad
   + \left.  e_{k,p^2m}^{(2)}(\tau,0) + p^{k-2} \left(\frac{-m}{p}\right) e_{k,m}^{(2)}(\tau,0)
           + p^{2k-3} e_{k,\frac{m}{p^2}}^{(2)}(\tau,0) \right)
   e^{2\pi \sqrt{-1} m \omega}
 \\ &=&
 \mathcal{H}_{k-\frac12}^{(3)}\!\! \left.\left(\begin{pmatrix} \tau & 0 \\ 0 & \omega \end{pmatrix}\right)\right|_{\omega}
 \left(p(p^{2k-5}+1) + T_1(p^2) \right).
\end{eqnarray*}
Similarly, we have
\begin{eqnarray*}
 &&
 \mathcal{H}_{k-\frac12}^{(3)}\!\!\left.\left(\begin{pmatrix} \tau & 0 \\ 0 & \omega \end{pmatrix}\right)\right|_{\tau}T_{2,2}(p^2) \\
 &=&
 \mathcal{H}_{k-\frac12}^{(3)}\!\!\left.\left(\begin{pmatrix} \tau & 0 \\ 0 & \omega \end{pmatrix}\right)\right|_{\omega}
 ((p^{2k-4}-p^{2k-6}) + p(p^{2k-5} + 1) T_1(p^2)).
\end{eqnarray*}
\qed

\section{Appendix}

\subsection{Values of some generalized Gauss sums}
In this subsection we shall give the values of generalized Gauss sums
$G_{\mathcal{M}}^{2,1}(\lambda)$ and $G_{\mathcal{M}}^{2,0}(\lambda)$, which are defined in \S\ref{ss:tilde_kij}.
For odd primes these values follow from the result in \cite{Sa}.
We need these values for the calculation of $\tilde{K}_{i,j}^{\beta}$ in \S\ref{ss:tilde_kij_n2}.

In this subsection we fix a matrix $\mathcal{M} = \smat{*  }{ * }{ * }{1}\in \mbox{Sym}_2^*$.
\begin{lemma}\label{lemma:gauss21}
Let $p$ be an odd prime and 
$X = \smat{1}{0}{x}{1} \in \Z^{(2,2)}$ be a matrix such that $\mathcal{M} \equiv {^t X} \smat{u}{0}{0}{1} X \!\! \mod p$.
Then, for $\lambda \in \Z^{(2,2)}$ we have
\begin{eqnarray*}
G_{\mathcal{M}}^{2,1}(\lambda) 
&=& 
   \sum_{\begin{smallmatrix}x \in Sym_2(\Z/p\Z) \\ rank_p(x) = 1
          \end{smallmatrix}}e\left(\frac{1}{p}\mathcal{M} {^t \lambda } x  \lambda\right) \\
&=&
\begin{cases}
  p^2 -1
 & \mbox{if } \mbox{rank}_p(\lambda ^tX) = 0 ,\\
  p^2 -1
 & \begin{array}{l} 
   \mbox{if } \mbox{rank}_p(\lambda ^tX) = 1 \mbox{ and } \lambda ^tX \equiv (\lambda',t \lambda')\!\!\mod p \\
   \mbox{ with } t \mbox{ such that } u+t^2 \in p\Z ,
   \end{array}
\\ 
  -1
 & \begin{array}{l}
   \mbox{if } \mbox{rank}_p(\lambda ^tX) = 1 \mbox{ and } \lambda ^tX \equiv (\lambda',t \lambda')\!\!\mod p\\
   \mbox{ with } t \mbox{ such that } u+t^2 \not \in p\Z ,
   \end{array}
\\
  -1
 & \mbox{if } \lambda ^tX \equiv (0, \lambda')\!\!\mod p \mbox{ with } \lambda' \not \in (p\Z)^{(2,1)} ,\\
  \left(\frac{-u}{p}\right)p -1
 & \mbox{if } \mbox{rank}_p(\lambda ^tX) = 2 .
\end{cases}
\end{eqnarray*}
For $p=2$ there exists a matrix $X = \smat{1}{0}{x}{1}$ such that 
$\mathcal{M}  = {^t X} \smat{u}{\frac{r}{2}}{\frac{r}{2}}{1} X $
with $r = 0$ or $1$.
Then, for $\lambda \in \Z^{(2,2)}$ we have
\begin{eqnarray*}
G_{\mathcal{M}}^{2,1}(\lambda) 
&=& 
   \sum_{\begin{smallmatrix}x \in Sym_2(\Z/2\Z) \\ rank_2(x) = 1
          \end{smallmatrix}}e\left(\frac{1}{2}\mathcal{M} {^t \lambda } x  \lambda\right) \\
&=&
 \begin{cases}
  3 
  & \mbox{if } \mbox{rank}_2(\lambda ^t X) = 0 ,\\
  -1 + 2(1 + (-1)^{u+t})
  & 
   \begin{array}{l}
     \mbox{if } \mbox{rank}_2(\lambda ^t X) = 1 \\
     \mbox{ and }\lambda ^tX \equiv (\lambda',t \lambda') \!\! \mod 2 \\
     \mbox{ and }{^t X}^{-1}\mathcal{M} X^{-1} =
      \left( \begin{smallmatrix}
        u&0\\0&1 
      \end{smallmatrix} \right) ,
   \end{array} \\
  -1 + 2(1+(-1)^u) 
  & 
   \begin{array}{l}
     \mbox{if } \mbox{rank}_2(\lambda ^t X) = 1 \\
     \mbox{ and }\lambda ^tX \equiv (\lambda',t \lambda') \!\! \mod 2 \\
     \mbox{ and }{^t X}^{-1}\mathcal{M} X^{-1} =
      \left( \begin{smallmatrix}
        u&\frac12\\\frac12&1 
      \end{smallmatrix} \right) ,
   \end{array} \\
  -1
  &
  \begin{array}{l}
     \mbox{if } \mbox{rank}_2(\lambda ^t X) = 1 \\
     \mbox{ and }\lambda ^tX \equiv (0, \lambda') \!\! \mod 2 ,
   \end{array} \\
  -1
  &
  \begin{array}{l}
     \mbox{if } \mbox{rank}_2(\lambda ^t X) = 2 \\
     \mbox{ and }{^t X}^{-1}\mathcal{M} X^{-1} =
      \left( \begin{smallmatrix}
        u&0\\0&1 
      \end{smallmatrix} \right) ,
   \end{array} \\
  1 - 2(1-(-1)^u)
  &
  \begin{array}{l}
     \mbox{if } \mbox{rank}_2(\lambda ^t X) = 2 \\
     \mbox{ and }{^t X}^{-1}\mathcal{M} X^{-1} =
      \left( \begin{smallmatrix}
        u&\frac12\\ \frac12&1 
      \end{smallmatrix} \right).
   \end{array}
 \end{cases}  
\end{eqnarray*}
\end{lemma}
\begin{proof}
For odd prime $p$ this lemma follows from \cite[Proposition 1.12]{Sa}.

For the case $p=2$ we can directly calculate $G_{\mathcal{M}}^{2,1}(\lambda)$.
The details is omitted here.
\end{proof}

\begin{lemma}
Let $p$ be an odd prime and 
$X = \smat{1}{0}{x}{1} \in \Z^{(2,2)}$ be a matrix such that $\mathcal{M} \equiv {^t X} \smat{u}{0}{0}{1} X \!\! \mod p$.
Then, for $\lambda \in \Z^{(2,2)}$ we have
\begin{eqnarray*}
 &&
G_{\mathcal{M}}^{2,0}(\lambda) \\
&=& 
   \sum_{\begin{smallmatrix}x \in Sym_2(\Z/p\Z) \\ rank_p(x) = 2
         \end{smallmatrix}}e\left(\frac{1}{p}\mathcal{M} {^t \lambda } x  \lambda\right) \\
&=& 
\begin{cases} 
 p^2(p-1) 
& \mbox{if } \mbox{rank}_p(\lambda ^tX) = 0 ,\\
 p^2(p-1)
&
 \begin{array}{l}
 \mbox{if }  \mbox{rank}_2(\lambda^t X) = 1 \\
 \mbox{ and }\lambda ^tX \equiv (\lambda',t \lambda') \!\! \mod p \mbox{ with } t \mbox{ such that } u+t^2 \in p\Z ,
 \end{array}
  \\
 0
&
 \begin{array}{l}
 \mbox{if }  \mbox{rank}_2(\lambda^t X) = 1 \\
 \mbox{ and }\lambda ^tX \equiv (\lambda',t \lambda') \!\! \mod p \mbox{ with } t \mbox{ such that } u+t^2 \not \in p\Z ,
 \end{array}
 \\
 0
& \mbox{if } \lambda ^tX \equiv (0, \lambda') \!\! \mod p\mbox{ with }
              \lambda' \not \in (p\Z)^{(2,1)}  , \\
 -\left(\frac{-u}{p} \right) p 
& \mbox{if } \mbox{rank}_p(\lambda ^tX) = 2.
\end{cases}
\end{eqnarray*}
For $p=2$ there exists a matrix $X = \smat{1}{0}{x}{1}$ such that 
$\mathcal{M}  = {^t X} \smat{u}{\frac{r}{2}}{\frac{r}{2}}{1} X $
with $r = 0$ or $1$.
Then, for $\lambda \in \Z^{(2,2)}$ we have
\begin{eqnarray*}
G_{\mathcal{M}}^{2,0}(\lambda) 
&=&
   \sum_{\begin{smallmatrix}x \in Sym_2(\Z/2\Z) \\ rank_2(x) = 1
          \end{smallmatrix}}e\left(\frac{1}{2}\mathcal{M} {^t \lambda } x  \lambda\right) \\
&=&
 \begin{cases}
  4 
  & \mbox{if } \mbox{rank}_2(\lambda ^t X) = 0 ,\\
  2(1 + (-1)^{u+t})
  & 
   \begin{array}{l}
     \mbox{if } \mbox{rank}_2(\lambda ^t X) = 1 \\
     \mbox{ and }\lambda ^tX \equiv (\lambda',t \lambda') \!\! \mod 2 \\
     \mbox{ and }{^t X}^{-1}\mathcal{M} X^{-1} =
      \left( \begin{smallmatrix}
        u&0\\0&1 
      \end{smallmatrix} \right) ,
   \end{array} \\
  2(1+(-1)^u) 
  & 
   \begin{array}{l}
     \mbox{if } \mbox{rank}_2(\lambda ^t X) = 1 \\
     \mbox{ and }\lambda ^tX \equiv (\lambda',t \lambda') \!\! \mod 2 \\
     \mbox{ and }{^t X}^{-1}\mathcal{M} X^{-1} =
      \left( \begin{smallmatrix}
        u&\frac12\\\frac12&1 
      \end{smallmatrix} \right) ,
   \end{array} \\
  0
  &
  \begin{array}{l}
     \mbox{if } \mbox{rank}_2(\lambda ^t X) = 1 \\
     \mbox{ and }\lambda ^tX \equiv (0, \lambda') \!\! \mod 2 ,
   \end{array} \\
  0
  &
  \begin{array}{l}
     \mbox{if } \mbox{rank}_2(\lambda ^t X) = 2 \\
     \mbox{ and }{^t X}^{-1}\mathcal{M} X^{-1} =
      \left( \begin{smallmatrix}
        u&0\\0&1 
      \end{smallmatrix} \right) ,
   \end{array} \\
  -2(-1)^u
  &
  \begin{array}{l}
     \mbox{if } \mbox{rank}_2(\lambda ^t X) = 2 \\
     \mbox{ and }{^t X}^{-1}\mathcal{M} X^{-1} =
      \left( \begin{smallmatrix}
        u&\frac12\\ \frac12&1
      \end{smallmatrix} \right) .
   \end{array}
 \end{cases}  
\end{eqnarray*}
\end{lemma}
\begin{proof}
For odd prime $p$, this lemma follows from \cite[Theorem 1.3]{Sa}.
For the case $p=2$ we can directly calculate $G_{\mathcal{M}}^{2,0}(\lambda)$.
We omit the detail.
\end{proof}

\subsection{Index-shift maps $V_{1,p}^{(2)}$, $V_{2,p}^{(2)}$ and Fourier coefficients}\label{ss:hecke_V_12}
Let
\begin{eqnarray*}
 \phi(\tau,z)
 &=&
 \sum_{\begin{smallmatrix}
        T \in Sym_2^*,\, S \in \Z^{(2,1)} \\
        4 T m - S {^t S} \geq 0
       \end{smallmatrix}} C(T,S)\, e(T\tau + S{^t z})
\end{eqnarray*}
be the Fourier expansion of $\phi \in J_{k-\frac12,m}^{(2)*}$.
In this subsection we shall express Fourier coefficients of
$\phi | V_{1,p}^{(2)}$ and $\phi | V_{2,p}^{(2)}$ as a linear combination of $C(T,S)$.

For any prime $p$ we put
\begin{eqnarray*}
 R(p)
 &:=&
 \left\{
  \smat{1}{x}{0}{1},
 \,
  \smat{0}{1}{-1}{0}
 \, | \,
  x \mbox{ mod } p
 \right\},
\\
 R(p^2)
 &:=&
 \left\{
  \smat{1}{x}{0}{1},
 \,
  \smat{py}{1}{-1}{0}
 \, | \,
  x \mbox{ mod } p^2, \,
  y \mbox{ mod } p
 \right\}.
\end{eqnarray*}

Then, the action of the index-shift maps $V_{1,p}^{(2)}$ and $V_{2,p}^{(2)}$
can be written as
\begin{eqnarray*}
 (\phi | V_{1,p}^{(2)})(\tau,z)
 &=&
 \sum_{\begin{smallmatrix}
        T \in Sym_2^*,\, S \in \Z^{(2,1)} \\
        4 T m p^2 - S {^t S} \geq 0
       \end{smallmatrix}} \sum_{i,j}\alpha_{1,i,j}(T,S)\, e(T \tau + S z), \\
 (\phi | V_{2,p}^{(2)})(\tau,z)
 &=&
 \sum_{\begin{smallmatrix}
        T \in Sym_2^*,\, S \in \Z^{(2,1)} \\
        4 T m p^2 - S {^t S} \geq 0
       \end{smallmatrix}} \sum_{i,j}\alpha_{2,i,j}(T,S)\, e(T \tau + S z),
\end{eqnarray*} 
where, for odd prime $p$, we have
\begin{eqnarray*}
 \alpha_{1,1,0}(T,S) 
  & = & 
 p^{2k-4}  \sum_{U \in R(p)} C\left( 
   \left( \begin{array}{cc} p^{-1} & \\ & 1 \end{array} \right)
    U T ^{t}U 
   \left( \begin{array}{cc} p^{-1} & \\ & 1 \end{array} \right),
   \frac{1}{p}\left( \begin{array}{cc} p^{-1} & \\ & 1 \end{array} \right)
    US 
  \right) , \\
 \alpha_{1,1,1}(T,S) 
  & = &  
 \sum_{U \in R(p)} C\left(  
  \left( \begin{array}{cc} 1 & \\ & p \end{array} \right)
  U T ^{t}U 
  \left( \begin{array}{cc} 1 & \\ & p \end{array} \right),
  \frac{1}{p}\left( \begin{array}{cc} 1 & \\ & p \end{array} \right) 
  US
 \right) , \\
 \alpha_{1,2,0}(T,S) 
  &=& 
 \left\{ 
 \begin{array}{cl}
  \left( \frac{-a}{p}  \right) p^{k-2}\, C(T,\frac{1}{p}S) & 
  \mbox{if} \ p  {\not|}  a \ \mbox{and} \ p | \det 2 T, \  \\ 
   & \\
   \left( \frac{-c}{p}  \right) p^{k-2}\, C(T,\frac{1}{p}S) &
   \mbox{if} \ p | a \ \mbox{and} \ p | \det 2 T, \  \\
  & \\
  0 & \mbox{otherwise}, 
 \end{array}
\right.
\end{eqnarray*}

\begin{eqnarray*}
 \alpha_{2,0,0}(T,S)
  &=& 
 p^{4k-8} \ C\left(\frac{1}{p^2}T,\frac{1}{p^2}S \right) , \\
 \alpha_{2,0,1}(T,S) 
  &=& 
 p^{2k-5}\ \sum_{U \in R(p^2)} C\left(  
  \left( \begin{array}{cc} p^{-1} & \\ & p \end{array} \right)
  U T ^{t}U 
  \left( \begin{array}{cc} p^{-1} & \\ & p \end{array} \right),
  \frac{1}{p}\left( \begin{array}{cc} p^{-1} & \\ & p \end{array} \right)
  U S
  \right) , \\
 \alpha_{2,0,2}(T,S) & = & C(p^2 T,  S) , 
\end{eqnarray*}

\begin{eqnarray*}
 \alpha_{2,1,0}(T,S) 
  &=& 
 p^{3k-7}\ \sum_{U \in R(p)} 
 \left( \frac{-c_{U}}{p} \right) \\
 && \times C\left(  
 \left( \begin{array}{cc} p^{-1} & \\ & 1 \end{array} \right)
  U T ^{t}U 
 \left( \begin{array}{cc} p^{-1} & \\ & 1 \end{array} \right),
 \frac{1}{p}\left( \begin{array}{cc} p^{-1} & \\ & 1 \end{array} \right)
  U S
 \right) , \\
 && where \left( \begin{array}{cc} * & * \\ * & c_{U} 
 \end{array}\right) \ = \ U T ^{t}U , \\
 \alpha_{2,1,1}(T,S) 
  &=& 
 p^{k-3} 
 \ \sum_{U \in R(p)} 
  \left( \frac{-a_{U}}{p} \right)\\
 && \times C \left(  
  \left( \begin{array}{cc} 1 & \\ & p \end{array} \right)
  U T ^{t}U 
  \left( \begin{array}{cc} 1 & \\ & p \end{array} \right),
  \frac{1}{p}\left( \begin{array}{cc} 1 & \\ & p \end{array} \right)
  U S
 \right) , \\
 && where  \left( \begin{array}{cc} a_{U} & * \\ 
 * & * \end{array}\right) \ = \ U T ^{t}U , \\
 \alpha_{2,2,0}(T,S) 
  &=& 
 \left\{ \begin{array}{cl} 
 {}- p^{2k-6} \ C(T,\frac{1}{p}S) & \mbox{if} \ p  {\not|}  \det 2 T,  \\
  & \\
 (p-1) p^{2k-6} \ C(T,\frac{1}{p}S) &  \mbox{if} \ p | \det 2 T.
\end{array} \right.
\end{eqnarray*}
For $p=2$ we have the same $\alpha_{1,i,j}(T,S)$ and $\alpha_{2,i,j}(T,S)$ in the above formula
by replacing the condition $p|\det(2T)$ or $p {\not|} \det(2T)$
by $8 | \det(T)$ or $8 {\not|} \det(T)$. (cf.~\cite[section~4.2]{HI}.)

\subsection{The function $\tilde{K}_{i,j}^{j-i}$ for $n=2$}\label{ss:table_n2_2}
In this subsection we shall give some necessary data for the calculation of $\tilde{K}_{i,j}^{j-i}$,
namely we need these data for the proof of Lemma~\ref{lemma:E_omega}.

We put
\begin{eqnarray*}
L(p) &:=& \left\{ \smat{1}{ }{x}{1}, \smat{ }{-1}{1}{p} \, | \, x \mod p \right\}, \\
L(p^2) &:=& \left\{ \smat{1}{ }{x}{1}, \smat{ }{-1}{1}{py} \, | \, x \mod p^2,\, y \mod p  \right\}.
\end{eqnarray*}
Then we have the following table:
\begin{eqnarray*}
\begin{array}{|l||c|c|c|}
\hline
(i,j)
 & \displaystyle{\sum_{u_2\in \left(\Z/p\Z\right)^{(j-i,1)}}G_{\mathcal{M}}^{(j-i,0)}(\lambda + (0,u_2))}
 & L_{i,j}^*
 & \Gamma(\delta_{i,j})\backslash \Gamma_{\infty}^{(2)} \\ \hline \hline
(2,2)
 & 1
 & \smat{p\Z}{\Z}{p\Z}{\Z}
 & 1 \\ \hline
(1,1)
 & 1
 & \begin{cases} \smat{p\Z}{\Z}{\frac{1}{p}\Z}{\Z} & \mbox{ if } p|f
              \\ \smat{p\Z}{\Z}{\Z}{\Z} & \mbox{ if } p {\not|} f
   \end{cases}
 & L(p^2) \\ \hline
(0,0)
 & 1
 & \begin{cases} \smat{\frac{1}{p}\Z}{\Z}{\frac{1}{p}\Z}{\Z} & \mbox{ if } p|f
              \\ \Z^{(2,2)} & \mbox{ if } p {\not|} f
   \end{cases}
 & 1 \\ \hline
(1,2)
 & \begin{cases} 0 & \mbox{ if } \lambda \in p\Z \times \Z
              \\ p\left(\frac{D_0f^2}{p}\right) & \mbox{ if } \lambda \not \in p\Z \times \Z
   \end{cases}
 & \smat{p\Z}{\Z}{\Z}{\Z}
 & L(p) \\ \hline
(0,1)
 & \begin{cases} 0 & \mbox{ if } \lambda \in p\Z \times \Z
              \\ p\left(\frac{D_0f^2}{p}\right) & \mbox{ if } \lambda \not \in p\Z \times \Z
   \end{cases}
 & \begin{cases} \smat{\Z}{\Z}{\frac{1}{p}\Z}{\Z} & \mbox{ if } p|f
              \\ \Z^{(2,2)} & \mbox{ if } p {\not|} f
   \end{cases}
 & L(p) \\ \hline
(0,2)
 & p^2(p-1)
 & \Z^{(2,2)}
 & 1 \\ \hline
\end{array}
\end{eqnarray*}

\subsection{Certain summations}
In this subsection we shall give some formulas which are needed for the proof of
Lemma~\ref{lemma:E_lambda} and \ref{lemma:E_omega}.
Let notation be as in \S\ref{ss:kij} and \S\ref{ss:tilde_kij}.
\begin{lemma}\label{lemma:F1}
Let $F$ be a function on $\Z^{(2,2)}$. Then we obtain
\begin{eqnarray*}
 \sum_{\smat{A}{B}{0_2}{^t A ^{-1}} \in \Gamma(\delta_{1,2})\backslash \Gamma_{\infty}^{(2)}} \sum_{\lambda\in L_{1,2}^*}
 F({^t A} \lambda)
 &=&
 \sum_{\lambda \in \Z^{(2,2)}} F(\lambda) + p \sum_{\lambda \in \Z^{(2,2)}} F(\lambda \smat{p}{0}{0}{1}),
\end{eqnarray*}
if the above summations are absolutely convergent.
\end{lemma}
\begin{proof}
A bijection map
 $\Gamma(\delta_{1,2})\backslash \Gamma_{\infty}^{(2)}
   \rightarrow
  \delta_{1,2} GL_2(\Z) \delta_{1,2}^{-1} \cap GL_2(\Z) \backslash GL_2(\Z)$
is given via
 $\smat{A}{B}{0_2}{^t A ^{-1}} \mapsto A$.
Thus
\begin{eqnarray*}
 \sum_{\smat{A}{B}{0_2}{^t A ^{-1}} \in \Gamma(\delta_{1,2})\backslash \Gamma_{\infty}^{(2)}} \sum_{\lambda\in L_{1,2}^*}
 F({^t A} \lambda)
 &=&
 \sum_{A \in \delta_{1,2} GL_2(\Z) \delta_{1,2}^{-1}\cap GL_2(\Z) \backslash GL_2(\Z)}
 \sum_{\lambda\in L_{1,2}^*}
 F({^t A} \lambda).
\end{eqnarray*}

Now we have
$ L_{1,2}^*
 =
 \left\{\left(\begin{smallmatrix}\lambda_1 \\ \lambda_2 \end{smallmatrix}\right) \in \Z^{(2,2)}
 \, | \, 
 \lambda_1 \in p \Z \times \Z, \, \lambda_2 \in \Z \times \Z \right\}.
$
We define
\begin{eqnarray*}
L(p) &:=& \left\{\smat{1}{0}{x}{1}, \smat{0}{-1}{1}{p} \, | \, x \mod p \right\}.
\end{eqnarray*}
Then $L(p)$ is a complete set of representatives of $\delta_{1,2} GL_2(\Z) \delta_{1,2}^{-1}\cap GL_2(\Z) \backslash GL_2(\Z)$.
For $\left(\begin{smallmatrix} a\\b \end{smallmatrix}\right) \in \Z^{(2,1)}$, we define
$S_{\left(\begin{smallmatrix} a\\b \end{smallmatrix}\right)} := \left\{A \in L(p) \, | \, {^t A}^{-1} \left(\begin{smallmatrix} a\\b \end{smallmatrix}\right) \in L_{1,2}^*  \right\}$.
We obtain
\begin{eqnarray*}
 \# S_{\left(\begin{smallmatrix} a\\b \end{smallmatrix}\right)}
 &=&
 \begin{cases}
  p+1 & \mbox{ if } a \equiv b \equiv 0 \!\! \mod p ,\\
  1 & \mbox{ otherwise}.
 \end{cases}
\end{eqnarray*}
Hence
\begin{eqnarray*}
 &&
 \sum_{A \in \delta_{1,2} GL_2(\Z) \delta_{1,2}^{-1}\cap GL_2(\Z) \backslash GL_2(\Z)}
 \sum_{\lambda\in L_{1,2}^*}
 F({^t A} \lambda) \\
 &=&
 (p+1) \sum_{\lambda \in \smat{p\Z}{\Z}{p\Z}{\Z}} F(\lambda)
 +  \sum_{\lambda \in \Z^{(2,2)} \backslash \smat{p\Z}{\Z}{p\Z}{\Z} } F(\lambda)
\\
 &=&
 \sum_{\lambda \in \Z^{(2,2)}} F(\lambda) + p \sum_{\lambda \in \Z^{(2,2)}} F(\lambda \smat{p}{0}{0}{1}).
\end{eqnarray*}
\end{proof}

\begin{lemma}\label{lemma:F2}
Let $F$ be a function on $\Z^{(2,2)}$. Then we obtain
\begin{eqnarray*}
 &&
 \sum_{\smat{A}{B}{0_2}{^t A ^{-1}} \in \Gamma(\delta_{1,1})\backslash \Gamma_{\infty}^{(2)}} \sum_{\lambda \in \smat{p^2\Z}{\Z}{\Z}{\Z}}
   F({^tA} \lambda)
\\
 &=&
 \sum_{\lambda \in \Z^{(2,2)}}F(\lambda)
 + (p-1) \sum_{\lambda \in \Z^{(2,2)}}F(\lambda \smat{p}{0}{0}{1})
 + p^2  \sum_{\lambda \in \Z^{(2,2)}}F(\lambda \smat{p^2}{0}{0}{1}),
\end{eqnarray*}
if the above summations are absolutely convergent.
\end{lemma}
\begin{proof}
We have
\begin{eqnarray*}
 &&
 \sum_{\smat{A}{B}{0_2}{^t A ^{-1}} \in \Gamma(\delta_{1,1})\backslash \Gamma_{\infty}^{(2)}} \sum_{\lambda\in \smat{p^2\Z}{\Z}{\Z}{\Z}}
 F({^t A} \lambda) \\
 &=&
 \sum_{A \in \delta_{1,1} GL_2(\Z) \delta_{1,1}^{-1}\cap GL_2(\Z) \backslash GL_2(\Z)}
 \sum_{\lambda\in \smat{p^2\Z}{\Z}{\Z}{\Z}}
 F({^t A} \lambda).
\end{eqnarray*}

We define
$L(p^2) := \left\{\smat{1}{0}{x}{1}, \smat{0}{-1}{1}{py} \, | \, x \mod p^2, \, y \mod p \right\}$.
Then $L(p^2)$ is a complete set of representatives of $\delta_{1,1} GL_2(\Z) \delta_{1,1}^{-1}\cap GL_2(\Z) \backslash GL_2(\Z)$.
For $\left(\begin{smallmatrix} a\\b \end{smallmatrix}\right) \in \Z^{(2,1)}$,
we define
 $S_{\left(\begin{smallmatrix} a\\b \end{smallmatrix}\right)}
  :=
  \left\{A \in L(p^2) \, | \,
  {^t A}^{-1} \left(\begin{smallmatrix} a\\b \end{smallmatrix}\right) \in \smat{p^2\Z}{\Z}{\Z}{\Z} \right\}
 $.
By a straightforward calculation we obtain
\begin{eqnarray*}
 \# S_{\left(\begin{smallmatrix} a\\b \end{smallmatrix}\right)}
 &=&
 \begin{cases}
  p^2+p & \mbox{ if } p^2|a \mbox{ and } p^2|b ,\\
  1 & \mbox{ if } p {\not|} a \mbox{ or } p {\not|} b ,\\
  p & \mbox{ otherwise}.
 \end{cases}
\end{eqnarray*}
Hence we get
\begin{eqnarray*}
 \\
 &&
 \sum_{A \in \delta_{1,1} GL_2(\Z) \delta_{1,1}^{-1}\cap GL_2(\Z) \backslash GL_2(\Z)}
 \sum_{\lambda\in \smat{p^2\Z}{\Z}{\Z}{\Z}}
 F({^t A} \lambda) \\
 &=&
 \sum_{\lambda \in \Z^{(2,2)}} F(\lambda) +
 (p-1) \sum_{\lambda \in \smat{p\Z}{\Z}{p\Z}{\Z}} F(\lambda) +
 p^2 \sum_{\lambda \in \smat{p^2\Z}{\Z}{p^2\Z}{\Z}} F(\lambda).
\end{eqnarray*}
\end{proof}

\vspace{1cm}

\noindent
Department of Mathematics, Joetsu University of Education,\\
1 Yamayashikimachi, Joetsu, Niigata 943-8512, JAPAN\\
e-mail hayasida@juen.ac.jp


\vspace{3cm}

\begin{thebibliography}{99}
 \bibitem[Ar 94]{Ar2}
   T.~Arakawa: Jacobi Eisenstein Series and a Basis Problem for Jacobi forms,
     {\it Comment.\ Math.\ Univ.\ St.\ Paul.}\ {\bf 43} No.2 (1994), 181--216.
 \bibitem[Ar 98]{Ar3}
   T.~Arakawa: K\"ocher-Maass Dirichlet Series Corresponding to Jacobi forms and Cohen Eisenstein Series,
     {\it Comment.\ Math.\ Univ.\ St.\ Paul.}\ {\bf 47} No.1 (1998), 93--122.
 \bibitem[Bo 83]{Bo}
   S.~B\"ocherer: 
     \"Uber die Fourier-Jacobi-Entwicklung Siegelscher Eisensteinreihen, 
     {\it Math.\ Z.}\ {\bf 183} (1983), 21--46
 \bibitem[Co 75]{Co}
   H.~Cohen: Sums involving the values at negative integers of $L$-functions of
    quadratic characters, {\it Math.\ Ann.}\ {\bf 217} (1975), 171--185.
 \bibitem[E-Z 85]{EZ}
  M.~Eichler and D.~Zagier: Theory of Jacobi Forms, Progress in
	 Math.\ {\bf 55}, Birkh\"auser, Boston-Basel-Stuttgart, (1985).
\bibitem[H 11]{half_deg2_lift}
 S.~Hayashida: On the lifting of pairs of elliptic modular forms to Siegel
                modular forms of half-integral weight of degree two.
    {\it in preparation} 2011.
\bibitem[H-I 05]{HI}
    S.~Hayashida and T.~Ibukiyama: 
    Siegel modular forms of half integral weights and a lifting conjecture,
    {\it Journal of Kyoto Univ,} {\bf 45} No.3 (2005), 489-530.
\bibitem[Ib 92]{Ib}
  T.~Ibukiyama: On Jacobi forms and Siegel modular forms of half
	 integral weights, {\it Comment.\ Math.\ Univ.\ St.\ Paul.}\ {\bf 41} 
	 No.2 (1992), 109--124.
\bibitem[Ik 01]{Ik}
  T.~Ikeda : On the lifting of elliptic cusp forms to Siegel cusp forms of degree 2n, 
         {\it Ann.\ of Math.\ (2)} {\bf 154} no.3, (2001), 641--681.
\bibitem[Ko 80]{Ko}
  W.~Kohnen: Modular forms of half integral weight on $\Gamma_{0}(4)$, 
	 {\it Math,\ Ann.}\ {\bf 248} (1980), 249--266.
\bibitem[Ko 02]{Ko2}
  W.~Kohnen : Lifting modular forms of half-integral weight to Siegel 
  	modular forms of even genus, 
  	{\it Math, Ann.} {\bf 322} (2002), 787--809.  
\bibitem[Sa 91]{Sa}
  H.~Saito: A generalization of Gauss sums and its applications to
             Siegel modular forms and L-functions associated with the vector
             space of quadratic forms,
  {\it J. Reine Angew. Math.}\ {\bf 416} (1991), 91--142.
\bibitem[Ta 86]{Tani}
  Y.~Tanigawa: Modular descent of Siegel modular forms of half integral weight
                and an analogy of the Maass relation,
  {\it Nagoya Math. J.}\ {\bf 102} (1986), 51--77.
 \bibitem[Ya 86]{Ya}
  T.~Yamazaki: Jacobi forms and a Maass relation for Eisenstein series,  
 {\it J.\ Fac.\ Sci.\ Univ.\ Tokyo Sect.\ IA, Math.}\ {\bf 33} (1986), 295--310.
 \bibitem[Ya 89]{Ya2}
  T.~Yamazaki: Jacobi forms and a Maass relation for Eisenstein series II,
 {\it J.\ Fac.\ Sci.\ Univ.\ Tokyo Sect.\ IA, Math.}\ {\bf 36} (1989), 373--386.
 \bibitem[Zh 84]{Zhu:euler}
  V.~G.~Zhuravlev: Euler expansions of theta transforms of Siegel modular forms
  of half-integral weight and their analytic properties, {\it Math.\
    sbornik.}\ {\bf 123 (165)} (1984), 174--194.   
 \bibitem[Zi 89]{Zi} 
  C.~Ziegler: Jacobi forms of higher degree, {\it Abh.\ Math.\ Sem.\ Univ.\  
  Hamburg.}\ {\bf 59} (1989), 191--224.  
\end{thebibliography}
\end{document}